\documentclass[11pt,reqno,oneside,a4paper]{amsart}
\usepackage[british]{babel}

\usepackage[utf8]{inputenc}
\usepackage{csquotes}
\usepackage{graphicx}
\usepackage{amscd}
\usepackage{amsmath}
\usepackage{amsfonts}
\usepackage{amssymb}
\usepackage{comment}
\usepackage{color}
\usepackage[usenames,dvipsnames,table,xcdraw]{xcolor}
\usepackage{pdfsync}
\usepackage{bbm}
\usepackage{dsfont}
\usepackage{enumerate}
\usepackage{booktabs}
\usepackage{float}
\usepackage{wrapfig}
\usepackage{caption}
\usepackage{subcaption}
\usepackage{longtable}
\usepackage{listings}
\usepackage[T1]{fontenc}
\usepackage[scaled]{beramono}
\usepackage{tikz}
\usepackage{pgffor}
\usepackage[all]{xy}
\usepackage{mathtools}
\usepackage{bookmark}

\definecolor{trp}{rgb}{1,1,1}

\newcommand\setItemnumber[1]{\setcounter{enumi}{\numexpr#1-1\relax}}

\definecolor{red}{rgb}{1,0,.2}

\newtheorem{theorem}{Theorem}[section]
\theoremstyle{plain}

\newtheorem{corollary}[theorem]{Corollary}

\newtheorem{definition}[theorem]{Definition}
\newtheorem{example}[theorem]{Example}
\newtheorem{question}[theorem]{Question}

\newtheorem{lemma}[theorem]{Lemma}

\newtheorem{prop}[theorem]{Proposition}
\newtheorem{remark}[theorem]{Remark}

\numberwithin{equation}{section}

\newcommand{\R}{\mathbb{R}}
\newcommand{\N}{\mathbb{N}}

\newcommand{\btau}{\boldsymbol\tau}

\newcommand{\ii}{\mathbf{i}}
\newcommand{\jj}{\mathbf{j}}

\newcommand{\iih}{\boldsymbol{\hat\imath}}
\newcommand{\jjh}{\boldsymbol{\hat\jmath}}
\newcommand{\kkh}{\mathbf{\hat k}}
\newcommand{\ih}{\hat\imath}
\newcommand{\jh}{\hat\jmath}
\newcommand{\kh}{\hat k}
\newcommand{\iiv}{\overline{\imath}}
\newcommand{\jjv}{\overline{\jmath}}

\newcommand{\supp}{\mathrm{supp}}

\usepackage[backend=biber,backref=true,doi=false,
url=false,isbn=false,date=year,sorting=nyt,
giveninits=true,maxbibnames=99,mincrossrefs=9,
sortcites,hyperref=true]{biblatex}
\renewbibmacro{in:}{%
  \ifentrytype{article}{}{\printtext{\bibstring{in}\intitlepunct}}}
\DeclareFieldFormat[article,periodical]{volume}{\mkbibbold{#1}}%
\DeclareFieldFormat[article,periodical,unpublished]{citetitle}{#1}
\DeclareFieldFormat[article,periodical,unpublished]{title}{#1}
\DeclareFieldFormat{pages}{#1}
\DeclareFieldFormat[inproceedings]{title}{#1\isdot}
\DeclareFieldFormat[misc]{title}{#1\isdot}
\DeclareFieldFormat[misc,article]{note}{#1\addcomma}

\AtEveryBibitem{
  \clearfield{number}}

\makeatletter
\def\verbatim@font{\normalfont\rmfamily}
\makeatother
\makeatletter
\DeclareFieldFormat{eprint:arxiv}{%
  arXiv\addcolon\space
  \ifhyperref
    {\href{http://arxiv.org/\abx@arxivpath/#1}{%
       \nolinkurl{#1}%
       \iffieldundef{eprintclass}
     {}
     {\addspace\mkbibbrackets{\thefield{eprintclass}}}}}
    {\nolinkurl{#1}
     \iffieldundef{eprintclass}
       {}
       {\addspace\mkbibbrackets{\thefield{eprintclass}}}}}
\makeatother

\usepackage{anysize}

\usepackage{caption}

\definecolor{blue}{rgb}{0,0,1}

\definecolor{red}{rgb}{1,0,.7}

\begin{document}
\title[Intermediate dimensions of Bedford--McMullen carpets]{Intermediate dimensions of Bedford--McMullen carpets with applications to Lipschitz equivalence}

\author{Amlan Banaji}
\address{Amlan Banaji, \newline Mathematical Sciences, Loughborough University, Loughborough, LE11 3TU, UK} \email{A.F.Banaji@lboro.ac.uk}

\author{Istv\'an Kolossv\'ary}
\address{Istv\'an Kolossv\'ary, \newline HUN-REN Alfréd Rényi Institute of Mathematics, 1053 Budapest, Reáltanoda u. 13–15, Hungary} 
\email{kolossvary.istvan@renyi.hu}


\begin{abstract}
Intermediate dimensions were introduced to provide a spectrum of dimensions interpolating between Hausdorff and box-counting dimensions for fractals where these differ. In particular, the self-affine Bedford--McMullen carpets are a natural case for investigation, but until now only very rough bounds for their intermediate dimensions have been found. In this paper, we determine a precise formula for the intermediate dimensions $\dim_{\, \theta}\Lambda$ of any Bedford--McMullen carpet~$\Lambda$ for the whole spectrum of~$\theta \in [0,1]$, in terms of a certain large deviations rate function. The intermediate dimensions exist and are strictly increasing in~$\theta$, and the function $\theta\mapsto \dim_{\, \theta}\Lambda$ exhibits interesting features not witnessed on any previous example, such as having countably many phase transitions, between which it is analytic and strictly concave.

We make an unexpected connection to multifractal analysis by showing that two carpets with non-uniform vertical fibres have equal intermediate dimensions if and only if the Hausdorff multifractal spectra of the uniform Bernoulli measures on the two carpets are equal. Since intermediate dimensions are bi-Lipschitz invariant, this shows that the equality of these multifractal spectra is a necessary condition for two such carpets to be Lipschitz equivalent. 
\end{abstract}

\maketitle
\begin{figure}[h]
	\centering
	\includegraphics[width=0.85\textwidth]{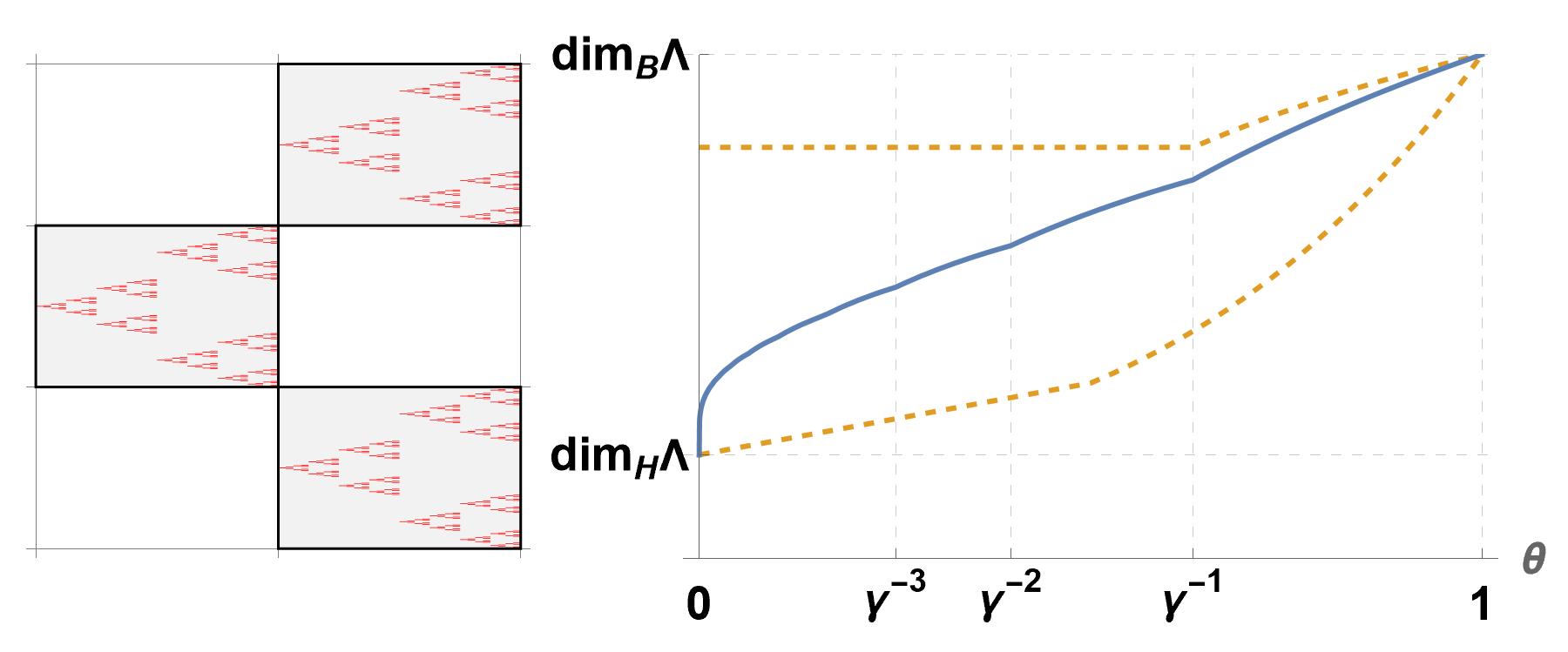}
	\caption{Intermediate dimensions give finer geometric information about the fractal set. Left: a Bedford--McMullen carpet $\Lambda$ in red and the shaded images of $[0,1]^2$ under the iterated function system generating $\Lambda$. Right: certain bounds for the intermediate dimensions $\dim_{\, \theta}\Lambda$ in dashed orange, and the true value in blue obtained in this paper.}
	\label{fig:BMPicture}
\end{figure}
\clearpage
\tableofcontents

\thispagestyle{empty}
\section{Introduction}\label{sec:01}

In dimension theory, the two most studied notions of dimension are the Hausdorff and box (also called Minkowski) dimension. For a specific set, a central question is whether its Hausdorff and box dimension are distinct or not. Intuitively, a dimension gap occurs when the set is inhomogeneous in space, and thus cannot be covered optimally by using sets of equal diameter. 
An important class of dynamically defined sets exhibiting this phenomena are self-affine carpet-like constructions in the plane and sponges in higher dimensions. 
There are many other sets in $\mathbb{R}^d$ whose Hausdorff and box dimensions differ: for example infinitely generated self-conformal sets, elliptical polynomial spirals, images of sets under certain stochastic processes~\cite{Burrell2022brownian,Daw2023}, or even the connected component of supercritical fractal percolation~\cite{Bromanetal2013FractalPerc}. 

A line of research has been initiated by Falconer, Fraser and Kempton~\cite{Falconer2020firstintermediate} in order to get a better understanding of the subtle differences between the Hausdorff and box dimensions. The idea is that by introducing a continuum of \emph{intermediate dimensions} which interpolate between these two notions of dimension, one can obtain more nuanced information about the geometry of the set. This is done by imposing increasing restrictions on the relative sizes of covering sets governed by a parameter $0\leq \theta\leq 1$. The Hausdorff and box dimension are the two extreme cases when $\theta=0$ and $1$, respectively. All sets we consider will be bounded and non-empty subsets of Euclidean space. The diameter of a set $F$ is denoted by $|F|$, its Hausdorff dimension by $\dim_{\mathrm H}F$ and its lower and upper box dimensions by $\underline{\dim}_{\mathrm B}F$ and $\overline{\dim}_{\mathrm B}F$. A finite or countable collection of sets $\{U_i\}$ is a \emph{cover} of $F$ if $F\subseteq \bigcup_i U_i.$

\begin{definition}\label{def:01}
	For $0\leq \theta\leq 1$, the \emph{lower $\theta$-intermediate dimension} of a set $F\subset \R^d$ is given by
	\begin{multline*}
		\underline{\dim}_{\, \theta}F = \inf \{s\geq 0:\, \text{for all } \varepsilon>0 \text{ and all } \delta_0>0, \text{ there exists } 0<\delta\leq \delta_0 \\
		\text{ and a cover } \{U_i\} \text{ of } F \text{ such that } \delta^{1/\theta}\leq |U_i| \leq \delta \text{ and } {\textstyle\sum} |U_i|^s\leq \varepsilon\},
	\end{multline*}
	while its \emph{upper $\theta$-intermediate dimension} is 
	\begin{multline*}%
		\overline{\dim}_{\, \theta}F = \inf \{s\geq 0:\, \text{for all } \varepsilon>0 \text{ there exists } \delta_0>0 \text{ such that for all } 0<\delta\leq \delta_0, \\
		\text{ there is a cover } \{U_i\} \text{ of } F \text{ such that } \delta^{1/\theta}\leq |U_i| \leq \delta \text{ and }  {\textstyle\sum} |U_i|^s\leq \varepsilon\}.
	\end{multline*}
	For a given $\theta$, if the values of $\underline{\dim}_{\, \theta}F$ and $\overline{\dim}_{\, \theta}F$ coincide, then the common value is called the \emph{$\theta$-intermediate dimension} and is denoted by  $\dim_{\, \theta}F$. We refer to the quantity ${\textstyle\sum} |U_i|^s$ as the \emph{$s$-cost} of the cover.
\end{definition}
With these definitions, $\underline{\dim}_{0}F=\overline{\dim}_{0}F=\dim_{\mathrm H}F$, $\underline{\dim}_{1}F=\underline{\dim}_{\mathrm B}F$ and $\overline{\dim}_{1}F=\overline{\dim}_{\mathrm B}F$, see~\cite[Chapter~2, Section~3.2]{Falconer2014main}. For the Hausdorff dimension there are no restrictions on the diameters of the covering sets, for the box dimension the diameters of the covering sets must be the same, while for $\dim_{\, \theta}F$ the restriction is to only consider covering sets with diameter in the range $[\delta^{1/\theta},\delta]$. As $\theta\to 0$, the $\theta$-intermediate dimension gives more insight into which scales are used in the optimal cover to reach the Hausdorff dimension. For $\theta<1$, a natural covering strategy to improve on the exponent given by the box dimension is to use covering sets with diameter of the two permissible extremes, in other words either $\delta^{1/\theta}$ or $\delta$. Intuitively, it is more cost efficient in terms of $s$-cost to cover $F$ with sets of diameter $\delta^{1/\theta}$ where $F$ is ``sparse'' in some appropriate sense and use the larger sets of diameter $\delta$ where $F$ is ``dense''.

Most previous examples whose $\theta$-intermediate dimensions are known explicitly \cite{Burrell2022spirals,Falconer2022seq,
Falconer2020firstintermediate,Tan2020} share the common feature that the function $\theta\mapsto \dim_{\, \theta}F$ for the range of $\theta$ where $\dim_{\, \theta}F>\dim_{\mathrm H}F$ is strictly increasing, concave and analytic, in particular there are no phase transitions. Interestingly, except in~\cite{Banaji2021infinite} and \cite{Banaji2022moran}, the heuristic of using the two extreme scales already leads to the optimal covering strategy. 
The standard approach to see whether a covering strategy gives the optimal exponent is to check whether a mass distribution principle~\cite[Proposition~2.2]{Falconer2020firstintermediate} applied to a measure constructed from this particular cover gives the matching lower bound. 

For a set $F$, the functions $\theta \mapsto \overline{\dim}_{\, \theta} F$ and $\theta \mapsto \underline{\dim}_{\, \theta} F$ are non-decreasing, and are continuous for $\theta \in (0,1]$ but not necessarily at $\theta = 0$. 
Banaji and Rutar \cite{Banaji2022moran} proved a necessary and sufficient condition for a given function $h(\theta)$ to be realised as the intermediate dimensions of a set. This condition involves a straightforward constraint on the Dini derivatives of $h(\theta)$, and demonstrates that the intermediate dimensions of sets can have highly varied behaviour. 
Banaji~\cite{Banaji2023gen} introduced a more general family of dimensions, called the $\Phi$-intermediate dimensions, which can provide more refined geometric information about sets whose intermediate dimensions are discontinuous at $\theta=0$. If the intermediate dimensions of a set are in fact continuous at $\theta = 0$, then this has consequences for its orthogonal projections, see~\cite[Section~6]{Burrell2021projections}, and its images under index-$\alpha$ fractional Brownian motion, see~\cite[Section~3]{Burrell2022brownian}.

Intermediate dimensions can also be formulated using capacity theoretic methods~\cite{Burrell2021projections}. This approach has been used to bound the H\"{o}lder regularity of maps that can deform one spiral into another~\cite{Burrell2022spirals},  compute the almost-sure value of the intermediate dimension of the image of Borel sets under index-$\alpha$ fractional Brownian motion in terms of dimension profiles~\cite{Burrell2022brownian} and under more general Rosenblatt processes~\cite{Daw2023}.

Qualitative information about $\dim_{\, \theta}F$ can also yield interesting applications. For example, the intermediate dimensions can be used to give information about the H{\"o}lder distortion of maps between sets (see \cite[Section~17.10]{Fraser2020book} for discussion of the H{\"o}lder mapping problem in the context of dimension theory). 
In \cite[Section~14.2.1~5.]{Falconer2021intdimsurvey} Falconer noted that if $F \subset \mathbb{R}^d$ and $f \colon F \to \mathbb{R}^d$ is an $\alpha$-H{\"o}lder map for some $\alpha \in (0,1]$, meaning that there exists $c>0$ such that $\|f(x)-f(y) \| \leq c\|x-y\|^\alpha$ for all $x,y \in F$, then 
\begin{equation}\label{eq:generalholderint} \overline{\dim}_{\, \theta} f(F) \leq \alpha^{-1} \overline{\dim}_{\, \theta} F 
\end{equation}
and $\underline{\dim}_{\, \theta}  f(F) \leq \alpha^{-1} \underline{\dim}_{\, \theta}  F$ (see \cite[Theorem~3.1]{Burrell2022brownian} and \cite[Section~4]{Banaji2023gen} for further H{\"o}lder distortion estimates). In \cite[Example~4.5]{Banaji2021infinite}, Banaji and Fraser showed that the intermediate dimensions can give better information about the H\"older distortion between some continued fraction sets than the Hausdorff or box dimensions.

We mention a similar concept of dimension interpolation between the upper box dimension and the (quasi-)Assouad dimension, called the \emph{Assouad spectrum}, which was initiated in~\cite{Fraser2018secondassouad,Fraser2018firstassspec}. We refer the reader to the surveys~\cite{Falconer2021intdimsurvey,Fraser2021interpolating} for additional references on the topic of dimension interpolation.

An important method for generating fractal sets is via \emph{iterated function systems (IFSs)}; see \cite{Falconer2014main} for general background. In this paper, we consider \emph{self-affine sets}, which are the attractors of IFSs consisting of affine contractions. We refer the reader to \cite{Falconer2013affine} for a survey of the dimension theory of self-affine sets, and to \cite{Barany2019HochmanRapaport} for an important result. 
Self-affine carpets are widely-studied families of sets in the plane which often have distinct Hausdorff and box dimensions. As such, it is natural to consider their intermediate dimensions. The main objective of this paper is to determine an explicit formula for the intermediate dimensions of the simplest model, originally introduced independently by~Bedford~\cite{Bedford1984phd} and McMullen~\cite{McMullen1984carpet}. Moreover, in the process we uncover new interesting features about the form of the intermediate dimensions and make an unexpected connection to multifractal analysis and bi-Lipschitz equivalence of these carpets.

\subsection{Bedford--McMullen carpets}

The construction of the carpet goes by splitting $[0,1]^2$ into $m$ columns of equal width and $n$ rows of equal height for some integers $n>m\geq 2$ and considering orientation preserving maps of the form
\begin{equation*}
	f_{(i,j)}(\underline{x})\coloneqq \begin{pmatrix} 1/m & 0 \\ 0 & 1/n \end{pmatrix} \begin{pmatrix} x \\ y \end{pmatrix} + \begin{pmatrix} (i-1)/m \\ (j-1)/n
	\end{pmatrix}
\end{equation*}
for the index set $(i,j)\in \mathcal{A}\subseteq \{1,\dotsc,m\}\times\{1,\dotsc,n\}$. It is well-known that associated to the iterated function system (IFS) $\mathcal{F}=\{f_{(i,j)}\}_{(i,j)\in \mathcal{A}}$ there exists a unique non-empty compact set $\Lambda=\Lambda_{\mathcal{F}}$, called the attractor, satisfying
\begin{equation*}
	\Lambda = \bigcup_{(i,j)\in \mathcal{A}} f_{(i,j)} ( \Lambda).
\end{equation*}
We call $\Lambda$ a \emph{Bedford--McMullen carpet}. The carpet $\Lambda$ can also be viewed as an invariant subset of the 2-torus $[0,1)^2$ under the toral endomorphism $(x,y)\mapsto (mx \mod 1, ny \mod 1)$. We refer the interested reader to the survey~\cite{Fraser2021bedford} for further references. The left hand side of Figure~\ref{fig:BMPicture} shows a simple example of a Bedford--McMullen carpet with distinct Hausdorff and box dimension. The three shaded rectangles show the image of $[0,1]^2$ under the three maps in the IFS, while the attractor is shown in red. For this carpet, $\dim_{\mathrm H}\Lambda\approx1.34968<1.36907\approx\dim_{\mathrm B}\Lambda$. 

Let $\Lambda$ be the Bedford--McMullen carpet associated to the IFS $\mathcal{F}$. Fix notation 
\[\gamma \coloneqq \frac{\log n}{\log m}>1.\]
For the remainder of the paper, we index the maps of $\mathcal{F}$ by $i\in\{1,\dotsc,N\}$. Let $1<M\leq m$ denote the number of non-empty columns and $\mathbf{N}\coloneqq (N_1,\dotsc, N_M)$, where $N_{\jh}$ is the number of maps $f_i$ that map $[0,1]^2$ to the $\jh$-th non-empty column. Observe that $N=N_1+\dotsb+N_M$.

Let $\mathcal{P}_M$ denote the set of probability vectors on $\{1,\dotsc,M\}$. 
The \emph{entropy} of $\mathbf{q}\in\mathcal{P}_M$ is
\begin{equation*}
	H(\mathbf{q}) = -\sum_{\jh=1}^M q_{\jh}\log q_{\jh},
\end{equation*}
where we use the convention that $0 \log 0 = 0$. 
We introduce
\begin{equation*}
	\mathbf{P}=(P_1,\dotsc,P_M) \coloneqq \left(\frac{N_1}{N},\dotsc,\frac{N_M}{N} \right)  \quad \mbox{ and } \quad   \mathbf{Q} \coloneqq \left( \frac{1}{M}, \dotsc, \frac{1}{M} \right) . 
\end{equation*}
For $\mathbf{q}\in\mathcal{P}_M$, it holds that $H(\mathbf{q})\leq \log M$ with equality if and only if $\mathbf{q} = \mathbf{Q}$. 
For the entire paper, we also introduce  
\begin{equation}\label{e:underlineoverlinedef}
\underline{t}\coloneqq \frac{1}{M} \sum_{\jh=1}^{M} \log N_{\jh} \quad \mbox{ and } \quad \overline{t}\coloneqq \log N-H(\mathbf{P}).
\end{equation}
We say that $\Lambda$ has \emph{uniform (vertical) fibres} if and only if $\mathbf{P}=\mathbf{Q}$, in other words each non-empty column has the same number of maps.

Bedford and McMullen showed that the Hausdorff and box dimensions of $\Lambda$ are equal to  
\begin{align}
	\dim_{\mathrm H}\Lambda &= \frac{1}{\log m}\log \Bigg( \sum_{\jh=1}^MN_{\jh}^{\frac{\log m}{\log n}} \Bigg) \label{eq:102}, \\          %
	\dim_{\mathrm B}\Lambda &= \frac{\log N}{\log n} + \left(1-\frac{\log m}{\log n}\right)\frac{\log M}{\log m}. \label{eq:103} 
\end{align}
In particular, $\dim_{\mathrm H}\Lambda = \dim_{\mathrm B}\Lambda$ if and only if $\Lambda$ has uniform fibres. In this case, $\theta\mapsto \dim_{\, \theta}\Lambda$ is just a constant function. Therefore, we assume throughout that the carpet has non-uniform fibres, in which case the appropriate dimensional Hausdorff measure of $\Lambda$ is infinite~\cite{Peres1994infinitemeasure}. Using the concavity of the logarithm function, an immediate consequence of non-uniform fibres is that $\underline{t} < \log(N/M)$.

\subsubsection*{Previous results}

Previous papers on the topic~\cite{Falconer2020firstintermediate,Fraser2021interpolating,Kolossvary2022bm} have established crude bounds for the intermediate dimensions and speculated about the possible form. The question of determining the intermediate dimensions of all Bedford--McMullen carpets was explicitly asked in \cite{Falconer2021intdimsurvey,Falconer2020firstintermediate,Fraser2021interpolating,Fraser2021bedford,  Kolossvary2022bm}. Loosely speaking, the results of Falconer, Fraser and Kempton~\cite{Falconer2020firstintermediate} concentrate on the behaviour of $\dim_{\, \theta}\Lambda$ for $\theta$ close to 0, while the results of Kolossv\'ary~\cite{Kolossvary2022bm} concentrate on the behaviour for $\theta\geq \gamma^{-1}$. 

More precisely, a linear lower bound was obtained in~\cite{Falconer2020firstintermediate} which shows that $\underline{\dim}_{\, \theta}\Lambda>\dim_{\mathrm H}\Lambda$ for every $\theta\in(0,1]$. This bound can be used to show that $\theta\mapsto\underline{\dim}_{\, \theta}\Lambda$ cannot be convex in general. Moreover, the authors give a general lower bound which reaches $\dim_{\mathrm B}\Lambda$ at $\theta=1$. This general lower bound was improved in \cite[Proposition~3.10]{Banaji2023gen}. For many, but not all carpets, a lower bound in~\cite{Kolossvary2022bm} performs better than these general bounds and the linear bound for larger values of $\theta$. The lower bound depicted in Figure~\ref{fig:BMPicture} in orange is the best combination of these results. 

Falconer, Fraser and Kempton~\cite[Proposition~4.1]{Falconer2020firstintermediate} show an upper bound of the form $\overline{\dim}_{\, \theta}\Lambda \leq \dim_{\mathrm H}\Lambda + c/(-\log \theta)$ for an explicit $c>0$ and $\theta$ sufficiently small. In particular, this implies that $\theta\mapsto\underline{\dim}_{\, \theta}\Lambda$ and $\theta\mapsto\overline{\dim}_{\, \theta}\Lambda$ are continuous also at $\theta=0$. Hence, the results of Burrell, Falconer and Fraser~\cite[Section~6]{Burrell2021projections} and Burrell~\cite[Section~3]{Burrell2022brownian} can be applied. For example, if $\dim_{\mathrm H} \Lambda < 1 \leq \dim_{\mathrm B} \Lambda$ then $\overline{\dim}_{\mathrm B} \pi(\Lambda) < 1$ for every orthogonal projection $\pi$ from $\mathbb{R}^2$ onto a 1-dimensional subspace. For almost every projection, $\underline{\dim}_{\, \theta} \pi(\Lambda)$ and $\overline{\dim}_{\, \theta}\pi(\Lambda)$ are continuous at $\theta = 0$, and if $\gamma \notin \mathbb{Q}$ then this holds for \emph{every} orthogonal projection. 
Furthermore, if $B_h \colon \mathbb{R}^2 \to \mathbb{R}^2$ is index-$h$ fractional Brownian motion, then $\theta\mapsto\underline{\dim}_{\, \theta}B_h(\Lambda)$ and $\theta\mapsto\overline{\dim}_{\, \theta}B_h(\Lambda)$ are almost surely continuous, and if $h > (\dim_{\mathrm H} \Lambda)/2$ then almost surely $\overline{\dim}_{\mathrm B} B_h (\Lambda) < 2$. 

A cover of $\Lambda$ is constructed in~\cite{Kolossvary2022bm} using just the two extreme scales to obtain an explicit upper bound of the form $\overline{\dim}_{\, \theta}\Lambda\leq\dim_{\mathrm B}\Lambda -\Delta(\theta)$ for $\theta\geq \gamma^{-1}$, where $\Delta(\theta)\searrow 0$ as $\theta\to 1$ and has a strictly positive derivative at $\theta=1$. 
This bound was used to show that $\overline{\dim}_{\, \theta}\Lambda$ is not concave for the whole range of $\theta$ in general, already hinting at richer behaviour than previously witnessed in other examples. Figure~\ref{fig:BMPicture} shows this upper bound in green. 

\subsection{Summary of results}
The formal statements are presented in Section~\ref{sec:CompResults}. In Theorem~\ref{thm:main} we state an explicit formula for $\underline{\dim}_{\, \theta}\Lambda=\overline{\dim}_{\, \theta}\Lambda$ for all $\theta\in(0,1]$, thus completely resolving the problem of calculating the intermediate dimensions of any Bedford-McMullen carpet $\Lambda$. For illustration see the right hand side of Figure~\ref{fig:BMPicture}, where $\theta\mapsto\dim_{\, \theta}\Lambda$ is plotted in blue for the carpet on the left hand side of Figure~\ref{fig:BMPicture}. Central to the formula is a large deviations rate function for which we give three additional equivalent characterisations in Proposition~\ref{prop:1}: one in terms of a pressure-like function, another as a certain probability vector with an entropy maximising property, and finally a relationship to the multifractal spectrum of the uniform self-affine measure on $\Lambda$. In the proof of Theorem~\ref{thm:main}, we construct a cover of $\Lambda$ that uses an increasing number of different scales in the permissible range as $\theta\to 0$. 
We also show in Corollary~\ref{cor:twoscales} that using more than two scales is necessary. 

We prove all the features suggested by the plot in Figure~\ref{fig:BMPicture} about the form of the intermediate dimensions for all carpets in Corollary~\ref{cor:allprop}. Namely, $\theta\mapsto\dim_{\, \theta}\Lambda$ is strictly increasing and has phase transitions at all negative integer powers of $\gamma$. Between consecutive phase transitions the intermediate dimensions are analytic and strictly concave. Moreover, for $\theta$ small enough $\dim_{\, \theta}\Lambda$ behaves like $\dim_{\mathrm H}\Lambda+c(\log\theta)^{-2}$. In particular, the derivative tends to $+\infty$ as $\theta\to0$. No previous family of sets has shown such rich and complex behaviour. Some illustrative examples are presented in Section~\ref{subsec:ex}.

We show in Theorem~\ref{thm:multifractal} that two different carpets with non-uniform vertical fibres have equal intermediate dimensions for every $\theta\in[0,1]$ if and only if the multifractal spectra of the uniform Bernoulli measure on the two carpets are equal. If, in addition, it is assumed that the two carpets are defined on the same grid, then Theorem~\ref{thm:multifractal} provides further equivalent conditions for their intermediate dimensions to be the same: a certain condition on the rate functions appearing in the formula or certain relationships between the parameters of the carpets, or the equality of the intermediate dimensions on any one open interval of $[\gamma^{-1},1]$.

Our main application relates to bi-Lipschitz equivalence. It is known~\cite[Corollary~1.1]{RaoPreprintlipschitz} that the equality of these multifractal spectra is necessary for two carpets to be bi-Lipschitz equivalent if it is assumed that the two carpets are defined on the same grid and are totally disconnected. Since bi-Lipschitz maps preserve intermediate dimensions, Theorem~\ref{thm:multifractal} implies that both of these assumptions can be dropped, see Corollary~\ref{cor:biLip}. In Example~\ref{ex:biLip} we construct two carpets which are not bi-Lipschitz equivalent by Corollary~\ref{cor:biLip}, but where this does not follow from~\cite{RaoPreprintlipschitz}. To our knowledge, this is the first instance where intermediate dimensions are used to show that two sets are not bi-Lipschitz equivalent, but where this fact does not follow from any other notion of dimension or existing result. For this example, we also use the intermediate dimensions to give estimates on the H\"older distortion.

For comparison, we mention that the calculation of the Assouad spectrum of Bedford--McMullen carpets~\cite{Fraser2018secondassouad} is not as involved as the proof of Theorem~\ref{thm:main} for the intermediate dimensions. Indeed, the intermediate dimensions are more subtle in that they depend on all the $N_i$ individually, as does the Hausdorff dimension. This is in contrast to the Assouad spectrum (and indeed the lower spectrum and the box, packing, Assouad, quasi-Assouad, lower and quasi-lower dimensions) which depend only on $m,n,N,M,\max_{1 \leq \ih \leq M} N_{\ih}$ and $\min_{1 \leq \ih \leq M} N_{\ih}$, see~\cite{Fraser2021bedford}. The Assouad spectrum also has just one phase transition at $\theta=\gamma^{-1}$, which occurs when the spectrum reaches the Assouad dimension and thus is constant for $\theta\in(\gamma^{-1},1)$.

\section{Complete statements and examples}\label{sec:CompResults}

\subsection{Main result: formula for intermediate dimensions}

Recalling~\eqref{e:underlineoverlinedef}, let $t\in(\underline{t},\overline{t})$. 
This is a non-empty interval because non-uniform fibres implies that $\underline{t} < \log (N/M) < \overline{t}$. 
Let $X_1,X_2,\dotsc,X_J,\dotsc $ be a sequence of independent and identically distributed random variables taking values in the set $\{\log N_1,\dotsc,\log N_M\}$, with 
\begin{equation}\label{e:definerandomvar} 
\mathds{P}(X_1 = \log N_{\ih}) = \frac{1}{M} \cdot \# \{ \, \jh \in \{1,\dotsc,M\} : N_{\jh} = N_{\ih} \, \} .
\end{equation}
 Then $\underline{t}$ is the expectation of $X_1$. 
The large deviations rate function of the average $\frac{1}{J}\sum_{i=1}^J X_i$ is
\begin{equation}\label{eq:22}
I(t)=\sup_{\lambda\in \R} \bigg\{\lambda t - \log\bigg( \frac{1}{M} \sum_{\jh=1}^{M} N_{\jh}^{\lambda}\bigg)\bigg\},
\end{equation} 
noting that $\frac{1}{M} \sum_{\jh=1}^{M} N_{\jh}^{\lambda}$ is the expectation of $e^{\lambda X_1}$. 
For $t \in [\underline{t},\max_{1 \leq \ih \leq M} \log N_{\ih})$, differentiating shows that the supremum in the definition of $I(t)$ is attained at the unique $\lambda \geq 0$ satisfying $t = \sum_{\ih=1}^M \frac{ N_{\ih}^\lambda }{\sum_{\jh=1}^M N_{\jh}^\lambda}\log N_{\ih}$. This allows $I(t)$ to be calculated numerically. On this interval, $I(t)$ is real analytic (as the Legendre transform of an analytic function). %
The derivative $I'(t)>0$ is the value of $\lambda$ at which the supremum is attained and $I(t)$ is strictly increasing for $t\in[\underline{t},\max \log N_{\ih}]$. %
Moreover, $I''(t) > 0$, so the function is strictly convex on this interval. Some particular values of interest are 
\begin{alignat*}{2}
I(\underline{t}) &= 0, \qquad \qquad &&I'(\underline{t}) = 0, \\ 
I(\overline{t}) &= \log M - H(\mathbf{P}), \qquad \qquad &&I'(\overline{t}) = 1, 
\end{alignat*}
see~\cite[Lemma~2.2.5]{Dembo2010largedeviations}. 
Moreover, 
\[ I(\max_{1 \leq \ih \leq M} \log N_{\ih}) = \log M - \log \# \{ \, \jh \in \{1,\dotsc,M\} : N_{\jh} = \max_{1 \leq \kh \leq M} N_{\kh} \, \}, \]
and 
\[ I'(t) \to \infty\quad \mbox{ as } t \to (\max_{1 \leq \ih \leq M} \log N_{\ih})^- . \] 

For $s \in \mathbb{R}$, %
we define the function $T_s \colon \mathbb{R} \to \mathbb{R}$ by 
\begin{equation}\label{eq:defineiteratingfunction}
T_s(t) \coloneqq \left(s-\frac{\log M}{\log m}\right)\log n +\gamma I(t). 
\end{equation}
For $\ell \in \mathbb{N}$ we denote the composition by $T_s^{\ell} \coloneqq \underbrace{T_s \circ \dotsb \circ T_s }_{\ell \mbox{ times}}$, and $T_s^0$ denotes the identity function. 
We use the sequences $(t_{\ell})_{\ell=1}^\infty = (t_{\ell} (s))_{\ell=1}^\infty$ defined by 
\begin{equation}\label{eq:definetsequence}
t_{\ell}(s) \coloneqq T_s^{\ell-1}\left(\left(s-\frac{\log M}{\log m}\right)\log n \right).
\end{equation}
Note that these depend only on $s$ and the carpet, but not on $\theta$. Observe that $T_s(\underline{t}) = t_1(s)$ for all $s \in \mathbb{R}$. 
We are now ready to state the main result of this paper; see Section~\ref{sec:proofofmain} for the rather involved proof. 
\begin{theorem}\label{thm:main}
Let $\Lambda$ be a Bedford--McMullen carpet with non-uniform vertical fibres. For all $\theta \in (0,1)$, $\dim_{\, \theta} \Lambda$ exists %
and is given in the following way. For fixed $\theta \in (0,1)$ let $L = L(\theta) \coloneqq 1 + \lfloor \frac{-\log \theta}{\log \gamma} \rfloor$, so $\gamma^{-L} < \theta \leq \gamma^{-(L-1)}$. %
Then there exists a unique solution $s = s(\theta) \in (\dim_{\mathrm H} \Lambda, \dim_{\mathrm B} \Lambda)$ to the equation 
\begin{equation}\label{eq:mainformula}
\gamma^L\theta \log N - (\gamma^L\theta - 1) t_L(s) + \gamma(1-\gamma^{L-1}\theta)(\log M - I(t_L(s))) - s\log n = 0, 
\end{equation}
and $s(\theta) = \dim_{\, \theta} \Lambda$. 
\end{theorem}
In the case $L=1$ the formula~\eqref{eq:mainformula} simplifies to 
\[ \dim_{\mathrm B} \Lambda - \frac{1}{\log n} \left( \frac{1}{\theta} - 1\right)I(t_1(s)) - s = 0.\] 
If $\theta = \gamma^{-(L-1)}$ for some $L \in \mathbb{N}$ with $L \geq 2$, then it becomes 
\[ \dim_{\mathrm B}\Lambda - \frac{1}{\log m}\left(1-\frac{1}{\gamma}\right)I(t_{L-1}(s)) - s = 0.\] %

Theorem~\ref{thm:main} and Corollary~\ref{cor:allprop} below fully resolve \cite[Problem~15.8.1]{Fraser2021bedford} and the questions about Bedford--McMullen carpets in Falconer's survey paper~\cite[Section~14.8]{Falconer2021intdimsurvey}, and indeed provide more information. In particular, this is the first time it has been shown that the intermediate dimensions of Bedford--McMullen carpets exist for $\theta \in (0,1)$. %
Tools used in the proof include the method of types (see~\cite{Kolossvary2022lqtypes}) and a variant of a mass distribution principle for the intermediate dimensions, see Proposition~\ref{prop:mdp}. 
The proof of the upper bound involves the construction of an explicit cover using scales $\delta, \delta^{\gamma}, \delta^{\gamma^2}, \dotsc, \delta^{\gamma^{L-1}}$ and $\delta^{1/\theta},\delta^{1/(\gamma\theta)}, \dotsc, \delta^{1/(\gamma^{L-1}\theta)}$. This cover consists of \emph{approximate squares}, which we define in Section~\ref{subsec:approxsquare}. We decide which parts of each approximate square to cover at which scale depending on how the different parts of the symbolic representation of the approximate square relate to each other. 
The proof simplifies when $\theta \geq 1/\gamma$ (where we just use the smallest and largest permissible scales), and when $\theta = \gamma^{-k}$ for $k \in \mathbb{N}$ (where we use scales $\delta,\delta^{\gamma},\dotsc,\delta^{\gamma^k}$ due to scales `lining up'). 
Note that the cover jumps from using $2k$ scales when $\theta \in (\gamma^{-k},\gamma^{-(k-1)})$ to using $2k + 2$ scales when $\theta \in (\gamma^{-(k+1)},\gamma^{-k})$ (and uses only $k$ scales when $\theta = \delta^{-k}$), which gives an indication of why one might expect a phase transition at $\theta = \gamma^{-k}$. 

In~\cite[Section~14.8]{Falconer2021intdimsurvey}, Falconer asked about the number of different scales of covering sets needed to approximate the intermediate dimensions of sets arbitrarily closely from above. 
Corollary~\ref{cor:twoscales}, which we prove in Section~\ref{subsec:proofform}, %
shows that for Bedford--McMullen carpets with non-uniform fibres, more than two scales are needed when $\theta$ is small. It has been shown that there exist sets where an unbounded number of scales is required as $\delta \to 0^+$, see \cite[Remark~3.11]{Banaji2022moran}. %
\begin{corollary}\label{cor:twoscales}
Let $\Lambda$ be a Bedford--McMullen carpet with non-uniform vertical fibres. There exist $\theta_0, \epsilon, \delta_0 > 0$ such that for all $\theta \in (0,\theta_0)$ and any cover $\{U_i\}$ of $\Lambda$ that uses at most two scales, both of which are less than $\delta_0$, we have $\sum_i |U_i|^{\dim_{\, \theta} \Lambda + \epsilon} \geq 1$. 
\end{corollary}%

A possible direction for further research could be to investigate the intermediate dimensions of higher dimensional Bedford--McMullen sponges, or the self-affine carpets in the plane of Lalley and Gatzouras~\cite{Lalley1992gatzouras} or Bara\'nski~\cite{Baranski2007carpet}. However, we expect this to be challenging, not least because it is not clear what would take the place of the important quantity $\gamma$. 
In light of the recent paper~\cite{BanajiPreprintgl} proving that the Assouad spectrum of Gatzouras--Lalley carpets (unlike Bedford--McMullen carpets) can be a differentiable function of $\theta$, it is natural to ask whether the intermediate dimensions of Gatzouras--Lalley carpets can also be differentiable on $(0,1)$. 

We now continue with corollaries of Theorem~\ref{thm:main} about the form of the graph of the function $\theta \mapsto \dim_{\, \theta} \Lambda$ that do not follow from the general theory, give a rather unexpected connection to multifractal analysis and bi-Lipschitz equivalence of two carpets, and give equivalent formulations of the rate function $I(t)$. 

\subsection{Form of the intermediate dimensions}

We assume that the carpet has non-uniform fibres, otherwise, $\theta \mapsto \dim_{\, \theta} \Lambda$ is just a constant function. %
We denote the left and right derivatives at $\theta$ by 
\begin{equation*}
\partial_- \dim_{\, \theta} \Lambda \coloneqq \lim_{h \to 0^+} \frac{\dim_{\, \theta} \Lambda - \dim_{\, \theta - h} \Lambda}{h} \;\text{ and }\; \partial_+ \dim_{\, \theta} \Lambda \coloneqq \lim_{h \to 0^+} \frac{\dim_{\, \theta + h} \Lambda - \dim_{\, \theta} \Lambda}{h}. 
\end{equation*}

\begin{corollary}\label{cor:allprop}
Let $\Lambda$ be a Bedford--McMullen carpet with non-uniform vertical fibres. Then the function $\theta \mapsto \dim_{\, \theta} \Lambda$ has the following properties:
\begin{enumerate}[(i)]
\item it is real analytic on the interval $(\gamma^{-L},\gamma^{-(L-1)})$ for all $L \in \mathbb{N}$; \label{itemi}
\item $\partial_- \dim_{\, \theta} \Lambda$ exists at every $\theta \in (0,1]$ and $\partial_+ \dim_{\, \theta} \Lambda$ exists at every $\theta \in (0,1)$; \label{itemii}
\item it is strictly increasing and has phase transitions at every negative integer power of $\gamma$. More precisely, there exists $C_0>0$ depending only on $\Lambda$ such that for all $\theta \in (0,1)$,
\begin{equation*}
C_0 < \partial_- \dim_{\, \theta} \Lambda  \leq  \partial_+ \dim_{\, \theta} \Lambda
\end{equation*}
with equality if and only if for all $L \in \mathbb{N}$ we have $\theta\neq \gamma^{-L}$. Moreover, $\frac{\partial_+ \dim_{\gamma^{-L}} \Lambda}{\partial_- \dim_{\gamma^{-L}} \Lambda}$ converges to a constant in $(1,\infty)$ as $L \to \infty$; \label{itemiii}
\item there exist $C \in [1,\infty)$ and $\theta_0 > 0$ depending only on $\Lambda$ such that for all $\theta \in (0,\theta_0]$,
\[ \dim_{\mathrm H} \Lambda + \frac{C^{-1}}{(\log \theta)^2} \leq \dim_{\, \theta} \Lambda \leq \dim_{\mathrm H} \Lambda + \frac{C}{(\log \theta)^2}; \]
\label{itemiv}
\item it is strictly concave on the interval $[\gamma^{-L},\gamma^{-(L-1)}]$ for all $L \in \mathbb{N}$. \label{itemv}
\end{enumerate}
\end{corollary}
In~\cite[Proposition~4.1]{Falconer2020firstintermediate} for small enough $\theta$ the upper bound $\overline{\dim}_{\, \theta}\Lambda\leq \dim_{\mathrm H}\Lambda+c(-\log \theta)^{-1}$ was proved for a constant $c$ depending only on $\Lambda$. Corollary~\ref{cor:allprop}~\eqref{itemiv} shows that although this bound is not sharp, we do indeed have that $\frac{\dim_{\, \theta} \Lambda - \dim_{\mathrm H} \Lambda}{\theta} \to \infty$ as $\theta \to 0^+$. 

\subsection{Connection to multifractal analysis and bi-Lipschitz equivalence}\label{subsec:multifractal}

In this section, and in Section~\ref{s:multifractalproof} where we prove the results in this section, it is convenient to change notation. Here, the parameters $(M_0,N_1,\dotsc,N_{M_0}, R_1,\dotsc, R_{M_0})$ will define a Bedford--McMullen carpet. 
Now $M_0$ denotes the number of different values that the number of maps in a non-empty column can take. 
Note that $M_0\geq 2$, since $M_0=1$ corresponds to the uniform fibre case. 
We write $N_1,\dotsc,N_{M_0}$ for the actual values that the number of maps in a non-empty column can take, and we order them as $N_1>N_2>\dotsb>N_{M_0}$. 
For each $\ih \in \{1,\dotsc,M_0\}$, we write $R_{\ih}$ for the number of columns containing exactly $N_{\ih}$ maps. 
As with the previous notation, we write $M = \sum_{\ih = 1}^{M_0} R_{\ih}$ for the number of non-empty columns, and $N = \sum_{\ih = 1}^{M_0} R_{\ih} N_{\ih}$ for the total number of maps. 
For example, for the carpet in Figure~\ref{fig:BMPicture}, the number of maps in a non-empty column is either~$1$ or~$2$, so $N_1 = 2$ and $N_2 = 1$; each corresponds to just one column, so $R_1 = R_2 = 1$, and $M_0 = \# \{1,2\} = 2$. 

A central problem in multifractal analysis is to examine the way a Borel measure $\mu$ is spread over its support $\supp\, \mu$. For a survey of this topic, we refer the reader to~\cite[Chapter~17]{Falconer2014main}. 
The \emph{local dimension of $\mu$ at $x$} is
\begin{equation*}
\dim_{\mathrm{loc}}(\mu,x) = \lim_{r\to 0} \frac{\log \mu(B(x,r))}{\log r}
\end{equation*}
if the limit exists, which approximately measures the rate of decay of $\mu(B(x,r))$ as a power law $r^{\alpha}$. The measure $\mu$ is \emph{exact dimensional} if $\dim_{\mathrm{loc}}(\mu,x)$ is equal to a specific $\alpha$ for $\mu$-almost all $x$. However, $\dim_{\mathrm{loc}}(\mu,x)$ can still potentially take up a whole spectrum of different $\alpha$. This motivates the definition of the \emph{fine} or \emph{Hausdorff multifractal spectra}
\begin{equation*}
f_{\mu}(\alpha) \coloneqq \dim_{\mathrm H} \{ \, x\in\supp\, \mu : \dim_{\mathrm{loc}}(\mu,x)=\alpha \, \}.
\end{equation*}

Concentrating on the self-affine setting, given a self-affine iterated function system $\mathcal{S}=\{S_1,\dotsc,S_N\}$, meaning that all $S_i\colon \mathbb{R}^d \to\mathbb{R}^d$ are contracting affine maps, and a probability vector $\mathbf{p}$ with strictly positive entries, the \emph{self-affine measure} $\mu_{\mathbf{p}}$ is the unique probability measure supported on the attractor of $\mathcal{S}$ satisfying
\begin{equation*}
\mu_{\mathbf{p}}(A) = \sum_{i=1}^N p_i \mu_{\mathbf{p}}(S_i^{-1}A) \;\text{  for all Borel sets } A\subset \mathbb{R}^d.
\end{equation*} 
It is known that all self-affine measures are exact dimensional~\cite{Barany2017exactdim,
Feng2023exactdim} and the dimension satisfies a Ledrappier--Young type formula; this was resolved earlier in~\cite{Kenyon1996bmexactdim} for those supported on Bedford--McMullen carpets. 
The fine multifractal spectrum of self-affine measures on Bedford--McMullen carpets is also known~\cite{Barral2007bmmultifractal,
Jordan2011bm,King1995bmmultifractal} and (under the separation condition assumed in~\cite{King1995bmmultifractal}) has been generalised to higher dimensions in~\cite{Olsen1998sponge}. 
When $\mathbf{p}=(1/N,\dotsc,1/N)$ is the uniform vector, we simply write $\nu=\mu_{\mathbf{p}}$ and call it the \emph{uniform self-affine measure}. 
In this case, define the function $\beta_{\nu}(\xi)$ for $\xi \geq 0$ by \begin{equation}\label{eq:definebeta}
m^{-\beta_{\nu}(\xi)}N^{-\xi} \sum_{\ih=1}^{M_0} R_{\ih} N_{\ih}^{\gamma^{-1} + (1-\gamma^{-1})\xi} = 1.
\end{equation}
 Note that because of the minus sign before $\beta_{\nu}(\xi)$ (which in some papers is erroneously omitted), $\beta_{\nu}(\xi)$ is a convex function. 
 Define
 \begin{equation*}
\alpha_{\text{min}} \coloneqq \frac{\log N}{\log m} - \frac{1 - \gamma^{-1}}{\log m} \log N_{1}; \qquad \alpha_{\text{max}} \coloneqq \frac{\log N}{\log m} - \frac{1 - \gamma^{-1}}{\log m} \log N_{M_0}. 
\end{equation*}%
 Then by \cite[Theorem~1]{Jordan2011bm}, the multifractal spectrum is 
\begin{equation}\label{eq:uniformmultifractal} f_{\nu}(\alpha) = \inf_{\xi} (\alpha \xi + \beta_{\nu}(\xi)) = -\beta_{\nu}^*(-\alpha) \qquad \mbox{for all} \qquad \alpha \in (\alpha_{\text{min}},\alpha_{\text{max}}), 
\end{equation}
where $\beta_{\nu}^*(\alpha') \coloneqq \sup_{\xi'} (\alpha' \xi' - \beta_{\nu}(\xi'))$ is the Legendre transform of $\beta_{\nu}$ (defined in the same way as for the rate function in~\eqref{eq:22}). 

Another quantity that is closely related to the multifractal spectrum is the $L^q$ spectrum of a measure $\mu$. It is a function $T_\mu \colon \R \to \R$ which quantifies the global fluctuation of $\mu$, and its value at $q=0$ describes the box dimension of the support of the measure. %
The $L^q$ spectrum can be defined by 
\begin{equation}\label{e:lqdef}
T_\mu(q) \coloneqq \lim_{\delta \to 0^+} \frac{\log T_{\delta}(\mu,q)}{-\log \delta} 
\end{equation}
(the limit exists for all measures we consider), where 
\[ T_{\delta}(\mu,q) \coloneqq \sup \left\{ \, \sum_i (\mu(B_i))^q : B_i \mbox{ disjoint balls of radius } \delta \mbox{ centred in } \supp(\mu) \, \right\}. \]
Let $\nu$ be the uniform self-affine measure on a Bedford--McMullen carpet with non-uniform fibres, and let $\nu_x$ be the measure obtained by projecting $\nu$ orthogonally onto the $x$-axis. 
Note that $\nu_x$ is a self-similar measure with all contraction ratios equal to $1/m$, satisfying the open set condition; the $L^q$ spectrum of such measures has been studied in~\cite{Cawley1992multifractal} and \cite[Chapter~17]{Falconer2014main}. 
For $q$ in an open neighbourhood of $1$, $T_{\nu_x}(q)$ satisfies  
\[ \sum_{\ih = 1}^{M_0} R_{\ih} \left(\frac{N_{\ih}}{N}\right)^q \left(\frac{1}{m}\right)^{T_{\nu_x}(q)} = 1. 
\]%
A direct calculation shows that 
\[ 
T_{\nu_x}(q) = -\frac{\log N}{\log m} q + \frac{\log \sum_{\ih = 0}^{M_0} R_{\ih} (N_{\ih})^q}{\log m}.
\] 
Measures such as $\nu_x$ satisfy the \emph{multifractal formalism}, meaning that the Legendre transform of its $L^q$ spectrum equals its multifractal spectrum, but this is not generally satisfied by $\nu$. 
Instead, a result of Feng and Wang \cite[Theorem~2]{Feng2005lq} shows that $T_\nu(q)$ satisfies 
\[ 
N \cdot N^{-q} m^{-T_{\nu_x}(q)} n^{-(T_\nu(q) - T_{\nu_x}(q))} = 1,
\]
and a direct manipulation shows that for $q$ in an open neighbourhood of~$1$, 
\begin{align}\label{e:lqformula}
\begin{split}
T_\nu(q) &= \frac{\log N}{\log n} - q \frac{\log N}{\log m} + \left( 1 - \frac{\log m}{\log n}\right) \frac{\log \sum_{\ih=1}^{M_0} R_{\ih} (N_{\ih})^q}{\log m} \\*
&= \left( 1-\frac{\log m}{\log n}\right) \beta_{\nu}\left( \frac{\log (n^{q}/m)}{\log(n/m)} \right). 
\end{split}
\end{align}

Before describing the connection between the multifractal spectra and the intermediate dimensions, we observe that the grid on which a carpet can be defined is not unique, and in fact by iterating the IFS one can see that every carpet can be defined on infinitely many grids. For example, by iterating the IFS, the carpet from Figure~\ref{fig:BMPicture} can be defined on a $2 \times 3$ grid or on a $4 \times 9$ grid (though of course many carpets on a $4 \times 9$ grid cannot be realised on a $2 \times 3$ grid). 
Theorem~\ref{t:grid} gives information about the grids on which a carpet can be defined, and is proved in Section~\ref{s:multifractalproof}. %
It demonstrates for example that since $2$ and $3$ are multiplicatively independent, a carpet with non-uniform fibres defined on a $2 \times 4$ grid cannot be defined on a $3 \times 9$ grid (even though $\log 4 / \log 2 = \log 9 / \log 3$). 

\begin{theorem}\label{t:grid}
\begin{enumerate}
\item\label{i:gridonecarpet} 
If a Bedford--McMullen carpet $\Lambda$ with non-uniform fibres can be defined on both a $m \times n$ grid and on a $m' \times n'$ grid, then $\log n / \log m = \log n' / \log m'$ and $\log n / \log n' \in \mathbb{Q}$. %
\item\label{i:gridtwocarpets} 
Consider two carpets $\Lambda_1$ and $\Lambda_2$ with non-uniform fibres which are defined by IFSs $\mathcal{S}_1$ and $\mathcal{S}_2$ on grids of size $m_1 \times n_1$ and $m_2 \times n_2$ respectively. 
Then they can be realised on the same grid if and only if 
\[ \frac{\log n_1}{\log n_2} = \frac{\log m_1}{\log m_2} \in \mathbb{Q}. \]  
\end{enumerate}
\end{theorem}
In fact, we will see below that if two carpets with non-uniform fibres are bi-Lipschitz equivalent then they can be defined on the same grid. The same is true even if we merely assume the carpets have the same intermediate dimensions, or support uniform Bernoulli measures with equal multifractal spectra. 

We make some remarks about part~\eqref{i:gridtwocarpets}. The reverse implication is immediate. Indeed, if $\log n_1 / \log n_2 = \log m_1 / \log m_2  = a/b \in \mathbb{Q}$, then the $b$-th iterate of $\mathcal{S}_1$ and the $a$-th iterate of $\mathcal{S}_2$ are both defined on the same grid of size $m_1^b \times n_1^b$. 
It is straightforward to see that the rate function $I(t)$ of $\mathcal{S}_1$ and the rate function $I^{(b)}(t)$ of the $b$-th iterate of $\mathcal{S}_1$ are related by $I^{(b)}(bt) =b I(t)$. 
For the forward implication, if both carpets can be realised on the same grid, then the fact that $\log n_1 / \log m_1 = \log n_2 / \log m_2$ was noted by Fraser and Yu using the Assouad spectrum in \cite[Theorem~3.3]{Fraser2018secondassouad}, and also follows from the intermediate dimension formula, noting that geometric quantities such as dimensions of course do not change by taking an iterate of the system. 
The fact that $n$ and $n'$ must be multiplicatively dependent follows from Lemma~\ref{lem:multifractallemma}, and is related to work of Meiri and Peres~\cite[Theorem~1.2]{Meiri1999furstenberg}. 
This is in turn related to Furstenberg's $\times 2, \times 3$ principle (which suggests that expansions in multiplicatively independent bases should have no common structure). 

Our next result, which we prove in Section~\ref{s:multifractalproof} using Theorem~\ref{thm:main}, gives a direct connection between the intermediate dimensions and the multifractal and $L^q$ spectra of the uniform self-affine measure. 

\begin{theorem}\label{thm:multifractal}
Let $\Lambda$ and $\Lambda'$ be two Bedford--McMullen carpets with non-uniform fibres, and denote the corresponding uniform self-affine measures by $\nu$ and $\nu'$. Then the following are equivalent:
\begin{enumerate}
\item $\dim_{\, \theta}\Lambda = \dim_{\, \theta}\Lambda'$ for every $\theta\in[0,1]$; \label{item1}
\item $f_{\nu}(\alpha)=f_{\nu'}(\alpha)$ for all $\alpha \in (\alpha_{\text{min}},\alpha_{\text{max}})$. \label{itemnow2}
\end{enumerate}
Moreover, if~\eqref{item1},~\eqref{itemnow2} hold, then both carpets can be defined on the same grid. %

Now assume that $\Lambda$ and $\Lambda'$ are defined on the same $m\times n$ grid to begin with, with parameters $\{M_0,N_1,\dotsc,N_{M_0}, R_1,\dotsc, R_{M_0}\}$ and $\{M'_0,N'_1,\dotsc,N'_{M'_0}, R'_1,\dotsc, R'_{M'_0}\}$, respectively. Denote the corresponding rate functions defined in~\eqref{eq:22} by $I(t)$ and $I'(t)$. Let $\underline{t}$ and $\overline{t}$ be as defined previously, for the carpet $\Lambda$. Let $(a,b) \subset (\gamma^{-1},1)$ be a (non-empty) open interval. Then each of~\eqref{item1},~\eqref{itemnow2} is equivalent to each of the following: 
\begin{enumerate}
\setItemnumber{3}\item $\dim_{\, \theta}\Lambda = \dim_{\, \theta}\Lambda'$ for every $\theta \in (a,b)$; \label{itemnow3}
\setItemnumber{4}\item $T_{\nu}(q) = T_{\nu'}(q)$ for all $q \in \R$. \label{itemlq}
	\setItemnumber{5}\item $I(t) = I'(t-\gamma \log(M'/M))$ for all $t \in (\underline{t},\overline{t})$; \label{item4}
	\setItemnumber{6}\item $M_0=M'_0$, furthermore, $N_{\ih}/N'_{\ih}=(R'_{\ih}/R_{\ih})^{\gamma}= (M'/M)^{\gamma}$ for all $\ih=1,\dotsc,M_0$. \label{itemnow5}
\end{enumerate}   
\end{theorem}

We make several comments about Theorem~\ref{thm:multifractal}. 
\begin{remark}\label{r:multifraccomments}
\begin{enumerate}[(i)]

\item\label{i:citerao}
For carpets defined on the same grid, the equivalence of~\eqref{itemnow2} and the explicit condition~\eqref{itemnow5} was proved by Rao, Yang and Zhang in~\cite[Theorem~1.2]{RaoPreprintlipschitz}, using~\cite{Jordan2011bm}. 

\item\label{i:multifractalratelink}
In Step~4 of the proof of Proposition~\ref{prop:1} in Section~\ref{sec:proofProp}, we use scaling properties of Legendre transforms to establish a direct link between the multifractal spectrum of the uniform Bernoulli measure and the rate function $I(t)$. 
Since $I(t)$ appears in the intermediate dimension formula, this indicates why the link between the intermediate dimensions and multifractal spectrum in Theorem~\ref{thm:multifractal} is to be expected. 

\item\label{i:raosamem}
For carpets defined on the same grid with $M=M'$,~\eqref{itemnow5} is simply saying that the column sequence of one carpet is a permutation of the column sequence of the other.

\item\label{i:coarsespec}
Equality of the $L^q$ dimensions and the coarse multifractal spectra can be added to the above equivalences in Theorem~\ref{thm:multifractal} if the carpets are defined on the same grid, since these quantities can be obtained from the $L^q$ spectrum by dividing by $1-q$ or taking the Legendre transform respectively. 

\item\label{i:otherdimmeas}
If~\eqref{itemlq} holds then other notions of dimension of $\nu$ and $\nu'$ which can be deduced from their $L^q$ spectra must be equal, such as exact (Hausdorff/packing/entropy) dimension, correlation dimension (R\'enyi entropy), and Frostman dimension. 

\item\label{i:otherdimsets}
The formulae in~\cite{Fraser2021bedford} and~\eqref{itemnow5} can be used to show that equality of intermediate dimensions implies equality of other notions of dimensions of sets such as packing, Assouad, quasi-Assouad, lower, quasi-lower or modified lower dimensions, or the Assouad spectrum or lower spectrum for any fixed $\theta \in (0,1)$. 
\end{enumerate}
\end{remark}

Since the multifractal spectrum is analytic (as the Legendre transform of an analytic function), if $I \subseteq (\alpha_{\text{min}},\alpha_{\text{max}})$ is an open interval, then~\eqref{itemnow2} holds for all $\alpha \in I$ if and only if it holds for all $\alpha \in (\alpha_{\text{min}},\alpha_{\text{max}})$. 
Similarly, if $J \subseteq (\underline{t},\overline{t})$ is an open interval then~\eqref{item4} holds for all $t \in J$ if and only if it holds for all $t \in (\underline{t},\overline{t})$. 
\begin{question}
In the statement of Theorem~\ref{thm:multifractal}, can $(a,b)$ be taken to be an arbitrary open subinterval of $(0,1)$? 
\end{question}

Since the proof strategy of Lemma~\ref{lem:multifractallemma} does not seem to work under the assumption that the $L^q$ spectra are equal, we ask the following question. 

\begin{question}
Do there exist two Bedford--McMullen carpets with non-uniform fibres which cannot be realised on the same grid but whose uniform Bernoulli measures have the same $L^q$ spectra? 
\end{question}

Turning now to bi-Lipschitz equivalence, recall that two metric spaces $(X,d_X)$ and $(Y,d_Y)$ are \emph{bi-Lipschitz equivalent} if there is a bi-Lipschitz map $f\colon X\to Y$, i.e. one for which there exists $C \geq 1$ such that $C^{-1}d_X(x,y) \leq d_Y(f(x),f(y)) \leq Cd_X(x,y)$ for all $x,y \in X$. In our setting $X$ and $Y$ are two Bedford--McMullen carpets with the Euclidean distance. The following open problem seems challenging: 
\begin{question}\label{ques:bilip}
Find an explicit necessary and sufficient condition that determines, given two iterated function systems each generating a Bedford--McMullen carpet, whether or not the two carpets are bi-Lipschitz equivalent. 
\end{question}%
Partial progress towards Question~\ref{ques:bilip} has been made in~\cite{Li2013lipschitz,RaoPreprintlipschitz,
Yang2020lipschitz}, all of which assume some disconnectivity property. Fraser and Yu~\cite{Fraser2018secondassouad} used the Assouad spectrum to show that $\gamma$ is a bi-Lipschitz invariant within the class of Bedford--McMullen carpets, a fact which is also evident from observing the form of the intermediate dimensions. Moreover, the \emph{gap sequence} of a set is a topological quantity which has been shown to be bi-Lipschitz invariant \cite{Rao2008lipschitz}, and which is known for Bedford--McMullen carpets \cite{Miao2017gapseq,Liang2022gapseq}. Using the fact that the intermediate dimensions are stable under bi-Lipschitz maps, we obtain the following necessary condition for bi-Lipschitz equivalence as an immediate corollary of Theorem~\ref{thm:multifractal}. 

\begin{corollary}\label{cor:biLip}
Let $\Lambda$ and $\Lambda'$ be two Bedford--McMullen carpets with non-uniform fibres which are bi-Lipschitz equivalent, and let $\nu$ and $\nu'$ be the corresponding uniform Bernoulli measures. Then $f_{\nu}(\alpha)=f_{\nu'}(\alpha)$ for all $\alpha \in (\alpha_{\text{min}},\alpha_{\text{max}})$ and $T_{\nu}(q) = T_{\nu'}(q)$ for all $q \in \R$, and both carpets can be defined on the same $m \times n$ grid, on which condition~\eqref{itemnow5} above holds. 
\end{corollary} 
This strengthens~\cite[Corollary~1.1]{RaoPreprintlipschitz}, where it is assumed that $\Lambda$ and $\Lambda'$ are totally disconnected and defined on the same grid. 
In Example~\ref{ex:biLip}, we construct two carpets which we know are not bi-Lipschitz equivalent by Corollary~\ref{cor:biLip}, but where \cite[Corollary~1.1]{RaoPreprintlipschitz} does not apply. 
Corollary~\ref{cor:biLip} also shows in particular that if two carpets defined on the same grid with the same number of non-empty columns are bi-Lipschitz equivalent then the column sequence of one must be a permutation of the column sequence of the other (though we are not able to draw this conclusion if the number of non-empty columns is different, see Example~\ref{ex:rao}). 

One natural question would be to investigate the intermediate dimensions of self-affine carpets of Lalley and Gatzouras~\cite{Lalley1992gatzouras} or Bara\'nski~\cite{Baranski2007carpet}, or higher-dimensional self-affine sponges. 
Indeed, in light of the recent paper~\cite{BanajiPreprintgl} proving that the Assouad spectrum of Gatzouras--Lalley carpets (unlike Bedford--McMullen carpets) can be a differentiable function of~$\theta$, it is natural to ask whether the intermediate dimensions of Gatzouras--Lalley carpets can also be differentiable on $(0,1)$. 
We expect calculating a formula for the intermediate dimensions of such self-affine sets to be challenging, not least because there is no clear single analogue of the important quantity~$\gamma$, and this is not something which we will explore. 
We remark, however, that after we posted this paper on arXiv, Huang \emph{et al.}~\cite{HuangPreprintboxcounting} introduced so-called box-counting measures of metric spaces and showed that these classes of self-affine carpets support such measures (indeed for Bedford--McMullen carpets they are simply the uniform Bernoulli measures). 
The authors proved (without any connectivity assumption and without using the intermediate dimensions) that the multifractal spectrum of box-counting measures is a bi-Lipschitz invariant. 
Their result therefore generalises both the result from~\cite{RaoPreprintlipschitz} and our result on the bi-Lipschitz invariance of the multifractal spectrum. 
Motivated by Theorem~\ref{thm:multifractal}, the authors also ask in \cite[Open problem~1]{HuangPreprintboxcounting} whether two Gatzouras--Lalley or Bara\'nski carpets have the same intermediate dimensions if and only if their corresponding box-counting measures have the same multifractal spectra. 

\subsection{Equivalent forms of the rate function}\label{subsec:eqformsofrate}%

In this section we provide equivalent formulations of the rate function $I(t)$ in terms of a pressure-like function, a certain probability vector with an entropy maximising property, and the multifractal spectra $f_{\nu}(\alpha)$ defined in~\eqref{eq:uniformmultifractal}. As a result, our main formula~\eqref{eq:mainformula} for $\dim_{\, \theta}\Lambda$ can be expressed with any of these quantities. 

We begin by defining the pressure-like function. For $\iiv=(i_1,\dotsc,i_J)\in\{1,\dotsc,M\}^J$ and $k\in\{0,1,\dotsc,J\}$, we introduce
\begin{equation}\label{eq:20}
	\psi_{\iiv|k}(s) \coloneqq M^{k \gamma}\cdot n^{-sk}\cdot \prod_{\ell=1}^k N_{i_\ell}.
\end{equation}
In particular, for $k=0$, $\psi_{\iiv|0}(s)\equiv 1$. The interpretation of $\psi_{\iiv|k}(s)$ later is that it gives the $s$-cost of a set in the cover with diameter related to $k$, see Remark~\ref{rem:pressureexp}. Moreover, we define the sum
\begin{equation*}
	\Psi_J(s)\coloneqq \sum_{\iiv\in\{1,\dotsc,M\}^J}\min_{k\in\{0,1,\dotsc,J\}} \psi_{\iiv|k}(s).
\end{equation*}
This is connected to the total $s$-cost of the optimal cover, see Remark~\ref{rem:pressureexp} for additional explanation. To determine the critical exponent it is natural to define a pressure-like quantity as the exponential growth rate of $\Psi_J(s)$, more formally,
\begin{equation}\label{eq:21}
	\underline{P}(s)\coloneqq \liminf_{J\to\infty}\frac{1}{J}\log \Psi_J(s) \;\text{ and }\; \overline{P}(s)\coloneqq \limsup_{J\to\infty}\frac{1}{J}\log \Psi_J(s).
\end{equation} 

The probability vector $\mathbf{Q}^*_{t}\in\mathcal{P}_M$ is defined by
\begin{equation}\label{eq:104}
	H(\mathbf{Q}^*_{t}) = \sup\Big\{ \, H(\mathbf{p}): \mathbf{p}\in\mathcal{P}_M \;\text{ such that }\; \sum_{\jh=1}^M p_{\jh}\log N_{\jh} = t \, \Big\}.
\end{equation}
It is well defined, see Lemma~\ref{lem:32}. Moreover, $H(\mathbf{Q}^*_{t})<\log M$ since $t>\underline{t}$.

We regularly relate the arguments $s$ and $t$ to each other via the transformation
\begin{equation}\label{eq:23}
	t= t_1(s) = \left(s-\frac{\log M}{\log m}\right) \log n, \;\text{ or equivalently } s=\frac{t}{\log n} + \frac{\log M}{\log m} .
\end{equation}
We do so to ensure that 
\begin{equation}\label{eq:309}
\psi_{\iiv|J}(s)\leq 1 \;\Longleftrightarrow\; \frac{1}{J} \sum_{\ell=1}^{J} \log N_{i_\ell} \leq \Big(s-\frac{\log M}{\log m}\Big)\cdot \log n =t,
\end{equation}
which now follows from~\eqref{eq:20} and straightforward algebraic manipulations. 
Now, using~\eqref{eq:102} and~\eqref{eq:103}, $\underline{t}$ maps to
\begin{equation}\label{e:sunderlinedef}
	\underline{s}\coloneqq \dim_{\mathrm H}\Lambda - \frac{1}{\log n} \left( \frac{\log n}{\log m} \log \Bigg(  \frac{1}{M}\sum_{\jh=1}^{M} N_{\jh}^{\frac{\log m}{\log n}} \Bigg) - \underline{t} \right),
\end{equation}
while $\overline{t}$ maps to
\begin{equation}\label{e:soverlinedef}
	\overline{s} \coloneqq \dim_{\mathrm B}\Lambda + \frac{\log M - H(\mathbf{P})}{\log n}.
\end{equation}
Observe that by Jensen's inequality and non-uniform fibres, $\underline{s}<\dim_{\mathrm H}\Lambda<\dim_{\mathrm B}\Lambda<\overline{s}$. %
In Proposition~\ref{prop:1}, the key technical result of Section~\ref{subsec:eqformsofrate}, we make a clear connection between~\eqref{eq:21}, \eqref{eq:104}, \eqref{eq:22} and~\eqref{eq:uniformmultifractal} for pairs of $(s,t)$ related by~\eqref{eq:23}. 
The proof is non-trivial and is given in Section~\ref{sec:proofProp}. 

\begin{prop}\label{prop:1}
	Fix $s\in(\underline{s},\overline{s})$. Then $\underline{P}(s)=\overline{P}(s)$; let $P(s)$ denote this common value. Furthermore, for every pair $(s,t)$ related by~\eqref{eq:23},
	\begin{equation*}
		\log M - I(t) = P(s) = H(\mathbf{Q}^*_{t}) =  (\log m ) f_{\nu} \left(\frac{\log N}{\log m} - \left(\frac{1}{\log m} - \frac{1}{\log n}\right) t   \right)  - \frac{t}{\gamma}. 
	\end{equation*}
\end{prop}
Note also that by~\eqref{e:lqformula} and standard properties of Legendre transforms, $f_{\nu}$ and $I$ can be written in terms of the Legendre transform of $T_{\nu}$. 

\subsection{Illustrative examples}\label{subsec:ex}%

The simplest example with non-uniform fibres is the one shown in Figure~\ref{fig:BMPicture}. The following examples show additional interesting behaviour. 

\begin{remark}
All figures of the graphs in this paper were created using \emph{Wolfram Mathematica 12.3}, keeping simple implementation in mind rather than efficiency. 
For a fixed $\theta\in(0,1)$, the value of $\dim_{\, \theta}\Lambda$ was approximated by taking $2^{25}$ equally spaced points in the interval $(\dim_{\mathrm{H}}\Lambda,\dim_{\mathrm B}\Lambda)$ and choosing the point $s(\theta)$ for which the expression in~\eqref{eq:mainformula} was closest to~$0$.
\end{remark}

\begin{remark}\label{rem:snake}
It was first observed in~\cite{Kolossvary2022bm} that the graph $\theta\mapsto \dim_{\, \theta}\Lambda$ can approach $\dim_{\mathrm B}\Lambda$ from below the straight line $\ell(\theta)=\dim_{\mathrm H}\Lambda +\theta(\dim_{\mathrm B}\Lambda-\dim_{\mathrm H}\Lambda)$, indicating that it is possible for $\dim_{\, \theta}\Lambda$ not to be concave on the whole range of $\theta$. 
From Corollary~\ref{cor:allprop} it follows that in this case the graph $\theta\mapsto \dim_{\, \theta}\Lambda$ must intersect $\ell(\theta)$. 
In fact, there are even carpets where the graph intersects $\ell(\theta)$ twice, as shown on the left of Figure~\ref{fig:exSeries}. 
For the carpets in this figure, all parameters remain the same except for $m$, which causes different behaviour for larger values of $\theta$ as it changes. 
For $m\leq 25$, the graph stays above $\ell(\theta)$ for all $\theta$. 
\end{remark}

\begin{figure}[ht]
	\centering
	\includegraphics[width=0.99\textwidth]{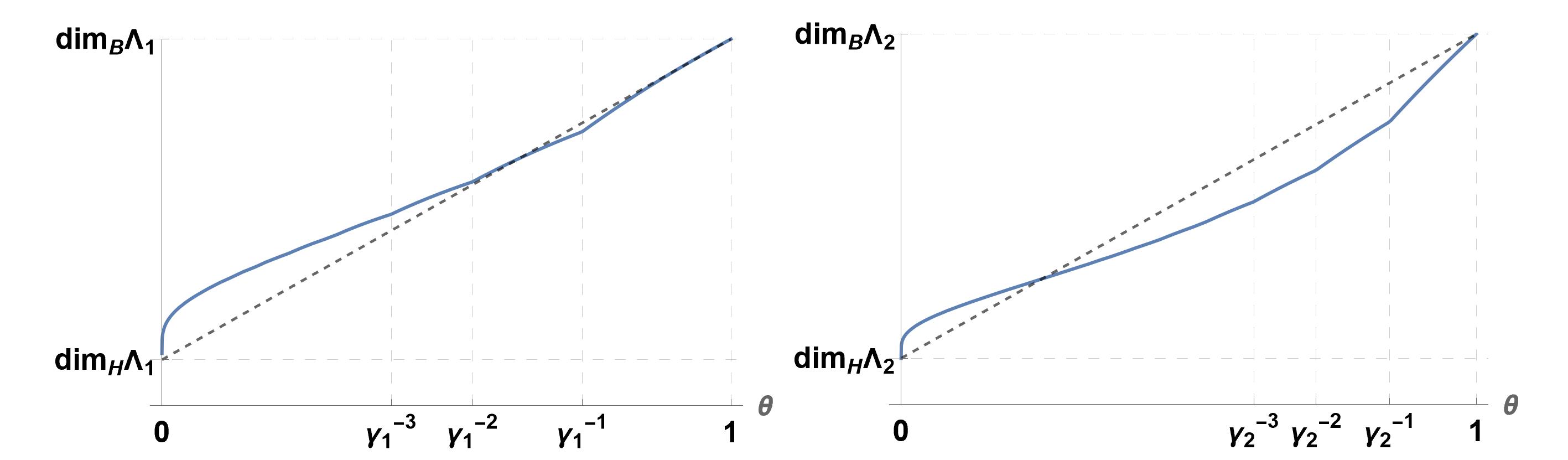}
	\caption{Parameters $n=100$ and $\mathbf{N}=(51,50,50,50,50,50)$ are the same in each example; only $m$ varies from $30$ on the left to $50$ on the right.}
	\label{fig:exSeries}
\end{figure}

Fraser and Yu \cite[Proposition~3.4]{Fraser2018secondassouad} gave two Bedford--McMullen carpets which have different Assouad spectra but for which some other notions of dimension coincide, and we have the following example. 
\begin{example}\label{ex:biLip}
Consider two Bedford--McMullen carpets %
$\Lambda$ and $\Lambda'$ with $m=M=32$ and $n=243$ and the following parameters: 
\begin{align*}
\Lambda &\colon \qquad M_0=3, \quad \{N_1,N_2,N_3\} = \{27,3,1\} \text{ and } \{R_1,R_2,R_3\} = \{2,11,19\}, \\
\Lambda' &\colon  \qquad M'_0=3, \quad \{N'_1,N'_2,N'_3\} = \{27,9,1\} \text{ and } \{R'_1,R'_2,R'_3\} = \{1,6,25\}.
\end{align*}
Then 
\begin{itemize}
	\item There exists $\theta \in (0,1)$ with $\dim_{\, \theta} \Lambda \neq \dim_{\, \theta} \Lambda'$ (by part~\eqref{itemnow5} of Theorem~\ref{thm:multifractal}, since all $R_{\ih}/R'_{\ih}$ are different) 
	\item We have $\dim \Lambda = \dim \Lambda'$ where $\dim$ can be any of the Hausdorff, box, packing, Assouad, quasi-Assouad, lower, quasi-lower or modified lower dimensions, or the Assouad spectrum or lower spectrum for any fixed $\theta \in (0,1)$. This holds since $N = 106$, $\max_{1 \leq i \leq 3} N_i =\max_{1 \leq i \leq 3} N'_i= 27$, $\min_{1\leq i \leq M} N_i = \min_{1\leq i \leq M} N'_i = 1$, so we can use~\eqref{eq:102} and \cite[Corollary~15.5.3]{Fraser2021bedford} to show that the Hausdorff and modified lower dimensions are equal, and the formulae in~\cite{Fraser2021bedford} for the other dimensions. 
\end{itemize}
Therefore $\Lambda$ and $\Lambda'$ are not bi-Lipschitz equivalent but this is revealed only by the intermediate dimensions, not by any of the other dimensions mentioned above. 
If all the rectangles are chosen in a specific row, then neither $\Lambda$ nor $\Lambda'$ is totally disconnected, so~\cite[Corollary~1.1]{RaoPreprintlipschitz} does not apply. 
Figure~\ref{fig:exRatio} shows the plots of $\dim_{\, \theta}\Lambda$ (blue) and $\dim_{\, \theta}\Lambda'$ (orange) side-by-side on the left and the ratio $\dim_{\, \theta}\Lambda'/\dim_{\, \theta}\Lambda$ on the right.

We can use H\"older distortion to obtain a quantitative improvement of the assertion that $\Lambda$ and $\Lambda'$ are not bi-Lipschitz equivalent. Assume $f \colon \Lambda' \to \mathbb{R}^2$ is $\alpha$-H\"older with $f(\Lambda') \supseteq \Lambda$. Then the optimal value of $\theta$ to consider is $\theta = \gamma^{-2} = \left(\frac{\log 2}{\log 3}\right)^2 \approx 0.40$, see the right hand side of Figure~\ref{fig:exRatio}. Then by~\eqref{eq:generalholderint}, 
\[ \alpha \leq \frac{\dim_{\, \gamma^{-2}} \Lambda'}{\dim_{\, \gamma^{-2}} f(\Lambda')} \leq \frac{\dim_{\, \gamma^{-2}} \Lambda'}{\dim_{\, \gamma^{-2}} \Lambda} < 0.9995, \]
with the last inequality computed numerically using Theorem~\ref{thm:main}. 

\begin{figure}[ht]
	\centering
	\includegraphics[width=0.99\textwidth]{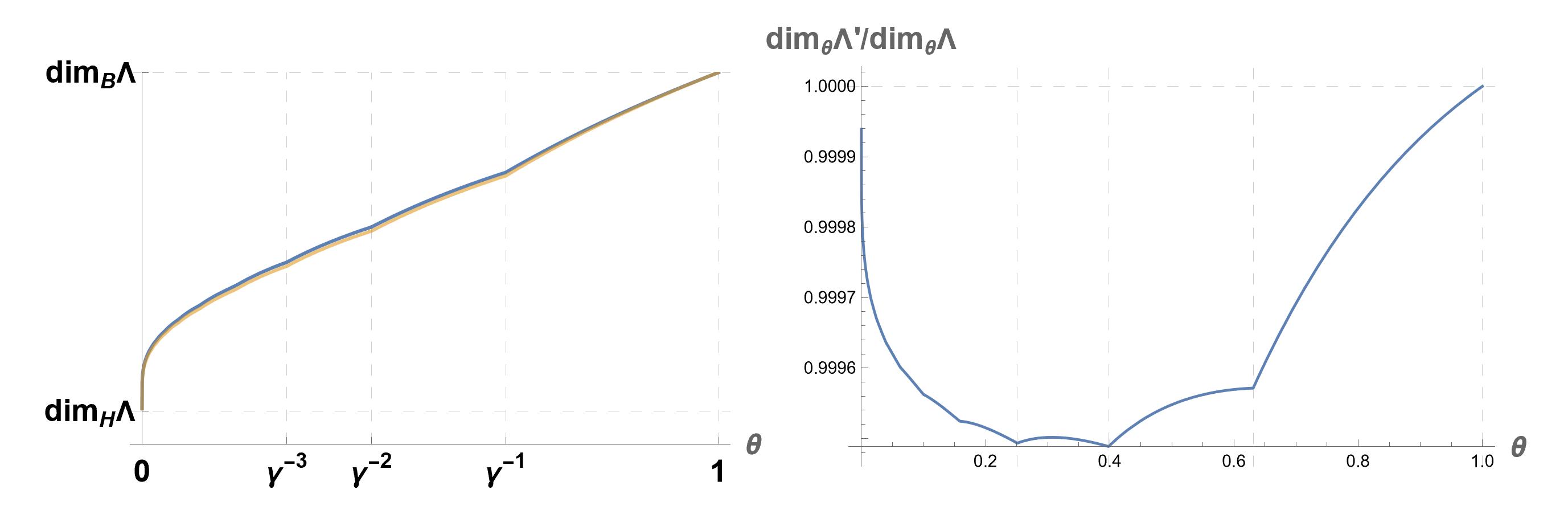}
	\caption{Left: plot of $\dim_{\, \theta}\Lambda$ (blue) and $\dim_{\, \theta}\Lambda'$ (orange) from Example~\ref{ex:biLip}. Right: ratio of $\dim_{\, \theta}\Lambda'/\dim_{\, \theta}\Lambda$ for $\theta\geq \gamma^{-35}$.}
	\label{fig:exRatio}
\end{figure}

\end{example}

\begin{prop}
For any carpet with just two column types, meaning that $M_0=2$ using notation from Section~\ref{subsec:multifractal}, the rate function can be given explicitly by $I(t)=\log M -H(\mathbf{Q}^*_t)$, where 
\[ H(\mathbf{Q}^*_t) = \frac{-1}{\log (N_1/N_2)}\left(\!\! (t-\log N_2)\log \frac{t-\log N_2}{R_1\log (N_1/N_2)}+ (\log N_1-t)\log \frac{\log N_1-t}{R_2\log (N_1/N_2)}  \!\right). \]
\end{prop}
\begin{proof}
Let $\mathbf{q}=(q_1,\dotsc,q_{M_0})\in\mathcal{P}_{M_0}$.
It is straightforward to see that the supremum in~\eqref{eq:104} will not change if we restrict to vectors $\mathbf{p}\in\mathcal{P}_M$ of the form 
\begin{equation}\label{eq:105}
p_{\ih} = q_j/R_j, \;\text{ if the } \ih\text{-th non-empty column has } N_j \text{ maps},
\end{equation}
in other words measure is distributed uniformly amongst columns with the same number of maps. 
As a result, the linear constraints in~\eqref{eq:104} can be rewritten as 
\begin{equation}\label{eq:106}
1 = \sum_{j=1}^{M_0} q_j  \,\;\text{ and }\;\, t =  \sum_{j=1}^{M_0} q_j \log N_j.
\end{equation}
In particular, since $M_0=2$, there is a single vector $(q^*_1,q^*_2)$ which satisfies~\eqref{eq:106}, namely 
\begin{equation*}
q^*_1= \frac{t-\log N_2}{\log( N_1/N_2)} \;\;\text{ and }\;\; q^*_2 =\frac{\log N_1 -t}{\log (N_1/N_2)},
\end{equation*}
recalling $N_1>N_2$. 
Using~\eqref{eq:105}, we can calculate the entropy of the entropy-maximising vector, 
\[ H(\mathbf{Q}^*_t) = - \sum_{j=1}^{M_0} q^*_j \log (q^*_j/R_j), \]
and conclude from Proposition~\ref{prop:1} that $I(t)=\log M -H(\mathbf{Q}^*_t)$, as required. 
\end{proof}

\begin{example}[using the parameters from Example~1.2 of Rao, Yang and Zhang~\cite{RaoPreprintlipschitz}]\label{ex:rao}
Consider two Bedford--McMullen carpets defined on the same grid with $m=8$, $n=27$ and the following parameters: 
\begin{align*}
 \Lambda: \qquad &M_0 = 2, \quad \{N_1,N_2\} = \{6,3\} \mbox{ and } \{R_1,R_2\} = \{1,1\}, \\*
\Lambda': \qquad &M_0' = 2, \quad \{N_1',N_2'\} = \{2,1\} \mbox{ and } \{R_1',R_2'\} = \{2,2\}. 
\end{align*}
Then condition~\eqref{itemnow5} from Theorem~\ref{thm:multifractal} holds, so $\dim_{\, \theta} \Lambda = \dim_{\, \theta} \Lambda'$ for all $\theta \in [0,1]$, despite the fact that the carpets are defined on the same grid with different parameters. This is only possible because the number of non-empty columns is different. 
\end{example}
It is shown in~\cite{RaoPreprintlipschitz} that the carpets in Example~\ref{ex:rao} are not bi-Lipschitz equivalent. Therefore equality of the intermediate dimensions is not a sufficient condition for two carpets with non-uniform fibres to be bi-Lipschitz equivalent, even if they are assumed to be defined on the same grid and totally disconnected. This raises the following question.  

\begin{question}
Suppose two Bedford--McMullen carpets both have non-uniform fibres, are defined on the same grid, and are bi-Lipschitz equivalent. Does it follow that both carpets must have identical parameters $(M_0,N_1,\dotsc,N_{M_0}, R_1,\dotsc, R_{M_0})$? 
\end{question}

\begin{example}\label{ex:Nice}%
Consider the two carpets $\Lambda$ and $\Lambda'$ with parameters 
\begin{align*}
\Lambda : \qquad &n=36, \quad m=6, \quad M = M_0 = 2,\quad \{ N_1,N_2\} = \{9,6\} \mbox{ and }\{R_1,R_2\} = \{1,1\} \\
\Lambda' : \qquad &n=36, \quad m=4,\quad M = M_0 = 2, \quad \{ N_1,N_2\} = \{6,4\} \mbox{ and }\{R_1,R_2\} = \{1,1\}.
\end{align*}
Then it can be checked from Theorem~\ref{thm:main} that $\dim_{\, \theta} \Lambda = \dim_{\, \theta} \Lambda'$ for all $\theta \in [1/2,1]$, but not for the whole range of $\theta$; by Theorem~\ref{thm:multifractal} this is only possible because the carpets are defined on different grids. By Corollary~\ref{cor:allprop}, the graph of $\dim_{\, \theta} \Lambda$ has a phase transition at $\theta = 1/2$ but the graph of $\dim_{\, \theta} \Lambda'$ does not, see Figure~\ref{fig:exNice}.
\end{example}
\begin{figure}[ht]
	\centering
	\includegraphics[width=0.5\textwidth]{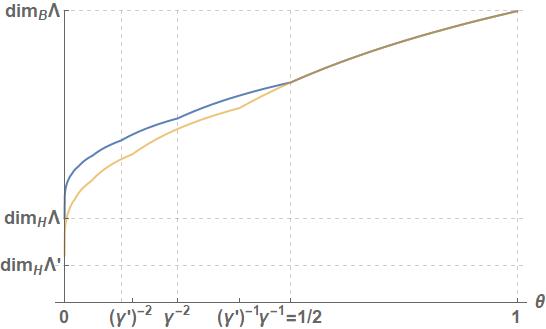}
	\caption{Plots of the intermediate dimensions of carpets in Example~\ref{ex:Nice}.}
	\label{fig:exNice}
\end{figure}

\section{Proof of Proposition \ref{prop:1}}\label{sec:proofProp}

Recall notation from Section~\ref{sec:CompResults}. Assume $s\in(\underline{s},\overline{s})$ and that $(s,t)$ are related by~\eqref{eq:23}. We prove Proposition~\ref{prop:1} (giving the equivalent forms of the rate function) in four steps:
\begin{description}
\item[Step 1] $H(\mathbf{Q}^*_{t}) = \log M - I(t)$ 
\item[Step 2] $\overline{P}(s) \leq H(\mathbf{Q}^*_{t})$ 
\item[Step 3] $H(\mathbf{Q}^*_{t}) \leq \underline{P}(s)$ 
\item[Step 4] $I(t) = -(\log m ) f_{\nu} \left(\frac{\log N}{\log m} - \left(\frac{1}{\log m} - \frac{1}{\log n}\right) t   \right)  + \frac{t}{\gamma} + \log M$ 
\end{description}
We will prove the steps in separate subsections, which will also contain some auxiliary results which are important in their own right in the proof of Theorem~\ref{thm:main}. 

\subsection{Preliminaries} 
First, we need some preliminaries, in particular to describe the probability vectors $\mathbf{P}^*_t$ and $\mathbf{Q}^*_t$ which have certain optimising properties, and to recall some facts from the method of types. 
If $k,k_1,k_2 \in \mathbb{N}$ satisfy $0 \leq k_1 < k_2 \leq k$, and $\iih \in \{1,\dotsc,M\}^k$, then we define the average 
\begin{equation}\label{eq:tauaverage} \tau(\iih,k_1,k_2) \coloneqq \max\Bigg\{ \underline{t} , \frac{1}{k_2 - k_1} \sum_{j=k_1 + 1}^{k_2} \log N_{\ih_j} \Bigg\}. 
\end{equation} 
Recall, $\mathcal{P}_M$ denotes the set of probability vectors on $[M]=\{1,\dotsc,M\}$ and we introduced two distinguished probability vectors $\mathbf{P}\coloneqq (N_1/N,\dotsc,N_M/N)$ and $\mathbf{Q} \coloneqq (1/M,\dotsc,1/M)$.
Recall that the \emph{Kullback--Leibler divergence}, also known as the \emph{relative entropy} of $\mathbf{p}\in\mathcal{P}_M$ with respect to $\mathbf{q}\in\mathcal{P}_M$ is
\begin{equation*}
H(\mathbf{p}\| \mathbf{q}) \coloneqq \sum_{\ih=1}^M p_{\ih} \log \left( \frac{p_{\ih}}{ q_{\ih}}\right) = - H(\mathbf{p}) - \sum_{\ih=1}^M p_{\ih} \log q_{\ih} ,
\end{equation*}
where we set $0 \log 0 = 0$ and $0 \log (0/q_{\ih}) = 0$ regardless of the value of $q_{\ih}$. 
It is asymmetric and $H(\mathbf{p}\| \mathbf{q})\geq 0$ with equality if and only if $\mathbf{p}=\mathbf{q}$. In particular, 
\begin{equation}\label{eq:32}
H(\mathbf{p}\| \mathbf{P}) = \log N -H(\mathbf{p}) -\sum_{\ih=1}^M p_{\ih}\log N_{\ih} \quad \text{ and } \quad H(\mathbf{p}\| \mathbf{Q}) = \log M - H(\mathbf{p}).
\end{equation}

Recall that $\underline{t}\coloneqq \frac{1}{M} \sum_{\jh=1}^{M} \log N_{\jh}$ and $\overline{t}\coloneqq \log N-H(\mathbf{P})$, and let $t\in(\underline{t},\overline{t})$. We divide the set $\mathcal{P}_M$ into two parts: 
\begin{equation}\label{eq:300}
\mathcal{G}_t \coloneqq \bigg\{ \, \mathbf{p}\in\mathcal{P}_M: \sum_{\ih=1}^M p_{\ih}\log N_{\ih} \leq t \, \bigg\} \;\text{ and }\; \mathcal{F}_t \coloneqq \bigg\{ \, \mathbf{p}\in\mathcal{P}_M: \sum_{\ih=1}^M p_{\ih}\log N_{\ih} \geq t \, \bigg\},
\end{equation}
so that $\mathcal{E}_t\coloneqq \mathcal{G}_t \cap \mathcal{F}_t = \big\{ \, \mathbf{p}\in\mathcal{P}_M: \sum_{\ih} p_{\ih}\log N_{\ih} = t \, \big\}$. 
The reason for doing this will become clear in~\eqref{eq:301} in the proof of Step~2. 
Since $t>\underline{t}$, it follows that $\mathbf{Q}\in\mathcal{G}_t\setminus\mathcal{E}_t$, whereas $\mathbf{P}\in\mathcal{F}_t\setminus\mathcal{E}_t$ because $t<\overline{t} = \sum_{\ih=1}^M P_{\ih}\log N_{\ih}$.

\begin{lemma}\label{lem:31}
Let $t\in(\underline{t},\overline{t})$ and $\mathbf{p}\in\mathcal{E}_t$. Then 
\begin{equation*}
H(\mathbf{p}\| \mathbf{P}) = H(\mathbf{p}\| \mathbf{Q})+\log(N/M) -t.
\end{equation*}
\end{lemma}

\begin{proof}
Let $\mathbf{p}\in\mathcal{E}_t$. Let $N'_1<N'_2< \dotsb <N'_{M_0}$ denote the different values that the set $\{N_1,\dotsc,N_M\}$ takes. %
For $1\leq j\leq M_0$ let $I_j\coloneqq \{\, \ih\in[M] :  N_{\ih} = N'_j \, \}$ and  $q_j\coloneqq \sum_{\ih\in I_j}p_{\ih}$. Then,
\begin{equation*}
1=\sum_{\ih=1}^M p_{\ih} = \sum_{j=1}^{M_0} q_j  \,\;\text{ and }\;\, t = \sum_{\ih=1}^M p_{\ih}\log N_{\ih} =  \sum_{j=1}^{M_0} q_j \log N'_j.
\end{equation*}
This is a linear system of equations for $\{q_j\}_{j=1}^{M_0}$. Straightforward Gaussian elimination yields
\begin{equation*}
\left[\begin{array}{@{}ccc|c@{}}
1  & \dots & 1 & 1 \\
\log N'_1 & \dots & \log N'_{M_0} & t
\end{array}\right] %
\;\sim\;
\left[\begin{array}{@{}ccccc|c@{}}
1 & 1 & 1 & \dots & 1 & 1 \\
0 & 1 & \frac{\log(N'_3/N'_1)}{\log(N'_2/N'_1)} & \dots & \frac{\log(N'_{M_0}/N'_1)}{\log(N'_2/N'_1)} & \frac{t- \log N'_1}{\log(N'_2/N'_1)} 
\end{array}\right].
\end{equation*}
Thus, for every solution $(q_1,\dotsc,q_{M_0})$, we see that $q_3,\dotsc,q_{M_0}$ are free variables, and moreover
\begin{equation*}
q_2 = \frac{1}{\log (N'_2 / N'_1)} \bigg( t - \sum_{j=3}^{M_0} q_j \log N'_j - \bigg(1-\sum_{j=3}^{M_0} q_j\bigg) \log N'_1\bigg)
\end{equation*}
and $q_1=1-q_2-\sum_{j=3}^{M_0} q_j$. It now follows by a straightforward calculation that $\sum_{j=1}^{M_0} q_j \log N'_j=t$. By~\eqref{eq:32}, the result follows. 
\end{proof}

We introduce $\mathbf{P}^*_t\in\mathcal{G}_t$, $\mathbf{Q}^*_t\in\mathcal{F}_t$ defined by
\begin{equation}\label{eq:defineqstar}
H(\mathbf{P}^*_t\| \mathbf{P}) = \inf_{\mathbf{p}\in\mathcal{G}_t} H(\mathbf{p}\| \mathbf{P}) \qquad \text{ and }\qquad H(\mathbf{Q}^*_t\| \mathbf{Q}) = \inf_{\mathbf{q}\in\mathcal{F}_t} H(\mathbf{q}\| \mathbf{Q}).
\end{equation}
Due to~\eqref{eq:32} and Lemma~\ref{lem:32}, %
 this definition of $\mathbf{Q}^*_t$ is equivalent to~\eqref{eq:104}.
\begin{lemma}\label{lem:32}
Let $t\in(\underline{t},\overline{t})$. Then both $\mathbf{P}^*_t$ and $\mathbf{Q}^*_t$ are well defined and unique, with $\mathbf{P}^*_t,\mathbf{Q}^*_t\in\mathcal{E}_t$. Moreover, 
\begin{equation}\label{eq:33}
H(\mathbf{P}^*_t\| \mathbf{P}) = H(\mathbf{Q}^*_t\| \mathbf{Q})+\log(N/M)-t.
\end{equation}
\end{lemma}
\begin{proof}
Both $H(\cdot\,\| \mathbf{P})$ and $H(\cdot\,\| \mathbf{Q})$ are continuous on the domain $\mathcal{P}_M$, and $\mathcal{G}_t$ and $\mathcal{F}_t$ are compact, so both infima are attained. 

We proceed to differentiate the relative entropy (with respect to a fixed vector $\mathbf{q}$ in the interior of $\mathcal{P}_M$) along straight lines in $\mathcal{P}_M$. 
Fix $\mathbf{p},\mathbf{q} \in \mathcal{P}_M$, and assume that all entries of $\mathbf{q}$ are positive. Fix $\mathbf{v} \in \R^M \setminus \{0\}$ satisfying $\sum_{\ih=1}^M v_{\ih} = 0$. 
Let $[-t_-,t_+]$ denote the maximal interval (containing the origin) such that $H(\mathbf{p} + t \mathbf{v} \| \mathbf{q} ) \in \mathcal{P}_M$ for all $t \in [-t_-,t_+]$. 
Then for all $t \in (-t_-,t_+)$, %
a direct computation gives 
\[ \frac{d}{dt} H(\mathbf{p} + t \mathbf{v} \| \mathbf{q} )  = \sum v_{\ih} \log \left( \frac{p_{\ih} + t v_{\ih}}{q_{\ih}} \right) ; \qquad \frac{d^2}{dt^2} H(\mathbf{p} + t \mathbf{v} \| \mathbf{q} ) = \sum \frac{v_{\ih}^2}{p_{\ih} + t v_{\ih}} > 0, \]
where the sums are taken over all indices $1 \leq \ih \leq M$ for which $p_{\ih} + t v_{\ih} > 0$ for all $t \in (-t_-,t_+)$. %
Therefore $t \mapsto H(\mathbf{p} + t \mathbf{v} \| \mathbf{q} )$ is strictly convex, and has at most one minimum in $[-t_-,t_+]$. 
In particular, the uniqueness of $\mathbf{P}^*_t$ and $\mathbf{Q}^*_t$ follows from the convexity of $\mathcal{G}_t$ and $\mathcal{F}_t$ respectively. 
Note also that for all $t \in (0,t_+)$,
\[ \frac{d}{dt} H(\mathbf{q} + t \mathbf{v} \| \mathbf{q} )  = \sum_{\ih=1}^M v_{\ih} \log \left(1 + \frac{v_{\ih}}{q_{\ih}}t \right) > 0,\] %
so $t \mapsto H(\mathbf{q} + t \mathbf{v} \| \mathbf{q} )$ is strictly increasing on $(0,t_+)$. 
Applying this with $\mathbf{q}$ taken to be $\mathbf{P}$ or $\mathbf{Q}$ respectively gives that $\mathbf{P}^*_t,\mathbf{Q}^*_t\in\mathcal{E}_t$.  

To conclude, Lemma~\ref{lem:31} gives that
\begin{align*}
H(\mathbf{Q}^*_t\| \mathbf{P}) = H(\mathbf{Q}^*_t\| \mathbf{Q}) +\log(N/M)-t \leq H(\mathbf{P}^*_t\| \mathbf{Q}) +\log(N/M)-t &= H(\mathbf{P}^*_t\| \mathbf{P}) \\
&\leq H(\mathbf{Q}^*_t\| \mathbf{P}),
\end{align*}
so there is equality throughout. 
\end{proof}

The importance of choosing $t$ to lie in the interval $(\underline{t},\overline{t})$ (or equivalently $s\in(\underline{s},\overline{s})$) is that in this case the hyperplane $\mathcal{E}_t$ separates $\mathbf{P}$ and $\mathbf{Q}$. Otherwise, either $H(\mathbf{P}^*_t\| \mathbf{P})=0$ or $H(\mathbf{Q}^*_t\| \mathbf{Q})=0$, and~\eqref{eq:33} does not necessarily hold.

\subsubsection*{Method of types} 
The method of types is an elementary tool developed to study discrete memoryless systems in information theory. It has since found applications in hypothesis testing, combinatorics and large deviations (see~\cite{Csiszar1998methodoftypes} for some background). Kolossv\'ary used it to calculate the box dimension and $L^q$ spectra of Gatzouras--Lalley and Bara\'nski type sponges in $\mathbb{R}^d$~\cite{Kolossvary2022lqtypes}. 

Let $[M]=\{1,\dotsc,M\}$ denote a finite alphabet and assume $\iiv=(i_1,\dotsc,i_J)\in[M]^J$. For $J'\leq J$, the \emph{type of $\iiv$ at level $J'$} is the empirical probability vector
\begin{equation*}%
\boldsymbol{\tau}_{J'}^{J}(\iiv) = \frac{1}{J'}\big( \#\{ \, 1\leq\ell\leq J': i_{\ell}=j \, \} \big)_{j\in[M]} \in [0,1]^M.
\end{equation*}
When $J'=J$, we simply write $\boldsymbol{\tau}_J(\iiv) \coloneqq \boldsymbol{\tau}_J^{J}(\iiv)$. The set of all possible types of $[M]^J$ is
\begin{equation*}
\mathcal{T}_J=\big\{ \, \mathbf{p}\in\mathcal{P}_M: \text{ there exists } \iiv\in[M]^J \text{ such that } \mathbf{p}=\boldsymbol{\tau}_J(\iiv) \, \big\},
\end{equation*}
and for $J'\leq J$, the \emph{type class of} $\mathbf{p}\in \mathcal{T}_{J'}$ amongst $[M]^J$ is the set
\begin{equation*}
T_{J'}^{J}(\mathbf{p}) = \big\{ \, \iiv\in[M]^J: \boldsymbol{\tau}_{J'}^J(\iiv)=\mathbf{p} \, \big\}.
\end{equation*}
Similarly, $T_{J}(\mathbf{p}) \coloneqq T_{J}^{J}(\mathbf{p})$. 

We use the following two simple facts:
\begin{equation}\label{eq:205}
\# \mathcal{T}_J \leq (J+1)^M
\end{equation}
and for every type class
\begin{equation}\label{eq:206}
\big(J+1\big)^{-M} e^{J\cdot H(\mathbf{p})} 
\leq \#T_{J}(\mathbf{p}) \leq e^{J\cdot H(\mathbf{p})},
\end{equation}
see~\cite[Lemmas~2.1.2 and~2.1.8]{Dembo2010largedeviations}. 
The importance of~\eqref{eq:205} is that $\# \mathcal{T}_J$ grows only polynomially in $J$; on the other hand, the exponential terms in~\eqref{eq:206} are the same in both the lower and upper bounds. Since we are looking for critical exponents, sub-exponential multiplicative terms do not influence our calculations. To simplify notation, we write $f(J)\asymp g(J)$ if the exponential rates of growth of $f(J)>0$ and $g(J)>0$ exist and are equal to each other, so 
\begin{equation*}
f(J)\asymp g(J) \;\Longleftrightarrow\; \lim_{J\to\infty}\frac{1}{J}\log f(J) = \lim_{J\to\infty}\frac{1}{J}\log g(J). 
\end{equation*}  
In particular, if $f(J)$ is sub-exponential in $J$ (for example when $f(J) = \# \mathcal{T}_J$), then $f(J)\asymp 1$. %

The set $\mathcal{T}_J$ is a discrete set with polynomially many points which becomes dense in $\mathcal{P}_M$ as $J\to\infty$. For $\mathbf{p}\in\mathcal{P}_M$ let $\mathbf{p}_{J}$ denote the `best approximation' of $\mathbf{p}$ in $\mathcal{T}_J$, in the sense that $\|\mathbf{p}-\mathbf{p}_{J}\| = \min_{\mathbf{q}\in\mathcal{T}_J} \|\mathbf{p}-\mathbf{q}\|$, where we can take any norm. If there are many such $\mathbf{p}_{J}$ then we can choose the one with smallest coordinates when ordered lexicographically. %
For large enough $J$, $\|\mathbf{p}-\mathbf{p}_{J}\|$ is arbitrarily small. In particular, property~\eqref{eq:206} and the continuity of the entropy imply that $\#T_{J}(\mathbf{p}_J) \asymp e^{J\cdot H(\mathbf{p})}$. 

\subsection{Proof of Step 1}

This is a standard argument in large deviations theory, and in fact holds for all $t \in (\underline{t},\max_{1\leq i \leq M} \log N_i)$. We include a sketch of it for the convenience of the reader. An alternative approach would be to use Lagrange multipliers.

Let $\mathbf{I}=I_1,I_2,\dotsc$ be an infinite sequence of independent and identically distributed random variables on the set $\{1,\dotsc,M\}$ according to $\mathbf{q}\in\mathcal{P}_M$. Let $\mathds{P}_{\mathbf{q}}\coloneqq \mathbf{q}^{\mathds{N}}$ denote the product measure corresponding to the distribution of the sequence $\mathbf{I}$. Then $\boldsymbol{\tau}_J(\mathbf{I})$, the type of $(I_1,\dotsc,I_J)$, is a vector-valued random variable. For all $\mathbf{p}\in\mathcal{T}_J$,
\begin{equation*}
(J+1)^{-M} e^{-J\cdot H(\mathbf{p}\|\mathbf{q})} 
\leq \mathds{P}_{\mathbf{q}}(\boldsymbol{\tau}_J(\mathbf{I})=\mathbf{p}) \leq e^{-J\cdot H(\mathbf{p}\|\mathbf{q})},
\end{equation*}
see~\cite[Lemma~2.1.9]{Dembo2010largedeviations}. Sanov's theorem~\cite[Theorem~2.1.10]{Dembo2010largedeviations} shows that the family of laws $\mathds{P}_{\mathbf{q}} (\boldsymbol{\tau}_J(\mathbf{I})\in \cdot )$ satisfies a large deviations principle with the rate function $H(\cdot\|\mathbf{q})$. In particular, for $\mathbf{q}=\mathbf{Q}=(1/M,\dotsc,1/M)$ and the subset $\mathcal{F}_{t}$:
\begin{equation*}
\mathds{P}_{\mathbf{Q}} (\boldsymbol{\tau}_J(\mathbf{I})\in\mathcal{F}_{t} ) \asymp e^{-J \inf_{\mathbf{q}\in\mathcal{F}_t} H(\mathbf{q}\|\mathbf{Q})} = e^{-J H(\mathbf{Q}^*_{t}\|\mathbf{Q})}. 
\end{equation*}
Now define the random variable $X_\ell\coloneqq \log N_{I_\ell}$ and the averages $Y_J\coloneqq \frac{1}{J}\sum_{\ell=1}^JX_{\ell}$. Then
\begin{equation*}
Y_J = \sum_{\ih=1}^M \tau_{J,\ih}(\mathbf{I}) \cdot \log N_{\ih}
\end{equation*}
is a continuous function of $\boldsymbol{\tau}_{J}(\mathbf{I})$. Hence, by the `contraction principle'~\cite[Section~4.2.1]{Dembo2010largedeviations}, the rate function $I(t)$ of $\mathds{P}_{\mathbf{Q}} (Y_J\in \cdot )$ is equal to
\begin{equation*}
I(t) = \inf \Big\{ \, H(\mathbf{q}\|\mathbf{Q}): \sum_{\ih=1}^M q_{\ih} \log N_{\ih} = t \, \Big\}.
\end{equation*}
In particular, Lemma~\ref{lem:32} implies that
\begin{equation*}
I(t) = \inf_{\mathbf{q}\in\mathcal{E}_t} H(\mathbf{q}\|\mathbf{Q}) \stackrel{\eqref{eq:defineqstar}}{=} H(\mathbf{Q}^*_{t}\|\mathbf{Q}) \stackrel{\eqref{eq:32}}{=} \log M - H(\mathbf{Q}^*_{t}).
\end{equation*}

\subsection{Proof of Step 2}

Since $\min_{k\in\{0,1,\dotsc,J\}} \psi_{\iiv|k}(s)\leq \min\{1,\psi_{\iiv|J}(s)\}$, recall the definition of $\psi_{\iiv|k}(s)$ from~\eqref{eq:20},
\begin{equation}\label{eq:302}
\Psi_J(s) \leq \#\big\{ \, \iiv\in[M]^J: 1<\psi_{\iiv|J}(s) \, \big\} + \sum_{\iiv\in[M]^J:\, \psi_{\iiv|J}(s)\leq 1} \psi_{\iiv|J}(s).
\end{equation} 

\begin{lemma}\label{lem:33-first}
Assume $t \in (\underline{t},\max_{1\leq i \leq N} \log N_i)$ and that $(s,t)$ are related by~\eqref{eq:23}. Then
\begin{equation*}
\#\big\{ \, \iiv\in[M]^J: 1<\psi_{\iiv|J}(s) \, \big\} \asymp e^{J\cdot H(\mathbf{Q}^*_{t})}.
\end{equation*}
\end{lemma}
\begin{proof}
For any given word $\iiv\in[M]^J$, the average of the $\log N_{i_\ell}$ does not depend on the particular order of the symbols in $\iiv$, just on the relative frequency of each symbol. In other words, only the type $\boldsymbol{\tau}_J(\iiv)=(\tau_{J,1}(\iiv),\dotsc,\tau_{J,M}(\iiv))$ of $\iiv$ matters, and (recalling~\eqref{eq:309})
\begin{equation*}
\psi_{\iiv|J}(s)\leq 1 \;\Longleftrightarrow\;  \sum_{\ih=1}^{M} \tau_{J,\ih}(\iiv) \log N_{\ih} \leq t.
\end{equation*}
This reduces the problem back to a condition on probability vectors $\mathbf{p}\in\mathcal{P}_M$. This is the reason why we introduced $\mathcal{G}_t$ and $\mathcal{F}_t$ the way we did in~\eqref{eq:300}; we now see that
\begin{equation}\label{eq:301}
\psi_{\iiv|J}(s)\leq 1 \;\Longleftrightarrow\; \boldsymbol{\tau}_J(\iiv)\in \mathcal{G}_t.
\end{equation}

We are now ready to determine the exponential rate of growth of the two terms in~\eqref{eq:302} separately by grouping together words according to type class.
Let $\mathbf{Q}^*_{t,J}\in(\mathcal{F}_{t}\cap \mathcal{T}_J)$ be the type for which $H(\mathbf{Q}^*_{t,J}) = \max_{\mathbf{q}\in(\mathcal{F}_{t}\cap \mathcal{T}_J)} H(\mathbf{q})$ (if there is more than one such vector then we can choose the smallest lexicographically). Then 
\begin{equation*}
(J+1)^{-M}e^{J\cdot H(\mathbf{Q}^*_{t,J})} \leq \#\big\{ \, \iiv\in[M]^J: 1<\psi_{\iiv|J}(s) \, \big\} = \sum_{\mathbf{q}\in(\mathcal{F}_{t}\cap \mathcal{T}_J)} \# T_{J}(\mathbf{q})  \leq \#\mathcal{T}_J\cdot e^{J\cdot H(\mathbf{Q}^*_{t,J})}, 
\end{equation*}
where we used~\eqref{eq:206} for the two inequalities. As $J\to\infty$, the set $\mathcal{T}_J$ becomes dense in $\mathcal{P}_M$, and as a result $\|\mathbf{Q}^*_{t,J}-\mathbf{Q}^*_{t}\|\to 0$ so $H(\mathbf{Q}^*_{t,J})\to H(\mathbf{Q}^*_{t})$. 
Hence, it follows from~\eqref{eq:205} that $\#\big\{ \, \iiv\in[M]^J: 1<\psi_{\iiv|J}(s) \, \big\} \asymp e^{J\cdot H(\mathbf{Q}^*_{t})}$. 

An alternative way to see this would be to let $X_1,X_2,\dotsc,X_J,\dotsc $ be a sequence of i.i.d. random variables defined by~\eqref{e:definerandomvar}. Then 
\[ \#\big\{ \, \iiv\in[M]^J: 1<\psi_{\iiv|J}(s) \, \big\} = M^J \mathds{P}\left( \sum_{i = 1}^J X_i > tJ \right) \asymp M^J e^{-J \cdot I(t)} = e^{J\cdot H(\mathbf{Q}^*_{t})}, \]
where we used~\eqref{eq:309} for the first equality, Cram\'er's theorem from large deviations theory for the asymptotic equality, and Step~1 for the final equality. 
\end{proof}

\begin{lemma}\label{lem:33-second}
Assume $s\in(\underline{s},\overline{s})$ and that $(s,t)$ are related by~\eqref{eq:23}. Then
\begin{equation*}
 \sum_{\iiv\in[M]^J:\, \psi_{\iiv|J}(s)\leq 1} \psi_{\iiv|J}(s) \asymp e^{J\cdot H(\mathbf{Q}^*_{t})}.
\end{equation*}
\end{lemma}

\begin{proof}
If $\iiv\in T_J(\mathbf{p})$, then
\begin{equation}\label{eq:308}
\psi_{\iiv|J}(s) = M^{J \log_m n} n^{-sJ}\cdot \prod_{\ih=1}^M N_{\ih}^{p_{\ih}J} = e^{J\left( \left(\frac{\log M}{\log m}-s\right)\cdot \log n  + \sum_{\ih} p_{\ih}\log N_{\ih} \right)} = e^{J( -t+ \sum_{\ih} p_{\ih}\log N_{\ih} )}.
\end{equation}
Using~\eqref{eq:206} and~\eqref{eq:32}, 
\begin{equation*}
\#T_J(\mathbf{p})\cdot \psi_{\iiv|J}(s) \asymp e^{J( -t + \sum_{\ih} p_{\ih}\log N_{\ih}+H(\mathbf{p})  )} 
= e^{J\left( -t + \log N - H(\mathbf{p}\| \mathbf{P})\right)}.
\end{equation*} 
Similarly to $\mathbf{Q}^*_{t,J}$, let $\mathbf{P}^*_{t,J}\in(\mathcal{G}_{t}\cap \mathcal{T}_J)$ satisfy $H(\mathbf{P}^*_{t,J}\| \mathbf{P}) = \min_{\mathbf{p}\in(\mathcal{G}_{t}\cap \mathcal{T}_J)} H(\mathbf{p}\|\mathbf{P})$. We have $H(\mathbf{P}^*_{t,J}\| \mathbf{P})\to H(\mathbf{P}^*_{t}\| \mathbf{P})$ as $J\to\infty$. Then
\begin{equation*}
\sum_{\iiv\in[M]^J:\, \psi_{\iiv|J}(s)\leq 1}  \psi_{\iiv|J}(s) = 
\sum_{\mathbf{p}\in(\mathcal{G}_{t}\cap \mathcal{T}_J)} \sum_{\iiv\in T_J(\mathbf{p})} \psi_{\iiv|J}(s) 
\asymp e^{J\left( -t + \log N - H(\mathbf{P}^*_{t,J}\| \mathbf{P})\right)}.
\end{equation*}
Using~\eqref{eq:33} and~\eqref{eq:32} in the exponent, $-t + \log N - H(\mathbf{P}^*_{t}\| \mathbf{P}) = \log M-H(\mathbf{Q}^*_{t}\| \mathbf{Q}) = H(\mathbf{Q}^*_{t})$, and the assertion follows. 
\end{proof}

Note the importance of the assumption $t < \overline{t}$ in the proof of Lemma~\ref{lem:33-second}. 
Lemmas~\ref{lem:33-first} and~\ref{lem:33-second} and~\eqref{eq:302} immediately imply that $\overline{P}(s) \leq H(\mathbf{Q}^*_{t})$. 

\subsection{Proof of Step 3}

This in fact holds for all $t \in (\underline{t},\max_{1\leq i \leq N} \log N_i)$. Assume $(s,t)$ are related by~\eqref{eq:23}. Fix $R \in \mathbb{N}$. For $J \in \mathbb{N}$, if $l \in \{0,1,\dotsc,R-1\}$, let $J_{l,R} \coloneqq  \lfloor (l+1) J / R \rfloor - \lfloor l J / R \rfloor$. We introduce 
\begin{align}\label{eq:definestjr}
\begin{split}
 S_{t,J,R} \coloneqq \big\{\, \iiv = (i_1,\dotsc,i_J) \in [M]^J : &(i_{\lfloor l J / R \rfloor + 1} , \dotsc, i_{\lfloor (l+1) J / R \rfloor }) \in T_{J_{l,R}}(\mathbf{Q}^*_{t,J_{l,R}}) \\*
&\mbox{ for all } l \in \{0,1,\dotsc,R-1\} \, \big\}.
\end{split}
 \end{align}
Then 
\begin{equation}\label{eq:lowerboundcard}
 \#  S_{t,J,R} = \prod_{l=0}^{R-1} \# T_{J_{l,R}}(\mathbf{Q}^*_{t,J_{l,R}}) \stackrel{\eqref{eq:206}}{\asymp} \prod_{l=0}^{R-1} e^{J_{l,R} \cdot H(\mathbf{Q}^*_{t,J_{l,R}})} \asymp e^{J \cdot H(\mathbf{Q}^*_{t})} = e^{J(\log M - I(t))}
 \end{equation}
as $J \to \infty$, using Step~1 of Proposition~\ref{prop:1} in the last step. %

Suppose $(i_1,\dotsc,i_J) \in S_{t,J,R}$. For all $l \in \{0,1,\dotsc,R-1\}$, $\mathbf{Q}^*_{t,J_{l,R}} \to \mathbf{Q}^*_t \in \mathcal{E}_t$, so by~\eqref{eq:309}, 
\[ \psi_{(i_{\lfloor l J / R \rfloor + 1} , \dotsc, i_{\lfloor (l+1) J / R \rfloor }) | J_{l,R}}(s) \asymp 1 \qquad \mbox{as } J \to \infty.\] 
Let $J' \in \mathbb{N}$ be large enough that for all $J \geq J'$ and $l \in \{0,1,\dotsc,R-1\}$, 
\[ \psi_{(i_{\lfloor l J / R \rfloor + 1} , \dotsc, i_{\lfloor (l+1) J / R \rfloor }) | J_{l,R}}(s) \geq e^{-\overline{t}J_{l,R}/R}.\] %
Assume $J \geq J'$. For each $k \in \{1,\dotsc,J\}$ let $\mathbf{p}(k)$ denote the type class of $(i_1,\dotsc,i_k)$ and let $l \in \{0,1,\dotsc,R-1\}$ be such that $\lfloor l J / R \rfloor < k \leq \lfloor (l+1) J / R \rfloor$. 
Then 
\begin{align*}
\psi_{(i_1,\dotsc,i_k) | k}(s) \stackrel{\eqref{eq:308}}{=} e^{k( -t+ \sum_{\ih} p_{\ih}(k)\log N_{\ih})} &\geq e^{\lfloor lJ/R\rfloor (-t  + \sum_{\ih} p_{\ih}(\lfloor lJ/R\rfloor)\log N_{\ih})} e^{(\lfloor lJ/R\rfloor - k)t} \\
&\geq \psi_{(i_1,\dotsc,i_{\lfloor l J / R \rfloor }) |  \lfloor l J / R \rfloor}(s) e^{-2\overline{t}J/R} \\*
&\geq e^{-3\overline{t}J/R},\stepcounter{equation}\tag{\theequation}\label{eq:lowerboundcombinatorialfudge}
\end{align*}
where in the last step we use that by~\eqref{eq:20}, 
\begin{equation*} \psi_{(i_1,\dotsc,i_A,i_{A+1},\dotsc,i_{A+A'}) | A+A'}(s) = \psi_{(i_1,\dotsc,i_A) | A}(s) \psi_{(i_{A+1},\dotsc,i_{A+A'}) | A'} (s).
\end{equation*} Therefore  
\begin{equation*}
\underline{P}(s) = \liminf_{J\to\infty} \frac{1}{J} \log  \Psi_J(s) \geq \liminf_{J\to\infty} \frac{1}{J} \log  \sum_{\iiv\in S_{t,J,R}} \min_{k\in\{0,1,\dotsc,J\}} \psi_{\iiv|k}(s) \geq
H(\mathbf{Q}^*_{t})-3\overline{t}/R. 
\end{equation*}
Since $R \in \mathbb{N}$ was arbitrary, the assertion $\underline{P}(s) \geq H(\mathbf{Q}^*_{t})$ follows.

	\subsection{Proof of Step 4}
This is an application of~\cite[Theorem~1]{Jordan2011bm}, and in fact holds for all 
\[ t \in \left(\min_{1 \leq i \leq M} \log N_i, \max_{1 \leq i \leq M} \log N_i\right).\]
 Indeed, by~\eqref{eq:definebeta}, for all $\lambda$, 
\begin{gather*}
(\log m) \left(  \beta_{\nu}\left( \frac{\lambda - \gamma^{-1}}{(\log m) (\frac{1}{\log m} - \frac{1}{\log n})}  \right)  + \frac{(\lambda - \gamma^{-1})\frac{\log N}{\log m}}{(\log m) (\frac{1}{\log m} - \frac{1}{\log n})}  \right)  -  \log M \\*
\qquad \qquad = \log \left( \frac{1}{M}\sum_{\ih=1}^M N_i^{\lambda} \right). 
\end{gather*}
Taking the Legendre transform of each side and using standard properties of Legendre transforms and~\eqref{eq:22} and~\eqref{eq:uniformmultifractal} proves Step~4. 
This completes the proof of Proposition~\ref{prop:1}.

\section{Preliminaries to the proof of Theorem~\ref{thm:main}}\label{sec:PrelimsProofMain}

Here we collect some notation and facts used in the proof of Theorem~\ref{thm:main}. Throughout the section, $\Lambda$ is a Bedford--McMullen carpet with non-uniform fibres.

\subsection{General theory}
Following~\cite{Burrell2021projections}, for a bounded and non-empty set $F\subset \mathbb{R}^d$, $\theta\in(0,1)$ and $s\in[0,d]$, let us introduce
\begin{equation}\label{eq:41}
S_{\delta, \theta}^{s}(F)\coloneqq \inf \Big\{\sum_{i}\left|U_{i}\right|^{s}:\left\{U_{i}\right\}_{i} \text { is a cover of } F \text { such that } \delta^{1/\theta} \leq\left|U_{i}\right| \leq \delta \text { for all } i\Big\}.
\end{equation}
The motivation for introducing $S_{\delta, \theta}^{s}(F)$ is that from~\cite[Lemma~2.1]{Burrell2021projections} and the definitions of $\underline{\dim}_{\, \theta}F$ and $\overline{\dim}_{\, \theta}F$ it follows that
\begin{equation*}
\underline{\dim}_{\, \theta} F=\text { the unique } s \in[0, d] \text { such that } \liminf_{\delta\searrow 0} \frac{\log S_{\delta, \theta}^{s}(F)}{-\log \delta}=0
\end{equation*}
and
\begin{equation}\label{eq:45}
\overline{\dim}_{\, \theta} F=\text { the unique } s \in[0, d] \text { such that } \limsup_{\delta\searrow 0} \frac{\log S_{\delta, \theta}^{s}(F)}{-\log \delta}=0.
\end{equation}
For each $\theta \in (0,1)$, $\liminf_{\delta\searrow 0} \frac{\log S_{\delta, \theta}^{s}(F)}{-\log \delta}$ and $\limsup_{\delta\searrow 0} \frac{\log S_{\delta, \theta}^{s}(F)}{-\log \delta}$ are strictly decreasing and continuous functions of $s$. 
To bound $S_{\delta, \theta}^{s}(\Lambda)$ from above, we construct an efficient cover of $\Lambda$ in Section~\ref{sec:mainupper}. %
The mass distribution principle is a useful tool to bound some dimension of a set from below by putting a measure on the set. For the matching lower bound, we will use the following version of the mass distribution principle for the intermediate dimensions of Falconer, Fraser and Kempton,~\cite[Proposition~2.2]{Falconer2020firstintermediate}. 

\begin{prop}\label{prop:mdp}
Let $F$ be a non-empty, bounded subset of $\mathbb{R}^d$, and let $\theta \in [0,1]$, $s \geq 0$, $\delta_0 \in (0,1)$. Suppose that for all $\delta \in (0,\delta_0)$ there exists a Borel measure $\mu_\delta$ with support $\supp(\mu_{\delta}) \subseteq F$ such that $\mu_\delta(U) \leq |U|^s$ for all Borel sets $U \subset \mathbb{R}^d$ with $\delta^{1/\theta} \leq |U| \leq \delta$. Then 
\[\liminf_{\delta\searrow 0} \frac{\log S_{\delta, \theta}^{s}(F)}{-\log \delta} \geq \liminf_{\delta\searrow 0} \frac{\log \mu_\delta(\supp(\mu_\delta))}{-\log \delta}.\]
The same holds if we replace $\liminf$ with $\limsup$.  
\end{prop}

\begin{proof}
If $\{U_i\}$ is a cover of $F$ with $\delta^{1/\theta} \leq |U_i| \leq \delta$ for all $i$, then $\mathrm{supp}(\mu_{\delta}) \subseteq F \subseteq \cup_i U_i$. Therefore 
\[ \mu_\delta (\mathrm{supp}(\mu_\delta)) \leq \sum_i \mu_\delta(U_i) \leq \sum_i |U_i|^s.\] 
Since the cover was arbitrary, also $\mu_\delta (\mathrm{supp}(\mu_\delta)) \leq S_{\delta,\theta}^s$. The result follows.  
\end{proof}

\subsection{Approximate squares}\label{subsec:approxsquare}

Let $\mathcal{F}=\{f_i\}_{i=1}^N$ be an IFS generating a Bedford--McMullen carpet $\Lambda$ with $M$ non-empty columns. Recall, $[N]=\{1,2,\dotsc,N\}$ and $[M]=\{1,2,\dotsc,M\}$. To keep track of which column $f_i$ maps to, we introduce the function
\begin{equation}
\phi\colon [N]\to[M],\;\; \phi(i)\coloneqq \ih\; \text{ if } f_i \text{ maps to column } \ih.
\end{equation}\label{eq:40}
We define the symbolic spaces
$\Sigma=[N]^\mathbb N$ and $\Sigma_{\mathcal{H}}=[M]^\mathbb N$
with elements $\ii=i_1i_2\dotsb\in\Sigma$ and $\iih=\ih_1\ih_2\dotsb\in\Sigma_{\mathcal{H}}$. We use the convention that indices $i$ corresponding to maps have a `dot' while the indices $\ih$ corresponding to columns have a `hat' on top. 
To truncate~$\ii$ (respectively~$\iih$) at the first $n$ symbols we write $\ii|n=i_1i_2\dotsb i_n$ (respectively $\iih|n$). The longest common prefix of $\ii$ and $\jj$ is denoted by $\ii\wedge\jj$: its length is $|\ii\wedge\jj|=\min \{\, k : i_k\neq j_k \, \}-1$. The function $\phi$ naturally induces the map $\Phi\colon  \Sigma\to\Sigma_{\mathcal{H}}$ defined by
\begin{equation*}
\Phi(\ii)\coloneqq \iih = \phi(i_1)\phi(i_2)\dotsb.
\end{equation*}
Slightly abusing notation, $\Phi$ is also defined on finite words: $\Phi(i_1\dotsb i_n)=\phi(i_1)\dotsb\phi(i_n)$.

For compositions of maps, we use the standard notation $f_{i_1\dotsb i_n}\coloneqq f_{i_1}\circ f_{i_2}\circ\dotsb \circ f_{i_n}$.
The $n$-th level cylinder corresponding to $\ii$ is
$ C_{n}(\ii)\coloneqq f_{\ii|n}([0,1]^2)$.
The sets $\{C_{n}(\ii)\}_{n=1}^\infty$ form a nested sequence of compact sets with diameter tending to zero, hence their intersection is a unique point $x\in\Lambda$. This defines the natural projection $\Pi\colon \Sigma\to\Lambda$,
\begin{equation*}
\Pi(\ii)\coloneqq \lim_{n\to\infty} \bigcap_{n=1}^\infty C_{n}(\ii)=\lim_{n\to\infty} f_{\ii|n}(\underline 0).
\end{equation*}
The coding of a point $x\in\Lambda$ is not necessarily unique, but $\Pi$ is finite-to-one.

It is not efficient to cover cylinder sets separately; instead they are grouped together to form `approximate squares' which play the role of balls in a cover of the attractor. 
Recall, $\gamma=\log_m n$ and for $\ell,K\in\mathbb{N}$ let
\begin{equation*}
\gamma^{\ell}(K)\coloneqq \lfloor \gamma^{\ell} \cdot K \rfloor.
\end{equation*}
In particular, we write $\gamma(K) = \gamma^1(K)$, and $n^{-K}\leq m^{-\gamma(K)} < n^{-(K-1)}$. A level-$K$ \emph{approximate square} is
\begin{equation*}
B_K(\ii) \coloneqq \big\{ \, C_{\gamma(K)}(\jjv): \jjv\in[N]^{\gamma(K)},\, \jjv|K=\ii|K,\, \Phi(\jjv) = \Phi(\ii|\gamma(K)) \, \big\}.
\end{equation*}
It is a collection of level-$\gamma(K)$ cylinder sets that lie in the same level-$\gamma(K)$ column of a specific level-$K$ cylinder set. In other words, $\Pi(\jj)\in B_K(\ii)$ if and only if $|\ii\wedge\jj|\geq K$ and $|\Phi(\ii)\wedge\Phi(\jj)|\geq \gamma(K)$. Hence, abusing notation slightly, we identify $B_K(\ii)$ with the single sequence
\begin{equation*}%
	B_K(\ii) = (i_1,\dotsc,i_{K}, \phi(i_{K+1}), \dotsc, \phi(i_{\gamma(K)})) = (i_1,\dotsc,i_{K}, \ih_{K+1}, \dotsc, \ih_{\gamma(K)}).
\end{equation*}
The choice of $\gamma(K)$ implies that there exists $C\geq 1$ independent of $K$ and $\ii$ such that $C^{-1}n^{-K}\leq |B_K(\ii)|\leq C n^{-K}$. The constant $C$ does not influence the behaviour of the $s$-cost of any cover with approximate squares. It is easy to see that two  approximate squares are either disjoint, completely agree or intersect just on the boundary. Hence, the set of all level-$K$ approximate squares, denoted by $\mathcal{B}_K$, gives an efficient $n^{-K}$-cover of $\Lambda$ with cardinality
\begin{equation*}
\# \mathcal{B}_K = N^{K}\cdot M^{\gamma(K)-K} \stackrel{\eqref{eq:103}}{\asymp} n^{K\dim_{\mathrm B}\Lambda}.
\end{equation*}

\subsection{Two lemmas}

Recall that $T_s(t)=\big(s-\frac{\log M}{\log m}\big)\log n +\gamma I(t)$. Since $I(t)$ is strictly convex, there exists a unique~$t'$ such that $I'(t')=\gamma^{-1}$. 

\begin{lemma}\label{lem:41}
For each fixed $s \in \R$, the function $T_s(t)$ is strictly convex with a minimum at $\underline{t}$, and satisfies $T_s'(t') = 1$. 
Moreover, for all $s\geq \dim_{\mathrm H}\Lambda$ and $t\in\mathbb{R}$ we have $T_s(t) \geq t$ with equality if and only if $s=\dim_{\mathrm H}\Lambda$ and $t=t'$.
\end{lemma}

\begin{proof}
Since $I(t)$ is strictly convex with a minimum at $\underline{t}$, the same is true of $T_s(t)$ for each fixed $s \in \R$, and the definition of $t'$ implies that $T'_s(t')=1$. 
Using the formula~\eqref{eq:102} for $\dim_{\mathrm H}\Lambda$ and then that 
\[I(t')=\gamma^{-1}t'-\log\bigg(\frac{1}{M}\sum_{\ih=1}^M N_{\ih}^{\gamma^{-1}}\bigg),\]
one gets $T_{\dim_{\mathrm H}\Lambda}(t')=t'$ after simplifications. 
Note that $T_s(t') > T_{\dim_{\mathrm H}\Lambda}(t')=t'$ for all $s > \dim_{\mathrm H}\Lambda$, which is enough to complete the proof. 
\end{proof}

 Since $I(t)$ is strictly increasing, let $t^*$ denote the unique solution to the equation
\begin{equation}\label{eq:44}
\dim_{\mathrm H}\Lambda = \dim_{\mathrm B}\Lambda - \big(1-\gamma^{-1}\big)\frac{I(t^*)}{\log m}. 
\end{equation}
Recall the notation for $t_\ell(s)$ from~\eqref{eq:definetsequence}. 

\begin{lemma}\label{lem:tprimelesststar}
We have $t_1(\dim_{\mathrm H} \Lambda) < t' < t^*$. 
\end{lemma}
\begin{proof}
Since $I(\underline{t}) = 0$, we have $t_1(\dim_{\mathrm H} \Lambda) = T_{\dim_{\mathrm H} \Lambda} (\underline{t})$. 
Also $I'(\underline{t}) < \gamma^{-1} = I'(t')$ and $I$ is strictly convex, so $\underline{t} < t'$. 
But $T_{\dim_{\mathrm H} \Lambda}$ is strictly increasing, so 
\[ t_1(\dim_{\mathrm H} \Lambda) = T_{\dim_{\mathrm H} \Lambda} (\underline{t}) < T_{\dim_{\mathrm H} \Lambda} (t') = t',\] 
where the last equality is by Lemma~\ref{lem:41}. 

To prove $t' < t^*$, for $z \in \mathbb{R}$ define 
\[ f(z) \coloneqq \log n \log M - \log n \log \sum_{\ih=1}^M N_{\ih}^{\frac{z}{\log n}}  + \log n \log (N/M) - (\log n - z)\frac{\sum_{\hat k=1}^M N_{\hat k}^{\frac{z}{\log n}}\log N_{\hat k}}{\sum_{\jh=1}^M N_{\jh}^{\frac{z}{\log n}}}. \]
Then after some algebraic manipulations, 
\begin{equation*}%
f'(z) = -\left(1 - \frac{z}{\log n}\right) \left( \sum_{\ih} \frac{N_{\ih}^{\frac{z}{\log n}}}{\sum_{\jh} N_{\jh}^{\frac{z}{\log n}}} (\log N_{\ih})^2  -  \Bigg( \sum_{\hat k} \frac{N_{\hat k}^{\frac{z}{\log n}}}{\sum_{\hat\ell} N_{\hat\ell}^{\frac{z}{\log n}}} \log N_{\hat k} \Bigg)^2  \right)  < 0 
\end{equation*}
for all $z \in [\log m,\log n)$ by Jensen's inequality, using that $\Lambda$ has non-uniform fibres. Moreover, $f$ is continuous on $\mathbb{R}$, so $f(\log m) > f(\log n) = 0$, so  
\begin{align}
0 &< \frac{f(\log m)}{\log (n/m)} \nonumber \\*
&= \frac{\mathrm{d}}{\mathrm{d}\lambda} \Bigg(   \lambda \Big( \Big( \dim_{\mathrm H} \Lambda - \frac{\log M}{\log m} \Big) \log n  +  \gamma (\dim_{\mathrm B} \Lambda - \dim_{\mathrm H} \Lambda)  \frac{\log m}{1-\gamma^{-1}} \Big)  \nonumber \\*
&\phantom{----} -  \log\Bigg(\frac{1}{M}\sum_{\jh = 1}^M N_{\jh}^\lambda \Bigg)  \Bigg) \Big|_{\lambda = \gamma^{-1}} \label{eq:tprimelesststar2} \\
&= \frac{\mathrm{d}}{\mathrm{d}\lambda} \left( \lambda T_{\dim_{\mathrm H} \Lambda}(t^*) - \log \Bigg(\frac{1}{M}\sum_{\jh = 1}^M N_{\jh}^\lambda \Bigg) \right) \Big|_{\lambda = \gamma^{-1}} \label{eq:tprimelesststar3}, 
\end{align}
where~\eqref{eq:tprimelesststar2} is by~\eqref{eq:102} and~\eqref{eq:103}, and~\eqref{eq:tprimelesststar3} is by~\eqref{eq:44} and~\eqref{eq:defineiteratingfunction}. 
This means that the value of $\lambda$ at which the supremum in the definition of $I(T_{\dim_{\mathrm H} \Lambda}(t^*))$ in~\eqref{eq:22} is attained is greater than $\gamma^{-1}$. Equivalently, $I'(T_{\dim_{\mathrm H} \Lambda}(t^*)) > \gamma^{-1}$. By the definition of $t'$, this means that $T_{\dim_{\mathrm H} \Lambda}(t^*) > t'$. By Lemma~\ref{lem:41}, it follows that $t' < t^*$. 
\end{proof}

\section{Proof of Theorem \ref{thm:main}}\label{sec:proofofmain}

Throughout, $\Lambda$ is a Bedford--McMullen carpet with non-uniform fibres. Recall, for $\theta\in(0,1)$ we defined $L \coloneqq 1 + \lfloor \frac{-\log \theta}{\log \gamma}\rfloor$ so that $\gamma^{-L} < \theta \leq \gamma^{-(L-1)}$, and for $\delta>0$ we define $K \coloneqq \lfloor\frac{-\log \delta}{\log n}\rfloor$. We begin in Section~\ref{sec:upperinteger} by constructing a simpler cover in the case when $\theta = \gamma^{-(L-1)}$ in order to establish certain relations which are crucial in bounding the $s$-cost of our general cover in Section~\ref{sec:mainupper}. Section~\ref{sec:prooflower} establishes the matching lower bound. 

\subsection{Upper bound \texorpdfstring{for $\theta = \gamma^{-(L-1)}$}{when θ is a power of 1/γ}}\label{sec:upperinteger}

In Lemma~\ref{lem:50} we construct a cover for $\theta = \gamma^{-(L-1)}$. 
The proof strategy is to keep a level-$K$ approximate square at level $K$ if and only if $\tau(\iih,K,\gamma(K))$ exceeds a constant $\tau_1$ which remains unspecified for now (recalling notation from~\eqref{eq:tauaverage}). Of those which we subdivide, we keep them at level $\gamma(K)$ if and only if $\tau(\iih,\gamma(K),\gamma^2(K)) \geq \tau_2$. Continuing this process gives a cover of $\Lambda$ using approximate squares at levels $K, \gamma(K),\dotsc,\gamma^{L-1}(K) = \lfloor K/\theta\rfloor$. This means that, up to some constant, all covering sets have diameter in the correct range $[\delta,\delta^{1/\theta}]$. In Lemma~\ref{lem:50} we calculate the $s$-cost of this cover for an arbitrary tuple, which will allow us to prove in Lemma~\ref{lem:51} that the relevant $t_i(s)$ are bounded above by $\overline{t}$, so results from Section~\ref{sec:proofProp} will apply. At the end we will optimise the thresholds so that the exponential rate of growth of each part is the same. The unique $s$ for which this can be done gives us the upper bound for $\overline{\dim}_{\, \theta}\Lambda$. 

Lemma~\ref{lem:simpleindcomb} will be used to calculate the cost of the cover in Lemma~\ref{lem:50}. 
Lemma~\ref{lem:simpleindcomb} is rather similar to Lemma~\ref{lem:33-second}, and is also proved using the method of types. 
For a probability vector $\mathbf{p}$ we write $t_{\mathbf{p}} \coloneqq \sum_{\ih} p_{\ih} \log N_{\ih}$. For a word $(i_1,\dotsc,i_J) \in [N]^J$, we write $\iih=(\ih_1,\dotsc,\ih_J)=(\phi(i_1),\dotsc,\phi(i_J))\in [M]^J$, recall~\eqref{eq:40}.

\begin{lemma}\label{lem:simpleindcomb}
For all $t \in [\underline{t}, \max_{1\leq \ih \leq M} \log N_{\ih})$, 
\[ \# \{ \, (i_1,\dotsc,i_J) \in [N]^J : \tau(\iih,0,J) \leq t \, \} \asymp e^{(\min\{t,\overline{t}\} + \log M - I(\min\{t,\overline{t}\} ))J}. \]
If, on the other hand, $t \in (\min_{1\leq \ih \leq M} \log N_{\ih}, \underline{t})$, then 
\[ \limsup_{J \to \infty} \frac{1}{J} \log \# \Big\{ \, (i_1,\dotsc,i_J) \in [N]^J : \sum_{j=1}^J \log N_{\ih_j} \leq t J \, \Big\} \leq t + \log M < \underline{t} + \log M. \]
\end{lemma}

\begin{proof}
The second part of the statement holds simply by the fact that there are $M^J$ strings of length $J$ on the alphabet $[M]$, so we only prove the first part of the statement. 
The strategy for the upper bound is to work with an arbitrary type class and then use the fact that there are only polynomially many type classes. Note that if $\mathbf{p} \in \mathcal{T}_J$ and $\jjv$ is any representative of the type class $T_J(\mathbf{p})$ then 
\[ \# \{ \, \mathbf{i} \in [N]^J : \iih = \jjv \, \} = \prod_{\kh = 1}^M N_{\kh}^{p_{\kh} J}.\]  Therefore 
\begin{align*}
\# &\{ \, (i_1,\dotsc,i_J) \in [N]^J : \tau(\iih,0,J) \leq t \, \} = \sum_{\mathbf{p} \in \mathcal{T}_J : t_{\mathbf{p}} \leq t}  \# T_J(\mathbf{p}) \cdot \prod_{\kh = 1}^M N_{\kh}^{p_{\kh} J}  \\
&\leq \sum_{\mathbf{p} \in \mathcal{T}_J : t_{\mathbf{p}} \leq t} \prod_{\kh = 1}^M N_{\kh}^{p_{\kh} J} \cdot \# \Big\{ \, \iih \in [M]^J : \frac{1}{J} \sum_{\ell=1}^{J} \log N_{\ih_\ell} \geq t_{\mathbf{p}} \, \Big\}  \\
&\asymp \sum_{\mathbf{p} \in \mathcal{T}_J : t_{\mathbf{p}} \leq t} e^{t_{\mathbf{p}} J} e^{J\cdot H(\mathbf{Q}^*_{t_{\mathbf{p}}})} \qquad \text{by Lemma~\ref{lem:33-first}} \\
&\asymp \sum_{\mathbf{p} \in \mathcal{T}_J : t_{\mathbf{p}} \leq t} e^{J( t_{\mathbf{p}} + \log M - I(t_{\mathbf{p}}))} \qquad \text{by Step~1 of Proposition~\ref{prop:1}} \\
& \leq \# \mathcal{T}_J \cdot e^{(\min\{t,\overline{t}\} + \log M - I(\min\{t,\overline{t}\} ))J} \quad \text{since } \substack{ 0 < I'(\tau) < 1 \text{ if } \tau \in (\underline{t},\overline{t}], \\ I'(\tau) > 1 \text{ if } \tau \in (\overline{t},\max_{1\leq \ih \leq M} \log N_{\ih})} \\
&\asymp e^{(\min\{t,\overline{t}\} + \log M - I(\min\{t,\overline{t}\} ))J} \qquad \text{by the upper bound of~\eqref{eq:206}}. 
\end{align*}

For the lower bound, if $\mathbf{q}$ is the closest approximation in $\mathcal{T}_J$ to $\mathbf{Q}^*_{\min\{t,\overline{t}\}}$ for which $t_{\mathbf{q}} \leq t$, then 
\begin{align*}
 &\# \{ \, (i_1,\dotsc,i_J) \in [N]^J : \tau(\iih,0,J) \leq t \, \} \geq  \# T_J(\mathbf{q}) \cdot \prod_{\kh = 1}^M N_{\kh}^{q_{\kh} J} \\
 &\geq e^{t_{\mathbf{q}} J} (J+1)^{-M} e^{J\cdot H(\mathbf{q})} \qquad \text{by the lower bound of~\eqref{eq:206}}\\
 &\asymp e^{(\min\{t,\overline{t}\} + H(\mathbf{Q}^*_{\min\{t,\overline{t}\}}))J} \qquad  \substack{\text{since } \mathbf{q} \to \mathbf{Q}^*_{\min\{t,\overline{t}\}} \in \mathcal{E}_{\min\{t,\overline{t}\}} \text{ by Lemma~\ref{lem:32}}\\
 \text{and since } H \text{ is continuous}}\\
 &\asymp e^{(\min\{t,\overline{t}\} + \log M - I(\min\{t,\overline{t}\}))J} \qquad \text{by Step~1 of Proposition~\ref{prop:1}}.
 \end{align*}
Therefore the first part of the statement of Lemma~\ref{lem:simpleindcomb} holds. 
\end{proof}

In order to calculate the $s$-cost of the cover we construct in the proof of Lemma~\ref{lem:50}, for $\btau = (\tau_1,\dotsc,\tau_{L-1}) \in (\underline{t},\overline{t})^{L-1}$ and $s \in [\dim_{\mathrm H} \Lambda,\dim_{\mathrm B} \Lambda]$ we introduce 
\begin{align*}
G_{1}^{\btau}(s) \coloneqq \frac{\log N}{\log n} &+ \gamma^{L-1}\big(1-\gamma^{-1}\big)\frac{\log M}{\log m} - \gamma^{L-1}s \\*
&+ \frac{\gamma-1}{\log n} \sum_{i=0}^{L-2}\gamma^{i}(\tau_{L-1-i}+\log M -I(\tau_{L-1-i})) 
\end{align*}
and 
\begin{align*}%
 G_{\ell}^{\btau}(s) \coloneqq \frac{\log N}{\log n} &+\frac{\gamma-1}{\log n} \Big( \gamma^{L-\ell}(\log M - I(\tau_{\ell - 1}))+\sum_{i=0}^{L-1-\ell} \gamma^{i}(\tau_{L-1-i}+\log M -I(\tau_{L-1-i})) \Big) \\*
 &-\gamma^{L-\ell}s 
\end{align*}
for $\ell = 2,3,\dotsc,L$. In particular, when $\ell=L$ the sum is empty and 
\[ G_{L}^{\btau}(s) = \dim_{\mathrm B} \Lambda - (1-\gamma^{-1})\frac{I(\tau_{L-1})}{\log m} - s\]
(note the similarity with~\eqref{eq:44}). 

\begin{lemma}\label{lem:50}
For all $L \geq 2$, all tuples $\btau = (\tau_1,\dotsc,\tau_{L-1}) \in (\underline{t},\overline{t})^{L-1}$, and all $s \in [\dim_{\mathrm H} \Lambda,\dim_{\mathrm B} \Lambda]$, 
\begin{equation*}
\limsup_{\delta\searrow 0} \frac{\log S_{\delta, \gamma^{-(L-1)}}^{s}(\Lambda)}{-\log \delta} \leq \max_{1 \leq \ell \leq L} G_{\ell}^{\btau}(s),
\end{equation*}
where $S_{\delta, \theta}^{s}(\Lambda)$ was introduced in~\eqref{eq:41}. 
\end{lemma}

\begin{proof}
We construct a cover of $\Lambda$ as follows. For $\ell \in \mathbb{N}$, let $J_{\ell}\coloneqq \gamma^{\ell}(K)-\gamma^{\ell-1}(K)$. %
Define 
\begin{equation*}
\mathcal{U}_{1}^{\btau} \coloneqq \left\{ B_{\gamma^{L-1}(K)}(\mathbf{i}) \in \mathcal{B}_{\gamma^{L-1}(K)} : \tau(\iih,\gamma^k(K),\gamma^{k+1}(K)) \leq \tau_{L-1-k} \;\text{ for } k=0,\dotsc,L-2 \right\}.
\end{equation*}

For $\ell \in \{2,3,\dotsc,L-1 \}$, define $\mathcal{U}_{\ell}^{\btau}$ to be the set of level-$\gamma^{L-\ell}(K)$ approximate squares $B_{\gamma^{L-\ell}(K)}(\mathbf{i})$ for which
\begin{align*}%
&\tau(\iih,\gamma^k(K),\gamma^{k+1}(K)) \leq \tau_{L-k-1} \;\text{ for } k=0,\dotsc,L-\ell-1, \\*
\text{ and } &\tau(\iih,\gamma^{L-\ell}(K),\gamma^{L-\ell+1}(K)) > \tau_{\ell - 1}. 
\end{align*}
Define $\mathcal{U}_{L}^{\btau} \coloneqq \big\{ \,  B_{K}(\ii)\in \mathcal{B}_{K} : \tau(\iih,K,\gamma(K)) > \tau_{L-1} \, \big\}$. 
By construction this is a cover: 
\[ \Lambda \subseteq \bigcup_{\ell=1}^{L}\mathcal{U}_{\ell}^{\btau}. \]
Observe that for all $0\leq k<\ell\leq L-1$, if $B\in\mathcal{U}_k$ and $B'\in\mathcal{U}_{\ell}$ then they are either disjoint or intersect on their boundary, but it can never happen that $B\cap B'=B'$. 

For $\ell \in \{2,3,\dotsc, L\}$, the symbolic representation of a level-$\gamma^{L-\ell}(K)$ approximate square $B_{\gamma^{L-\ell}(K)}(\mathbf{i}) \in \mathcal{U}_{\ell}^{\btau}$ is 
\[ (\underbrace{ i_1,\dotsc, i_{K} }_{\in  [N] \mbox{ freely}} , \underbrace{ i_{K+1},\dotsc, i_{\gamma(K)} }_{\tau(\iih,K,\gamma(K)) \leq \tau_{L-1}} , \dotsb , \underbrace{ i_{\gamma^{L-\ell-1}(K)+1},\dotsc,i_{\gamma^{L-\ell}(K)} }_{\tau(\iih,\gamma^{L-\ell-1}(K),\gamma^{L-\ell}(K)) \leq \tau_{\ell}}, \underbrace{ \ih_{\gamma^{L-\ell}(K) + 1},\dotsc,\ih_{\gamma^{L-l+1}(K)} }_{\tau(\iih,\gamma^{L-\ell}(K),\gamma^{L-l+1}(K)) > \tau_{\ell - 1}}). \] 
Therefore the $s$-cost of $\mathcal{U}_{\ell}^{\btau}$ is 
\begin{align*}
\sum_{U \in \mathcal{U}_{\ell}^{\btau}} |U|^s &\asymp \# \mathcal{U}_{\ell}^{\btau} \cdot n^{-\gamma^{L-\ell}Ks} \\
&\asymp N^K \prod_{k=0}^{L-\ell-1} \# \{ \, \mathbf{i} \in [N]^{J_{k+1}} : \tau(\iih,0,J_{k+1}) \leq \tau_{L-k-1} \, \} \\*
&\phantom{\leq}\times \# \{ \,  \jjh \in [M]^{J_{L-\ell+1}} : \tau(\jjh,0,J_{L-\ell+1}) > \tau_{\ell - 1} \,\} \cdot n^{-\gamma^{L-\ell}Ks} \\
&\asymp N^K \prod_{k=0}^{L-\ell-1} e^{(\tau_{L-k-1} + \log M - I(\tau_{L-k-1}))J_{k+1}} \cdot e^{(\log M - I(\tau_{\ell - 1}))J_{L-l+1}} n^{-\gamma^{L-\ell}Ks} \\
&\asymp n^{K\cdot G_{\ell}^{\btau}(s)},\stepcounter{equation}\tag{\theequation}\label{eq:intcost1}
\end{align*}
by Lemmas~\ref{lem:simpleindcomb} and~\ref{lem:33-first} and algebraic manipulations. In the case $l=L$ we used the convention that the empty product equals 1. 

The symbolic representation of $B_{\gamma^{L-1}(K)}(\mathbf{i}) \in \mathcal{U}_1^{\btau}$ is 
\[ (\underbrace{ i_1,\dotsc, i_{K} }_{\in  [N] \mbox{ freely}} , \underbrace{ i_{K+1},\dotsc, i_{\gamma(K)} }_{\tau(\iih,K,\gamma(K)) \leq \tau_{L-1}} , \dotsb , \underbrace{ i_{\gamma^{L-2}(K)+1},\dotsc,i_{\gamma^{L-1}(K)} }_{\tau(\iih,\gamma^{L-2}(K),\gamma^{L-1}(K)) \leq \tau_1} , \underbrace{ \ih_{\gamma^{L-1}(K) + 1},\dotsc,\ih_{\gamma^{L}(K)} }_{\in [M] \mbox{ freely}} ). \]
Therefore, as in~\eqref{eq:intcost1}, 
\begin{equation}\label{eq:intcost2}
\sum_{U \in \mathcal{U}_1^{\btau}} |U|^s \asymp N^K \prod_{k=0}^{L-2} e^{(\tau_{L-k-1} + \log M - I(\tau_{L-k-1}))J_{k+1}} \cdot M^{\gamma^L K - \gamma^{L-1}K} n^{-\gamma^{L-1}Ks} \asymp n^{K \cdot G_{1}^{\btau}(s)}. 
\end{equation}
We have bounded the $s$-cost of each part of the cover, so the proof is complete. 
\end{proof}
In~\eqref{eq:intcost1} and~\eqref{eq:intcost2} it was crucial that each $\tau_i \in (\underline{t},\overline{t})$ when using Lemma~\ref{lem:simpleindcomb}. 
Lemma~\ref{lem:50} tells us the exponential growth rate of the cover for all tuples $\btau$, but motivated by~\eqref{eq:45}, of particular interest is the case when 
\[  G_{1}^{\btau}(s) = \dotsb = G_{L}^{\btau}(s) = 0.\]
A tuple $\btau = (\tau_1,\dotsc,\tau_{L-1}) \in (\underline{t},\overline{t})^{L-1}$ satisfies $G_{1}^{\btau}(s) = \dotsb = G_{L}^{\btau}(s)$ if and only if $\tau_1 = t_1(s)$ (from $G_{1}^{\btau}(s) = G_{2}^{\btau}(s)$) and $\tau_{k} = T_s(\tau_{k-1})$ for $k = 2,3,\dotsc , L$ (from $G_{k}^{\btau}(s) = G_{k+1}^{\btau}(s)$). Equivalently, $\tau_k = t_k(s)$ for $k = 1,2,\dotsc,L-1$. 
The next lemma ensures that each of $t_1(s), \dotsc,t_L(s)$ lies in the correct range $(\underline{t},\overline{t})$.  In particular, writing $\mathbf{t}\coloneqq (t_1(s), \dotsc,t_{L-1}(s))$, we can then apply Lemma~\ref{lem:50} to obtain the upper bound 
\[ \limsup_{\delta\searrow 0} \frac{\log S_{\delta, \gamma^{-(L-1)}}^{s}(\Lambda)}{-\log \delta} \leq G_{L}^{\mathbf{t}}(s) = \dim_{\mathrm B} \Lambda - (1-\gamma^{-1})\frac{I(t_{L-1}(s))}{\log m} - s. \] 

\begin{lemma}\label{lem:51}
Let $\theta\in(0,1)$, $L = L(\theta) \coloneqq 1 + \lfloor \frac{-\log \theta}{\log \gamma} \rfloor$ and $s\in(\dim_{\mathrm H}\Lambda,\overline{\dim}_{\gamma^{-(L-1)}}\,\Lambda]$. Then, using notation from~\eqref{eq:definetsequence} and~\eqref{eq:44},  
\[ \underline{t} < t_1(\dim_{\mathrm H} \Lambda) < t_1(s) < t_2(s) < \dotsb < t_L(s) \leq t_L(\overline{\dim}_{\gamma^{-(L-1)}} \Lambda) < T_{\dim_{\mathrm H} \Lambda}(t^*) < \overline{t}. \]%
\end{lemma}

\begin{proof}
 Recall from~\eqref{e:sunderlinedef} and~\eqref{e:soverlinedef} that $\underline{s}<\dim_{\mathrm H}\Lambda<s$. It is therefore immediate that
\begin{equation*}
\underline{t} = t_1(\underline{s}) < t_1(\dim_{\mathrm H} \Lambda) < t_1(s). 
\end{equation*}
It follows from Lemma~\ref{lem:41} that $t_1(s) < t_2(s) < \dotsb < t_L(s)$ for all $s>\dim_{\mathrm H}\Lambda$. 
Since $s \leq \overline{\dim}_{\gamma^{-(L-1)}} \Lambda$ we have $t_L(s) \leq t_L(\overline{\dim}_{\gamma^{-(L-1)}} \Lambda)$. 

We now prove $T_{\dim_{\mathrm H} \Lambda}(t^*) < \overline{t}$. To do so, we define, for every fixed $\mathbf{p} \in \mathcal{P}_M$, the function $f_{\mathbf{p}} \colon (0,\infty) \to \mathbb{R}$ by 
\[ f_{\mathbf{p}}(z) \coloneqq \log n \cdot \log \sum_{\ih=1}^M p_{\ih}^{z/\log n}   -   (\log n -z)H(\mathbf{p}) . \]
Recall that $\mathbf{Q} = (1/M,\dotsc,1/M)$ and $\mathbf{P} = (N_1/N,\dotsc,N_M/N)$. 
Clearly, $f_{\mathbf{p}}(\log n) = 0$ for all $\mathbf{p}$ and $f_{\mathbf{Q}}(z) = 0$ for every $z$. The derivative of $f_{\mathbf{p}}(z)$ with respect to $z$ is
\begin{equation*}
f'_{\mathbf{p}}(z) = H(\mathbf{p}) + \sum_{\ih=1}^M \frac{p_{\ih}^{z/\log n}}{\sum_{\jh=1}^M p_{\jh}^{z/\log n}} \log p_{\ih}, 
\end{equation*}
while after some algebraic manipulations, we obtain that
\begin{equation*}
f''_{\mathbf{p}}(z) =\frac{1}{\log n} \left(
\sum_{\ih=1}^M \frac{p_{\ih}^{z/\log n}}{\sum_{\jh=1}^M p_{\jh}^{z/\log n}} (\log p_{\ih})^2 -
\Bigg( \sum_{\ih=1}^M \frac{p_{\ih}^{z/\log n}}{\sum_{\jh=1}^M p_{\jh}^{z/\log n}} \log p_{\ih} \Bigg)^2 
\right) \geq 0
\end{equation*}
by Jensen's inequality with equality if and only if $\mathbf{p}=\mathbf{Q}$ (here we use that $\mathbf{p}$ has strictly positive entries). Hence, for $\mathbf{p}\neq\mathbf{Q}$, $f_{\mathbf{p}}(z)$ is a strictly convex function for all $z>0$, and also $f'_{\mathbf{p}}(\log n)=0$, so $f_{\mathbf{p}}(z)$ has a global minimum at $z=\log n$. In particular, since $m \neq n$ and $\mathbf{P} \neq \mathbf{Q}$ as we assume the carpet has non-uniform fibres, $f_{\mathbf{P}}(\log m) > 0$. 
Using formulae~\eqref{eq:102} and $\eqref{eq:103}$ for $\dim_{\mathrm H}\Lambda$ and $\dim_{\mathrm B}\Lambda$, algebraic manipulations show that $f_{\mathbf{P}}(\log m) > 0$ is equivalent to 
\[ \left( \dim_{\mathrm H} \Lambda - \frac{\log M}{\log m} \right) \log n  +  \gamma \frac{\log m}{1-\gamma^{-1}} (\dim_{\mathrm B} \Lambda - \dim_{\mathrm H} \Lambda)   <   \log N  - H(\mathbf{P}). \]
But we can express $I(t^*)$ from~\eqref{eq:44} and use the definition $\overline{t} \coloneqq \log N  - H(\mathbf{P})$ to show that this is equivalent to the assertion $T_{\dim_{\mathrm H} \Lambda}(t^*) < \overline{t}$, as required. 

It remains to prove $t_L(\overline{\dim}_{\gamma^{-(L-1)}}\Lambda) < T_{\dim_{\mathrm H} \Lambda}(t^*)$. To do so, we first prove the weaker claim $t_{L-1}(\overline{\dim}_{\gamma^{-(L-1)}} \Lambda)  <  t^*$ using the fact that Lemma~\ref{lem:50} holds for an arbitrary tuple $\btau$. Assume for a contradiction that $t_{L-1}(\overline{\dim}_{\gamma^{-(L-1)}} \Lambda)  \geq t^*$. We define a tuple $\btau = (\tau_1,\dotsc,\tau_{L-1}) \in (\underline{t},\overline{t})^{L-1}$ as follows. %
If $t_1(\overline{\dim}_{\gamma^{-(L-1)}} \Lambda) \geq t^*$, then define $\tau_l \coloneqq t^*$ for all $l \in \{ 1,2,\dotsc,L-1\}$, noting that $\underline{t} < t^* \leq T_{\dim_{\mathrm H} \Lambda}(t^*) < \overline{t}$. If, on the other hand, $t_1(\overline{\dim}_{\gamma^{-(L-1)}} \Lambda) < t^*$, then define 
\[ l_* \coloneqq \max \{ \, l \in \{ 1,2,\dotsc,L-2 \} : t_{l}(\overline{\dim}_{\gamma^{-(L-1)}} \Lambda) < t^* \, \} .\] %
For $l \in \{1,2,\dotsc,l_*\}$ let $\tau_l \coloneqq t_l(\overline{\dim}_{\gamma^{-(L-1)}} \Lambda)$. For $l \in \{l_*+1,l_* + 2,\dotsc,L-1\}$ let $\tau_l \coloneqq t_{l_*} + \frac{l - l_*}{L-1-l_*}(t^* - t_{l_*}(s))$, %
so 
\[ \underline{t} < t_1(\dim_{\mathrm H} \Lambda) < t_1(\overline{\dim}_{\gamma^{-(L-1)}} \Lambda) = \tau_1 < \tau_2 < \dotsb < \tau_{L-2} < \tau_{L-1} = t^* \leq T_{\dim_{\mathrm H} \Lambda}(t^*) < \overline{t}.\] 
In either case, $\tau_1 \leq t_1(\overline{\dim}_{\gamma^{-(L-1)}} \Lambda)$ and $\tau_{l+1} \leq T_{\overline{\dim}_{\gamma^{-(L-1)}} \Lambda}(\tau_{l})$ for all $l \in \{1,2,\dotsc,L-2\}$, so 
\[ G_{1}^{\btau}(\overline{\dim}_{\gamma^{-(L-1)}} \Lambda) \leq G_{2}^{\btau}(\overline{\dim}_{\gamma^{-(L-1)}} \Lambda) \leq \dotsb \leq G_{L}^{\btau}(\overline{\dim}_{\gamma^{-(L-1)}} \Lambda).\]
 Therefore by Lemma~\ref{lem:50}, 
\begin{align*}
0 = \limsup_{\delta\searrow 0} \frac{\log S_{\delta, \gamma^{-(L-1)}}^{\overline{\dim}_{\gamma^{-(L-1)}} \Lambda}(\Lambda)}{-\log \delta} &\leq \max_{1 \leq \ell \leq L} G_{\ell}^{\btau}(\overline{\dim}_{\gamma^{-(L-1)}} \Lambda) 
 = G_{L}^{\btau}(\overline{\dim}_{\gamma^{-(L-1)}} \Lambda) < 0 
\end{align*} 
(the last inequality holds since $\tau_{L-1} = t^*$ and $\overline{\dim}_{\gamma^{-(L-1)}} \Lambda > \dim_{\mathrm H} \Lambda$, see \cite[Section~4]{Falconer2020firstintermediate}), a contradiction. Thus $t_{L-1}(\overline{\dim}_{\gamma^{-(L-1)}} \Lambda)  <  t^* \leq T_{\dim_{\mathrm H} \Lambda}(t^*) < \overline{t}$ (using Lemma~\ref{lem:41}). 

To complete the proof that $t_L(\overline{\dim}_{\gamma^{-(L-1)}} \Lambda) < T_{\dim_{\mathrm H} \Lambda}(t^*)$, we apply Lemma~\ref{lem:50} again but this time with the optimal tuple $\mathbf{t} \coloneqq (t_1(\overline{\dim}_{\gamma^{-(L-1)}} \Lambda),\dotsc,t_{L-1}(\overline{\dim}_{\gamma^{-(L-1)}} \Lambda))$ which we now know lies in the correct range: 
\begin{align*}
 0 &= \limsup_{\delta\searrow 0} \frac{\log S_{\delta, \gamma^{-(L-1)}}^{\overline{\dim}_{\gamma^{-(L-1)}} \Lambda}(\Lambda)}{-\log \delta} \\
 &\leq \max_{1 \leq \ell \leq L} G_{\ell}^{\mathbf{t}}(\overline{\dim}_{\gamma^{-(L-1)}} \Lambda) \\
 &= G_{L}^{\mathbf{t}}(\overline{\dim}_{\gamma^{-(L-1)}} \Lambda) \\
 &= \dim_{\mathrm B}\Lambda - (1-\gamma^{-1})\frac{I\big( t_{L-1}(\overline{\dim}_{\gamma^{-(L-1)}} \Lambda)\big)}{\log m} - \overline{\dim}_{\gamma^{-(L-1)}} \Lambda,\stepcounter{equation}\tag{\theequation}\label{eq:combrightrangeupperbound}
 \end{align*}%
noting that all terms in the maximum are in fact equal by the definition of $\mathbf{t}$. Therefore 
\begin{align}
t_L(\overline{\dim}_{\gamma^{-(L-1)}} \Lambda) 
&=  T_{\overline{\dim}_{\gamma^{-(L-1)}} \Lambda}(t_{L-1}(\overline{\dim}_{\gamma^{-(L-1)}} \Lambda)) \label{eq:rightrangelast1}\\
&\leq \left(\overline{\dim}_{\gamma^{-(L-1)}} \Lambda  -  \frac{\log M}{\log m}\right) \log n  +   \gamma\frac{\log m}{1-\gamma^{-1}}  (  \dim_{\mathrm B} \Lambda - \overline{\dim}_{\gamma^{-(L-1)}} \Lambda )  \label{eq:rightrangelast2} \\
&<   \left( \dim_{\mathrm H} \Lambda -  \frac{\log M}{\log m}\right) \log n  +   \gamma\frac{\log m}{1-\gamma^{-1}}  (  \dim_{\mathrm B} \Lambda -  \dim_{\mathrm H} \Lambda)  \label{eq:rightrangelast3}\\
&= T_{\dim_{\mathrm H} \Lambda}(t^*), \label{eq:rightrangelast4}
\end{align}
where~\eqref{eq:rightrangelast1} is by~\eqref{eq:definetsequence};~\eqref{eq:rightrangelast2} is by~\eqref{eq:defineiteratingfunction} and~\eqref{eq:combrightrangeupperbound};~\eqref{eq:rightrangelast3} is since $\overline{\dim}_{\gamma^{-(L-1)}} \Lambda > \dim_{\mathrm H} \Lambda$; and~\eqref{eq:rightrangelast4} is by~\eqref{eq:defineiteratingfunction} and~\eqref{eq:44}. 
This completes the proof. 
\end{proof}

\subsection{Lower bound}\label{sec:prooflower}%
For $\theta \in (0,1)$, $s \in (\dim_{\mathrm{H}} \Lambda,\overline{\dim}_{\gamma^{-(L-1)}} \Lambda]$, sufficiently small $\delta$, and $R \in \mathbb{N}$, we define a measure $\mu = \mu_{\delta,s,\theta,R}$ which we will use to apply the mass distribution principle. Recall that $K = K(\delta) = \lfloor\frac{-\log \delta}{\log n}\rfloor$. The measure will be defined by putting point masses on some carefully chosen level-$\lfloor K/\theta \rfloor$ approximate squares (corresponding to the very smallest scale $\delta^{1/\theta}$). 

If $k \in \mathbb{N}$ and $B_k$ is a level-$k$ approximate square in $\mathcal{B}_k$, we can choose a point $\Lambda_{B_k} \in \Lambda$ in the interior of $B_k$. We can make this choice explicitly by choosing the image of a distinguished point in $\Lambda$ inside the top-most (in the plane) level-$\gamma(k)$ cylinder within $B_k$. 
Let $\delta_{\Lambda_{B_k}}$ denote a unit point mass at $\Lambda_{B_k}$. 
For $l \in \mathbb{N}$ define 
\begin{equation}\label{eq:definejnumbers}
J_l \coloneqq \gamma^l(K)  -   \lfloor K/(\gamma^{L-l}\theta) \rfloor;  \qquad J_l' \coloneqq \lfloor K/(\gamma^{L-l}\theta) \rfloor - \gamma^{l-1}(K). 
\end{equation} %
Using notation from~\eqref{eq:definestjr}, define $C_{K,s,\theta,R}$ to be the set of level-$\lfloor K/\theta\rfloor$ approximate squares $(i_1,\dotsc,i_{\lfloor K/\theta\rfloor}, \ih_{\lfloor K/\theta\rfloor+1}, \dotsc, \ih_{\gamma(\lfloor K/\theta\rfloor)}) \in \mathcal{B}_{\lfloor K/\theta\rfloor}$ for which 
\begin{equation}\label{eq:definemass1}
(\ih_{\lfloor K/(\gamma^{L-l}\theta) \rfloor + 1},\dotsc,\ih_{\gamma^l(K)}) \in S_{t_{L-l}(s),J_l,R}  \quad \mbox{for} \quad l \in \{1,2,\dotsc,L-1\}.
\end{equation}
and
\begin{equation}\label{eq:definemass2}
(\ih_{\gamma^{l-1}(K) + 1},\dotsc,\ih_{\lfloor K/(\gamma^{L-l}\theta) \rfloor}) \in        S_{t_{L-l+1}(s),J_l',R} \quad \mbox{for} \quad l \in \{1,2,\dotsc,L\}.
\end{equation}
Note that when $\theta = \gamma^{-(L-1)}$ we do not impose the condition~\eqref{eq:definemass2}. 
Now we define the measure 
\begin{equation}\label{eq:definemu}
\mu = \mu_{\delta,s,\theta,R} \coloneqq \sum_{B_{\lfloor K/\theta \rfloor} \in C_{K,s,\theta,R}} n^{-Ks/\theta} \delta_{\Lambda_{B_{\lfloor K/\theta \rfloor}}}.
\end{equation}
This is clearly supported on $\Lambda$. 

For the remainder of this section, for 
\[ k \in \{ K, {K+1},\dotsc,{\lfloor K/\theta \rfloor} \}, \]
$B_k$ will denote an arbitrary level-$k$ approximate square in $\mathcal{B}_k \cap \mathrm{supp}(\mu)$. By the definition of the $S_{t,J,R}$ sets, $\mu(B_k)$ depends on $k,\delta,s,\theta,R$ and the carpet, but if $k = \gamma^l(K)$ or $k = \lfloor K/\gamma^j \theta \rfloor$ for some $j \in \{0,\dotsc,L-1\}$, then $\mu(B_k)=\mu(B'_k)$ for all $B_k,B'_k\in\mathcal{B}(k) \cap \mathrm{supp}(\mu)$.  

\begin{lemma}\label{lem:lowernicescales}
Fix $\theta \in (0,1)$, $s \in (\dim_{\mathrm{H}} \Lambda,\overline{\dim}_{\, \theta} \Lambda]$, and $R \in \mathbb{N}$ as above. %
For all $l = 0,1,\dotsc,L-1$, as $K \to \infty$, the following two asymptotic equalities hold: 
\begin{equation}\label{eq:nicescalesfirst} \mu_{\delta,s,\theta,R}(B_{\gamma^{l}(K)}) \asymp n^{-\gamma^l K s},
\end{equation}
\begin{equation}\label{eq:nicescaleslast}
\mu_{\delta,s,\theta,R}\left(B_{\left\lfloor \frac{K}{\theta \cdot \gamma^{L-l-1}} \right\rfloor} \right) \asymp n^{- \frac{Ks}{\theta \cdot \gamma^{L-l-1}}}.
\end{equation}
\end{lemma}

\begin{proof}%
The proof is an induction argument, starting with the smaller scales. Note that~\eqref{eq:nicescaleslast} holds for $l=L-1$ by the definition of $\mu$. We first use this to prove~\eqref{eq:nicescalesfirst} for $l=L-1$. 
Indeed, consider an approximate square $B_{\gamma^{L-1}(K)}(\mathbf{i}) \in \mathcal{B}_{\gamma^{L-1}(K)}$ and assume it intersects $\mathrm{supp}(\mu)$. 
Because of the way $\mu$ is defined, the mass of all such squares will be the same. 
To calculate this mass, we need to count up the number of level-$\lfloor K/\theta \rfloor$ approximate squares $B_{\lfloor K/\theta \rfloor}(\mathbf{j})$ which lie inside the level-$\gamma^{L-1}(K)$ square (so $B_{\lfloor K/\theta \rfloor}(\mathbf{j}) \subset B_{\gamma^{L-1}(K)}(\mathbf{i})$), and which also carry mass. 

To count the number of such smaller squares, it is helpful to compare the symbolic representation of any such square $B_{\lfloor K/\theta \rfloor}(\mathbf{j})$ with that of $B_{\gamma^{L-1}(K)}(\mathbf{i})$: 
\begin{align*}
B_{\gamma^{L-1}(K)}(\mathbf{i}):  ( i_1,\dotsc,i_{\gamma^{L-1}(K)}, \overbrace{  \ih_{\gamma^{L-1}(K) + 1},\dotsc,\ih_{\lfloor K/\theta \rfloor}  }^{ \in S_{t_1(s),J_L',R} }     &,\ih_{\lfloor K/\theta \rfloor + 1},\dotsc, \ih_{\gamma^L (K)}  ) \\*
B_{\lfloor K/\theta \rfloor}(\mathbf{j}): ( \underbrace{ j_1, \dotsc, j_{\gamma^{L-1}(K)} }_{\mbox{equal}}, \underbrace{  j_{\gamma^{L-1}(K) + 1},\dotsc,j_{\lfloor K/\theta \rfloor}}_{\mbox{same column}} &, \underbrace{ \jh_{\lfloor K/\theta \rfloor + 1},\dotsc, \jh_{\gamma^L(K)} }_{\mbox{equal}},\underbrace{  \jh_{\gamma^L(K) + 1},\dotsc, \jh_{\gamma(\lfloor K/\theta \rfloor)} }_{\in [M] \mbox{ freely}}) 
\end{align*}
For $t \in (\underline{t},\overline{t})$, $J \in \mathbb{N}$, $R \in \mathbb{N}$, let $\mathbf{p}(t,J,R) = (p_1(t,J,R),\dotsc,p_M(t,J,R))$ be the type class of every element of $S_{t,J,R}$. 
 Then $\mathbf{p}(t,J,R) \xrightarrow[J \to \infty]{} \mathbf{Q}^*_{t} \in \mathcal{E}_{t}$ by Lemma~\ref{lem:32}. Therefore 
 \begin{align*}
 \mu(B_{\gamma^{L-1}(K)}(\mathbf{i})) &\asymp  \prod_{\ih=1}^M N_{\ih}^{p_{\ih}(t_1(s),J_L',R) \cdot J_L'} M^{\gamma K/\theta - \gamma^L K} n^{-Ks/\theta} \\
 &= e^{t_1(s) J_L'  + K((\gamma/\theta - \gamma^L)\log M - (s\log n)/\theta)} \\
 &\asymp n^{-\gamma^{L-1} K s},
\end{align*}
using~\eqref{eq:definejnumbers} and the definition of $t_1(s)$ in~\eqref{eq:definetsequence}. Therefore~\eqref{eq:nicescalesfirst} holds for $l=L-1$. 

We now use this to prove~\eqref{eq:nicescaleslast} for $l=L-2$. If %
$B_{\gamma^{L-1}(K)}(\mathbf{j}) \subset B_{\lfloor K/(\gamma \theta)\rfloor}(\mathbf{i})$, %
then 
\begin{align*}
&(  i_1,\dotsc,i_{\lfloor K/(\gamma \theta)\rfloor}  , \overbrace{ \ih_{\lfloor K/(\gamma \theta)\rfloor + 1}, \dotsc, \ih_{\gamma^{L-1}(K)}  }^{\in S_{t_1(s),J_{L-1},R} } , \ih_{\gamma^{L-1}(K) + 1},\dotsc, \ih_{\lfloor K/\theta \rfloor}  )  \\*
&(  \underbrace{ j_1,\dotsc,j_{\lfloor K/(\gamma \theta)\rfloor} }_{\mbox{equal}}, \underbrace{ j_{\lfloor K/(\gamma \theta)\rfloor +1} ,\dotsc,  j_{\gamma^{L-1}(K)}   }_{\mbox{same column}}  , \underbrace{ \jh_{\gamma^{L-1}(K) + 1}, \dotsc, \jh_{\lfloor K/\theta \rfloor} }_{\mbox{equal}}, \underbrace{ \jh_{\lfloor K/\theta \rfloor + 1},\dotsc, \jh_{\gamma^L (K)} }_{\in [M] \mbox{ freely}} ).
\end{align*}
Therefore by Lemma~\ref{lem:32}, case $l=L-1$ of~\eqref{eq:nicescalesfirst},~\eqref{eq:definejnumbers}, and~\eqref{eq:definetsequence}, 
\begin{equation*}
\mu(B_{\lfloor K/(\gamma \theta)\rfloor}(\mathbf{i}))  \asymp e^{t_1(s)J_{L-1}} M^{\gamma^L K - K/\theta} n^{-\gamma^{L-1}Ks}  \asymp n^{-Ks/(\gamma \theta)}, 
\end{equation*}
so~\eqref{eq:nicescaleslast} holds for $l=L-2$. 

Now fix any $l \in \{0,1,\dotsc, L-2\}$ and assume that~\eqref{eq:nicescaleslast} holds for $l$. We show that this implies that~\eqref{eq:nicescalesfirst} holds for $l$. Indeed, if %
$B_{\lfloor \frac{K}{\theta \cdot \gamma^{L-l-1}} \rfloor}(\mathbf{j}) \subset B_{\gamma^l (K)}(\mathbf{i})$, %
then 
\begin{align*}
&(  i_1, \dotsc, i_{\gamma^l (K)} , \overbrace{  \ih_{\gamma^l (K) + 1}, \dotsc,  \ih_{\left\lfloor \frac{K}{\gamma^{L-l-1} \cdot \theta} \right\rfloor}  }^{ \in S_{t_{L-l}(s),J_{l+1}',R}} ,  \ih_{\left\lfloor \frac{K}{\gamma^{L-l-1}\cdot \theta} \right\rfloor + 1}, \dotsc, \ih_{\gamma^{l+1} (K)}  )  \\*
&( \underbrace{ j_1, \dotsc, j_{\gamma^l (K)} }_{\mbox{equal}}  ,  \underbrace{ j_{\gamma^l(K) + 1} , \dotsc, j_{\left\lfloor \frac{K}{\gamma^{L-l-1}\cdot \theta} \right\rfloor} }_{\mbox{same column}}  ,  \underbrace{ \jh_{\left\lfloor \frac{K}{\gamma^{L-l-1}\cdot \theta} \right\rfloor + 1} , \dotsc, \jh_{\gamma^{l+1}(K)} }_{\mbox{equal}}  , \underbrace{  \jh_{\gamma^{l+1}(K) + 1} , \dotsc,      \jh_{\left\lfloor \frac{K}{\gamma^{L-l-2} \cdot \theta} \right\rfloor}  }_{ \in S_{t_{L-l-1}(s),J_{l+2}',R}} ). 
\end{align*}
Therefore %
\begin{align*}
\mu(B_{\gamma^l(K)}(\mathbf{i})) &\asymp e^{t_{L-l}(s) J_{l+1}'} \cdot \#  S_{t_{L-l-1}(s),J_{l+2}',R} \cdot n^{-K s \gamma^{l+1-L}\theta^{-1}} &\substack{\text{by Lemma~\ref{lem:32}}\\
\text{and case } l \text{ of~\eqref{eq:nicescaleslast}}} \\
&\asymp e^{t_{L-l}(s) J_{l+1}'} \cdot e^{(\log M - I(t_{L-l-1}(s))) J_{l+2}'} n^{-K s \gamma^{l+1-L}\theta^{-1}}  &\text{by~\eqref{eq:lowerboundcard}}  \\*
&\asymp  n^{-\gamma^l K s} &\text{by~\eqref{eq:definejnumbers} and~\eqref{eq:definetsequence}}, 
\end{align*}
so indeed~\eqref{eq:nicescalesfirst} holds for $l$. 

Finally, fix any $l \in \{0,1,\dotsc, L-3\}$ %
and assume~\eqref{eq:nicescalesfirst} holds for $l+1$. We now show that this implies that~\eqref{eq:nicescaleslast} holds for $l$. Indeed, if %
$B_{\gamma^{l+1}(K)}(\mathbf{j}) \subset B_{\left\lfloor \frac{K}{\gamma^{L-l-1} \theta} \right\rfloor}(\mathbf{i})$, %
then 
\begin{align*}
&(  i_1,\dotsc, i_{\left\lfloor \frac{K}{\gamma^{L-l-1} \theta} \right\rfloor} , \overbrace{ \ih_{\left\lfloor \frac{K}{\gamma^{L-l-1} \theta} \right\rfloor + 1} , \dotsc, \ih_{\gamma^{l+1}(K)} }^{\in S_{t_{L-l-1}(s),J_{l+1},R}} , \ih_{\gamma^{l+1}(K) + 1} , \dotsc, \ih_{\left\lfloor \frac{K}{\gamma^{L-l-2} \theta} \right\rfloor }  )  \\*
&( \underbrace{ j_1,\dotsc, j_{\left\lfloor \frac{K}{\gamma^{L-l-1} \theta} \right\rfloor} }_{\mbox{equal}} , \underbrace{ j_{\left\lfloor \frac{K}{\gamma^{L-l-1} \theta} \right\rfloor + 1} , \dotsc, j_{\gamma^{l+1}(K)} }_{\mbox{same column}} , \underbrace{ \jh_{\gamma^{l+1}(K) + 1} , \dotsc, \jh_{\left\lfloor \frac{K}{\gamma^{L-l-2} \theta} \right\rfloor } }_{\mbox{equal}} ,\underbrace{ \jh_{\left\lfloor \frac{K}{\gamma^{L-l-2} \theta} \right\rfloor  + 1},\dotsc,  \jh_{\gamma^{l+2}(K) } }_{\in S_{t_{L-l-2}(s),J_{l+2},R}}  ). 
\end{align*}
As above, 
\begin{equation*}
\mu\left(B_{\left\lfloor \frac{K}{\gamma^{L-l-1} \theta} \right\rfloor}(\mathbf{i})\right) \asymp e^{t_{L-l-1}(s) J_{l+1}} \cdot e^{(\log M - I(t_{L-l-2}(s))) J_{l+2}} n^{-\gamma^{l+1} K s} \asymp n^{- \frac{Ks}{\gamma^{L-l-1} \theta}}, 
\end{equation*}
so indeed~\eqref{eq:nicescaleslast} holds for $l$. 
The proof is complete by induction. 
\end{proof}

In Lemma~\ref{lem:lowerallscales} we prove that if we make $R$ large enough then the mass is sufficiently evenly distributed for us to apply the mass distribution principle Proposition~\ref{prop:mdp} in Section~\ref{sec:proofconclusion}. 

\begin{lemma}\label{lem:lowerallscales}
Let $\theta \in (0,1)$ and $s \in (\dim_{\mathrm{H}} \Lambda,\overline{\dim}_{\, \theta} \Lambda]$. For all $s'<s$ there exists $\delta_0 \in (0,1)$ and $R \in \mathbb{N}$ depending on $s,s',\theta$ and the carpet such that for all $\delta \in (0,\delta_0)$ and $k \in \{K,K+1,\dotsc,\lfloor K/\theta \rfloor \}$, if $b_k$ is any level-$k$ approximate square then $\mu_{\delta,s,\theta,R}(b_k) < n^{-ks'}$. 
\end{lemma}
\begin{proof}
Fix $\theta \in (0,1)$, $s \in (\dim_{\mathrm{H}} \Lambda,\overline{\dim}_{\, \theta} \Lambda]$ and $R \in \mathbb{N}$. 
The idea is that for each scale $k$, we will choose from the finitely many scales considered in Lemma~\ref{lem:lowernicescales} the one which corresponds to the largest size that is smaller than $n^{-k}$. We will then bound the number of approximate squares of this level which are contained in each level-$k$ approximate square which carries mass, and use Lemma~\ref{lem:lowernicescales} to bound the mass of the level-$k$ approximate square. 

Let $J' \in \mathbb{N}$ be large enough that for each $t \in \{t_1(s),\dotsc,t_{L-1}(s)\}$ and $(s_t,t)$ related by~\eqref{eq:23}, \eqref{eq:lowerboundcombinatorialfudge} holds for all $J \geq J'$ and $k' \in \{1,\dotsc,J\}$. 
By Lemma~\ref{lem:32}, we may increase $J'$ to assume further that if $\iiv \in S_{t,J,R}$ then $\prod_{j=1}^J N_{\ih_j} \leq e^{(t+1/R)J}$. Then 
\begin{align}\label{eq:lastpartproduct}
\begin{split}
\prod_{j=k'}^{J} N_{\ih_j} = \frac{\prod_{j=1}^{J} N_{\ih_j}}{\prod_{j'=1}^{k'-1} N_{\ih_{j'}}} \leq \frac{e^{(t+1/R)J}}{\psi_{(\ih_1,\dotsc,\ih_{k'}) | k'}(s_t) n^{k's_t}M^{-k'\gamma}} &\leq e^{(1+3\overline{t})J/R} e^{tJ}n^{-k's_t} M^{\gamma k'} \\*
&= e^{(1+3\overline{t})J/R  +    (J-k')t}. 
\end{split}
\end{align}
We may increase $J'$ to ensure that for all $J \geq J'$ and $k' \in \{1,\dotsc,J\}$, letting $l' \in \{0,1,\dotsc,R-1\}$ be such that $\lfloor l'J/R \rfloor < k' \leq \lfloor (l'+1)J/R \rfloor $, if $(\jh_1,\dotsc,\jh_{k'}) \in [M]^{k'}$ then 
\begin{align}\label{eq:boundrightpart}
\begin{split}
 \# \{ \, (\ih_1,\dotsc,\ih_J) \in S_{t,J,R} : \ih_p = \jh_p \mbox{ for } p \in \{1,\dotsc,k'\} \, \} &\leq e^{(J-\lfloor l'J/R \rfloor) (H(\mathbf{Q}^*_t) + 1/R)} \\*
&\leq e^{(J-k') (\log M - I(t)) + 3J (1+H(\mathbf{Q}^*_t))/R}.
\end{split}
\end{align}
Let $\delta_0>0$ be small enough that for all $\delta \in (0,\delta_0)$, $J' < J_1 \leq J_2 \leq J_3 \leq \dotsb$ and, if $\theta \neq \gamma^{-(L-1)}$, $J' < J_1' \leq J_2' \leq J_3' \leq \dotsb$. By decreasing $\delta_0$ further we may assume by Lemma~\ref{lem:lowernicescales} that for all $l \in \{0,1,\dotsc,L-1\}$, 
\begin{equation}\label{eq:massonesquare}
\mu_{\delta,s,\theta,R}(B_{\gamma^{l}(K)}) \leq n^{\gamma^l K (-s+1/R)} \qquad \mbox{and} \qquad \mu_{\delta,s,\theta,R}\left(B_{\left\lfloor \frac{K}{\gamma^{L-l-1} \theta} \right\rfloor} \right) \leq n^{ \frac{K}{\gamma^{L-l-1} \theta}(-s+1/R)}.
\end{equation}

We now consider symbolic representations of approximate squares in a similar way to Lemma~\ref{lem:lowernicescales}. 

\emph{Case 1:} 
Suppose $l \in \{0,1,\dotsc,L-2\}$ and 
\[ k \in \{{\gamma^l(K)+1},{\gamma^{l}(K)+2},\dotsc, {\lfloor K \gamma^{l+1-L}\theta^{-1} \rfloor - 1}\}.\] 
The symbolic representations of approximate squares $B_{\lfloor K \gamma^{l+1-L}\theta^{-1} \rfloor} (\mathbf{j}) \subset B_{k}(\mathbf{i})$ are as follows (broken onto two lines because they do not fit onto one line): 
\begin{align*}
&(i_1,\dotsc, i_{\gamma^l(K)}, \overbrace{ i_{\gamma^l(K)+1},\dotsc,i_{k}  , \ih_{k+1},\dotsc,\ih_{\left\lfloor \frac{K}{\gamma^{L-l-1}\theta} \right\rfloor } }^{\in S_{t_{L-l}(s),J_{l+1}',R}} , \\*
&(\underbrace{ j_1,\dotsc,j_{\gamma^l(K)},j_{\gamma^l(K)+1},\dotsc,j_k}_{\mbox{equal}}, \underbrace{ j_{k+1},\dotsc,j_{\left\lfloor \frac{K}{\gamma^{L-l-1}\theta} \right\rfloor}}_{\mbox{same column}} ,
\end{align*}
and continuing
\begin{align*}
&\ih_{\left\lfloor \frac{K }{\gamma^{L-l-1}\theta} \right\rfloor + 1},\dotsc,\ih_{\gamma^{l+1}(K)},\ih_{\gamma^{l+1}(K)+1},\dotsc, \ih_{\gamma(k)}  ) \\*
& \rlap{\ensuremath{\underbrace{ \phantom{\jh_{\left\lfloor \frac{K }{\gamma^{L-l-1}\theta} \right\rfloor + 1} ,\dotsc, \jh_{\gamma^{l+1}(K)},  \jh_{\gamma^{l+1}(K)+1},\dotsc, \jh_{\gamma(k)} }}_{\mbox{equal}}}}  \jh_{\left\lfloor \frac{K }{\gamma^{L-l-1}\theta} \right\rfloor + 1} ,\dotsc,  \jh_{\gamma^{l+1}(K)}, \underbrace{ \vphantom{\frac{a}{\frac{a}{\frac{b}{c}}}}  \jh_{\gamma^{l+1}(K)+1},\dotsc, \jh_{\gamma(k)} , \jh_{\gamma(k)+1} ,\dotsc, \jh_{\left\lfloor \frac{K}{\gamma^{L-l-2}\theta} \right\rfloor}}_{ \in S_{t_{L-l-1}(s),J_{l+2}',R}} )  .
\end{align*}
Therefore we can bound the mass 
\begin{align}
\mu(B_k(\mathbf{i})) &\leq C \prod_{y=k+1}^{\left\lfloor \frac{K}{\gamma^{L-l-1}\theta} \right\rfloor} N_{\ih_y}  e^{ \left(\frac{K}{\gamma^{L-l-2}\theta} - \gamma k \right)(\log M - I(t_{L-l-1}(s)))  +  3J_{l+2}'(1+H(\mathbf{Q}^*_{t_{L-l-1}(s)}))/R} \nonumber \\*
&\phantom{\leq}\times n^{\frac{K}{\gamma^{L-l-1}\theta}(-s+1/R)} \label{eq:allscales1} \\
&\leq C e^{\left(\frac{K}{\gamma^{L-l-1}\theta} - k\right)t_{L-l}(s) + (1+3\overline{t})J_{l+1}'/R} e^{ \left(\frac{K}{\gamma^{L-l-2}\theta} - \gamma k \right)(\log M - I(t_{L-l-1}(s)))} \nonumber \\*
 &\phantom{\leq}\times e^{3J_{l+2}'(1+H(\mathbf{Q}^*_{t_{L-l-1}(s)}))/R}  n^{\frac{K}{\gamma^{L-l-1}\theta}(-s+1/R)} \label{eq:allscales2} \\
&\leq C n^{-ks + \big((1+3\overline{t})J_{l+1}'/(\log n) +3J_{l+2}'(1+H(\mathbf{Q}^*_{t_{L-l-1}(s)}))/(\log n) + \frac{K}{\gamma^{L-l-1}\theta}  \big)/R } \label{eq:allscales3} \\
&\leq C n^{\Big(-s + \frac{(1+3\overline{t}) + 3(1+H(\mathbf{Q}^*_{t_{L-l-1}(s)})) + \log n}{R\theta \log n}     \Big)k}, \label{eq:allscales4}
\end{align}
where $C$ is a constant depending only on the carpet,  %
~\eqref{eq:allscales1} is by~\eqref{eq:boundrightpart} and~\eqref{eq:massonesquare}; \eqref{eq:allscales2} is by~\eqref{eq:lastpartproduct}; \eqref{eq:allscales3} is by~\eqref{eq:definetsequence} and algebraic manipulations; and~\eqref{eq:allscales4} is since 
\[ \max\left\{ J_{l+1}' ,J_{l+2}',\frac{K}{\gamma^{L-l-1}\theta} \right\} \leq K/\theta < k/\theta.\] 

\emph{Case 2:} Suppose $l \in \{0,1,\dotsc,L-3\}$ and $k \in \{\lfloor K \gamma^{l+1-L}\theta^{-1} \rfloor + 1,\dotsc,\gamma^{l+1}(K) - 1\}$. If $B_{\gamma^{l+1}(K)}(\mathbf{j}) \subset B_k(\mathbf{i})$, then 
\begin{align*}
&(i_1,\dotsc,i_{\left\lfloor \frac{K }{\gamma^{L-l-1}\theta} \right\rfloor }, \overbrace{ i_{\left\lfloor \frac{K }{\gamma^{L-l-1}\theta} \right\rfloor + 1},\dotsc, i_{k},\ih_{k+1},\dotsc,\ih_{\gamma^{l+1}(K)} }^{\in S_{t_{L-l-1}(s),J_{l+1},R}} , \\*
&(\underbrace{ j_1,\dotsc,j_{\left\lfloor \frac{K }{\gamma^{L-l-1}\theta} \right\rfloor },j_{\left\lfloor \frac{K }{\gamma^{L-l-1}\theta} \right\rfloor + 1},\dotsc, j_{k} }_{\mbox{equal}},\underbrace{ j_{k+1},\dotsc,j_{\gamma^{l+1}(K)} }_{\mbox{same column}} , 
\end{align*}
continuing 
\begin{align*}
&\ih_{\gamma^{l+1}(K) + 1},\dotsc,\ih_{\left\lfloor \frac{K }{\gamma^{L-l-2}\theta} \right\rfloor },\ih_{\left\lfloor \frac{K }{\gamma^{L-l-2}\theta} \right\rfloor  + 1},\dotsc,\ih_{\gamma(k)} ) \\*
& \rlap{\ensuremath{\underbrace{ \phantom{ \jh_{\gamma^{l+1}(K) + 1},\dotsc,\jh_{\left\lfloor \frac{K }{\gamma^{L-l-2}\theta} \right\rfloor },\jh_{\left\lfloor \frac{K }{\gamma^{L-l-2}\theta} \right\rfloor  + 1},\dotsc,\jh_{\gamma(k)} }}_{\mbox{equal}}}} \jh_{\gamma^{l+1}(K) + 1},\dotsc,\jh_{\left\lfloor \frac{K }{\gamma^{L-l-2}\theta} \right\rfloor },\underbrace{\vphantom{\frac{a}{\frac{a}{\frac{b}{c}}}}  \jh_{\left\lfloor \frac{K }{\gamma^{L-l-2}\theta} \right\rfloor  + 1},\dotsc,\jh_{\gamma(k)},\jh_{\gamma(k) + 1},\dotsc,\jh_{\gamma^{l+2}(K)} }_{\in S_{t_{L-l-2}(s),J_{l+2},R}}). 
\end{align*}
Therefore there is a constant $C>0$ such that 
\begin{align*}
\mu(B_k(\mathbf{i})) &\leq C e^{(\gamma^{l+1}K - k) t_{L-l-1}(s) + (1+3\overline{t})J_{l+1}/R }  e^{(\gamma^{l+2}K - \gamma k)(\log M - I(t_{L-l-2}(s)))} \\*
&\phantom{\leq}\times e^{3 J_{l+2} (1 + H(\mathbf{Q}^*_{t_{L-l-2}(s)}))/R } n^{\gamma^{l+1}K(-s+1/R)} \\
&\leq C n^{\Big(-s + \frac{(1+3\overline{t}) + 3(1+H(\mathbf{Q}^*_{t_{L-l-2}(s)})) + \log n}{R\theta \log n}     \Big)k}.  
\end{align*}

\emph{Case 3:} If $k \in \{ \lfloor K/(\gamma\theta) \rfloor + 1,\dotsc, \gamma^{L-1}(K) - 1\}$ and $B_{\gamma^{L-1}(K)}(\mathbf{j}) \subset B_{k}(\mathbf{i})$, then 
\begin{align*}
&( i_1,\dotsc,i_{\left\lfloor \frac{K}{\gamma\theta}\right\rfloor}, \overbrace{ i_{\left\lfloor \frac{K}{\gamma\theta}\right\rfloor + 1} , \dotsc, i_k, \ih_{k+1},\dotsc,\ih_{\gamma^{L-1}(K)} }^{\in S_{t_1(s),J_{L-1},R}}, \\*
&( \underbrace{ j_1,\dotsc,j_{\left\lfloor \frac{K}{\gamma\theta}\right\rfloor},j_{\left\lfloor \frac{K}{\gamma\theta}\right\rfloor + 1} , \dotsc, j_k }_{\mbox{equal}},\underbrace{ j_{k+1},\dotsc,j_{\gamma^{L-1}(K)} }_{\mbox{same column}},
\end{align*}
continuing
\begin{align*}
&\ih_{\gamma^{L-1}(K)+1},\dotsc,\ih_{\lfloor K/\theta\rfloor},\ih_{\lfloor K/\theta\rfloor + 1},\dotsc,\ih_{\gamma(k)} ) \\*
&\underbrace{ \jh_{\gamma^{L-1}(K)+1},\dotsc,\jh_{\lfloor K/\theta\rfloor},\jh_{\lfloor K/\theta\rfloor + 1},\dotsc,\jh_{\gamma(k)}}_{\mbox{equal}},\underbrace{ \jh_{\gamma(k)+1},\dotsc,\jh_{\gamma^L(K)} }_{\in [M] \mbox{ freely}} ). 
\end{align*}
Therefore by the definition of $t_1(s)$ in~\eqref{eq:definetsequence} there is a constant $C>0$ such that 
\begin{equation*}
\mu(B_k(\mathbf{i})) \leq C e^{(\gamma^{L-1}K-k)t_1(s) + (1+3\overline{t})J_{L-1}/R} M^{\gamma^L K - \gamma k} n^{\gamma^{L-1} K(-s+1/R)} \leq C n^{\left( -s + \frac{1+3\overline{t} + \log n}{R\theta \log n} \right) k}. 
\end{equation*}

\emph{Case 4:} Finally, if $k \in \gamma^{L-1}(K) + 1,\dotsc , \lfloor K/\theta \rfloor - 1$ and $B_{\lfloor K/\theta \rfloor}(\mathbf{j}) \subset B_k(\mathbf{i})$ then 
\begin{align*}
&(i_1,\dotsc,i_{\gamma^{L-1}(K)}, \overbrace{ i_{\gamma^{L-1}(K)+1},\dotsc,i_k,\ih_{k+1},\dotsc,\ih_{\lfloor K/\theta \rfloor} }^{\in S_{t_1(s),J_L',R}} \\*
&(\underbrace{ j_1,\dotsc,j_{\gamma^{L-1}(K)},j_{\gamma^{L-1}(K)+1},\dotsc,j_k }_{\mbox{equal}} ,\underbrace{ j_{k+1} ,\dotsc , j_{\lfloor K/\theta \rfloor} }_{\mbox{same column}},
\end{align*}
continuing 
\begin{align*}
&\ih_{\lfloor K/\theta \rfloor + 1},\dotsc,\ih_{\gamma^L(K)},\ih_{\gamma^L(K) + 1},\dotsc, \ih_{\gamma(k)} ) \\*
&\underbrace{ \jh_{\lfloor K/\theta \rfloor + 1},\dotsc,\jh_{\gamma^L(K)},\jh_{\gamma^L(K) + 1},\dotsc, \jh_{\gamma(k)} }_{\mbox{equal}},\underbrace{ \jh_{\gamma(k)+1},\dotsc,\jh_{\gamma(K/\theta)} }_{\in [M] \mbox{ freely}} ). 
\end{align*}
Therefore by the definition of $t_1(s)$ there exists a constant $C>0$ such that 
\begin{equation*}
\mu(B_k(\mathbf{i})) \leq C e^{(K/\theta - k)t_1(s) + (1+3\overline{t})J_L'/R} M^{\gamma K/\theta - \gamma k} n^{-Ks/\theta} \leq C n^{( -s+ (1+3\overline{t})/R ) k}.
\end{equation*}
Therefore the result follows (using Lemma~\ref{lem:lowernicescales} if $k \in \{K, \gamma(K),\dotsc,\gamma^{L-1}(K)\}$) if we take $R$ large enough depending on $s,s',\theta,C$ and the carpet. 
\end{proof} 

We write 
\begin{equation}\label{eq:definemainquantity} 
G(\theta,s) \coloneqq \gamma^L\theta \log N - (\gamma^L\theta - 1) t_L(s) + \gamma(1-\gamma^{L-1}\theta)(\log M - I(t_L(s))) - s\log n.
\end{equation}
\begin{lemma}\label{lem:lowertotalmass}
Fix $\theta \in (0,1)$, $s \in (\dim_{\mathrm{H}} \Lambda,\overline{\dim}_{\gamma^{-(L-1)}} \Lambda]$, and $R \in \mathbb{N}$. The total mass 
\[ \mu_{\delta,s,\theta,R}(\Lambda) \asymp e^{K \cdot G(\theta,s) / (\gamma^L \theta)} \mbox{ as } K \to \infty. \]
\end{lemma}
\begin{proof}
The symbolic representation of a level-$K$ approximate square $B_K(\mathbf{i}) \in \mathcal{B}_{K} \cap \mathrm{supp}(\mu)$ is 
\begin{equation*}
( \underbrace{ i_1,\dotsc, i_K }_{\in [N] \mbox{ freely}}, \underbrace{ \ih_{K+1},\dotsc, \ih_{\lfloor K \gamma^{-(L-1)}\theta^{-1}\rfloor} }_{\in S_{t_L(s),J_1',R}} , \underbrace{ \ih_{\lfloor K \gamma^{-(L-1)}\theta^{-1}\rfloor + 1},\dotsc, \ih_{\gamma(K)} }_{\in S_{t_{L-1}(s),J_1,R}} ).
\end{equation*}
Therefore 
\begin{align*}
 \mu(\Lambda) &\asymp \# (\mathcal{B}_{K} \cap \mathrm{supp}(\mu)) \cdot n^{-Ks} &\text{by case } l=0 \text{ of~\eqref{eq:nicescalesfirst}} \\
 &\asymp N^K e^{(\log M - I(t_L(s)))J_1'} e^{(\log M - I(t_{L-1}(s)))J_1} n^{-Ks} &\text{by~\eqref{eq:lowerboundcard}} \\ 
 &\asymp e^{K \cdot G(\theta,s) / (\gamma^L \theta)} &\text{by~\eqref{eq:definejnumbers},\eqref{eq:definetsequence},\eqref{eq:definemainquantity}},
 \end{align*}
 completing the proof. 
\end{proof}
We have now proved enough to give Theorem~\ref{thm:main} in the case when~$\theta$ is a negative integer power of~$\gamma$, see Section~\ref{sec:proofconclusion}. 
 
 \subsection{Upper bound for general \texorpdfstring{$\theta$}{θ}}\label{sec:mainupper}%
 Suppose $L \in \mathbb{N}$, $\theta \in (\gamma^{-L},\gamma^{-(L-1)})$, $s \in (\dim_{\mathrm H} \Lambda,\overline{\dim}_{\gamma^{-(L-1)}} \Lambda]$ and $0 < \delta \ll 1$. We define a cover $\{V_j\}_j$ of $\Lambda$ (depending on $\theta$, $s$ and $\delta$) as follows. Every level-$K$ cylinder will be covered in the same way, and the cover will consist of approximate squares $(i_1,\dotsc,i_k,\ih_{k+1},\dotsc,\ih_{\gamma(k)})$ of different levels $k \in \{ K, K+1,\dotsc, \lfloor K/\theta \rfloor\}$. This means that the diameter of each element of the cover will, up to an irrelevant multiplicative constant depending only on the carpet, lie in the interval $[\delta^{1/\theta},\delta]$. In fact, we will use only the scales $\gamma^l(K)$ and $\lfloor K/(\gamma^l \theta)\rfloor$ for $l \in \{0,1,2,\dotsc,L-1\}$. 
Figure~\ref{fig:CoverHelp} provides a diagram which may help the reader follow the construction of the cover. 
\begin{figure}[th]
\center{\includegraphics[width=.99\textwidth]
        {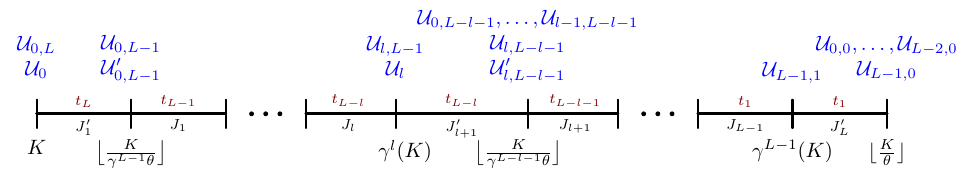}}
        \caption{Visualising the cover in~\eqref{eq:definerealcover} for $L \geq 3$. 
        Here, $l$ denotes an arbitrary number in $\{1,2,\dotsc,L-2\}$. 
The indices of the symbolic representation and the lengths of the different parts are in black. Above the scales explicitly written out are the sets (in blue) which make up the part of the cover consisting of approximate squares of the corresponding level. The `critical' thresholds $t_i$ for the averages of the different parts of the symbolic representation are in red. Recall that the $t_i$ depend on $s$, and the sets that make up the cover depend on $s$ and $\theta$.}
\label{fig:CoverHelp}
\end{figure}

Recall that $L \coloneqq 1 + \lfloor \frac{-\log \theta}{\log \gamma} \rfloor$, and $\iih = (\ih_1,\dotsc,\ih_k,\ih_{k+1},\dotsc,\ih_{\gamma(k)})$, and we use the notation from~\eqref{eq:tauaverage}. 
We define $\mathcal{U}_{L-1,1}$ to be the set of level-$\gamma^{L-1}(K)$ approximate squares for which~\eqref{eq:secondthreshold} and~\eqref{eq:firstthreshold} below hold for all $j \in \{ 1,2,\dotsc,L-1\}$, %
and~\eqref{eq:lastthreshold} holds: 
\begin{align}
	\tau(\iih,\lfloor K/(\gamma^{L-j}\theta)\rfloor,\gamma^j(K))   &<    t_{L-j}(s); \label{eq:secondthreshold} \\
	\tau(\iih,\gamma^{j-1}(K),\lfloor K/(\gamma^{L-j}\theta)\rfloor )  &<   \frac{\gamma^j  -  (\gamma^{L-j}\theta)^{-1}}{(\gamma^{L-j}\theta)^{-1}  -  \gamma^{j-1}}  \big( t_{L-j}(s) - \tau(\iih, \lfloor K/(\gamma^{L-j}\theta)\rfloor , \gamma^j(K)) \big) \nonumber \\*
	&\phantom{-}+ t_{L-j+1}(s) ; \label{eq:firstthreshold} \\
	\tau(\iih,\gamma^{L-1}(K),\lfloor K/\theta \rfloor)    &\geq   t_1(s).  \label{eq:lastthreshold}
\end{align}
Note that when defining $\mathcal{U}_{L-1,1}$ we imposed no restriction on $\ih_{\lfloor K/\theta \rfloor + 1},\dotsc,\ih_{\gamma^L(K)}$, or on $i_1,\dotsc,i_{K-1}$. 
Define $\mathcal{U}_{L-1,0}$ to be the set of level-$\lfloor K/\theta \rfloor$ approximate squares for which~\eqref{eq:secondthreshold} and~\eqref{eq:firstthreshold} hold for all $j \in \{ 1,2,\dotsc,L-1 \}$, and~\eqref{eq:lastthreshold} does \emph{not} hold (no restriction on $\ih_{\lfloor K/\theta \rfloor + 1},\dotsc,\ih_{\gamma(\lfloor K/\theta \rfloor)}$). 
If $L=1$ then our cover is simply $\mathcal{U}_{0,0} \cup \mathcal{U}_{0,1}$, so for the rest of the construction of the cover we assume that $L > 1$.

For $l = 0,1,\dotsc,L-2$ we define $\mathcal{U}_l$ to be the set of level-$\gamma^l(K)$ approximate squares which satisfy condition~\eqref{eq:secondthreshold} for all $j \in \{1,2,\dotsc,l+1\}$, and which satisfy~\eqref{eq:firstthreshold} for all $j \in \{1,2,\dotsc, l\}$ %
 but do \emph{not} satisfy~\eqref{eq:firstthreshold} for $j=l+1$. 
For $l = 0,1,\dotsc,L-2$ we define $\mathcal{U}_{l,L-l}$ to be the set of level-$\gamma^l(K)$ approximate squares for which 
\eqref{eq:secondthreshold} holds for all $j \in \{ 1,2,\dotsc,l\}$ %
but not for $j=l+1$, and~\eqref{eq:firstthreshold} holds for all $j \in \{ 1,2,\dotsc,l\}$, %
and~\eqref{eq:tsaverage} holds: 
\begin{align}\label{eq:tsaverage}
\begin{split}
\tau(\iih,\gamma^l(K),&\lfloor K/(\gamma^{L-l-1}\theta) \rfloor) \geq \\*
&T_s \Big(  \max\Big\{ \underline{t},  \frac{1}{(\gamma^{L-l-2}\theta)^{-1}  -  \gamma^{l+1}} \Big( \Big(  \frac{1}{\gamma^{L-l-2}\theta}  - \frac{1}{\gamma^{L-l-1} \theta} \Big)  t_{L-l-1}(s)  \\*
&\phantom{T_s \;}-   \big(\gamma^{l+1}  -  (\gamma^{L-l-1}\theta)^{-1}  \big) \tau(\iih,\lfloor K/(\gamma^{L-l-1}\theta) \rfloor  , \gamma^{l+1}(K))   \Big)    \Big\} \Big).
\end{split}
\end{align}%

For $l = 0,1,\dotsc,L-2$ define $\mathcal{U}_{l,L-l-1}'$ to be the set of level-$\lfloor K/(\gamma^{L-l-1}\theta) \rfloor$ approximate squares for which~\eqref{eq:secondthreshold} and~\eqref{eq:firstthreshold} hold for all $j \in \{1,2,\dotsc,l\}$, %
 and~\eqref{eq:tsaverage} does not hold, and~\eqref{eq:secondpartverybig} holds: 
 \begin{align}\label{eq:secondpartverybig}
 \begin{split}
\tau(\iih,&\lfloor K/(\gamma^{L-l-1}\theta) \rfloor  , \gamma^{l+1}(K)) \geq \\* 
&\frac{1}{\gamma^{l+1}  -  (\gamma^{L-l-1}\theta)^{-1}  }   \Big(  \Big(  \frac{1}{\gamma^{L-l-2}\theta}  - \frac{1}{\gamma^{L-l-1} \theta} \Big)  t_{L-l-1}(s)  
-   ((\gamma^{L-l-2}\theta)^{-1}  -  \gamma^{l+1}) \underline{t}  \Big) .
\end{split}
 \end{align}
 Note that~\eqref{eq:secondpartverybig} means that~\eqref{eq:secondthreshold} does not hold for $j=l+1$, and that the maximum in~\eqref{eq:tsaverage} equals $\underline{t}$ (since $t_{L-l-1}(s) > \underline{t}$). 
 Note also that no restriction was imposed upon $(\ih_{\gamma^{l+1}(K) + 1},\dotsc,\ih_{\lfloor K/(\gamma^{L-l-2}\theta))\rfloor})$. %
For $l = 0,1,\dotsc,L-2$ define $\mathcal{U}_{l,L-l-1}$ to be the set of level-$\lfloor K/(\gamma^{L-l-1}\theta) \rfloor$ approximate squares for which~\eqref{eq:secondthreshold} holds for all $j \in \{1,2,\dotsc,l\}$ %
 but not for $j=l+1$, and~\eqref{eq:firstthreshold} holds for all $j \in \{1,2,\dotsc,l\}$, %
 and~\eqref{eq:tsaverage} does not hold, and~\eqref{eq:secondpartverybig} does not hold, and~\eqref{eq:thirdthreshold} holds: 
 \begin{align}\label{eq:thirdthreshold}
\begin{split} 
 \tau(\iih ,    \gamma^{l+1}(K) , \lfloor K/(\gamma^{L-l-2}\theta) \rfloor ) &\geq  \frac{1}{(\gamma^{L-l-2}\theta)^{-1}  -  \gamma^{l+1}} \Big( \Big(  \frac{1}{\gamma^{L-l-2}\theta}  - \frac{1}{\gamma^{L-l-1} \theta} \Big)  t_{L-l-1}(s)  \\*
&\phantom{\geq}-   (\gamma^{l+1}  -  (\gamma^{L-l-1}\theta)^{-1}  ) \tau(\iih,\lfloor K/(\gamma^{L-l-1}\theta) \rfloor  , \gamma^{l+1}(K))   \Big) .
 \end{split}
\end{align}

For $l = 0,1,\dotsc,L-2$ define $\mathcal{U}_{l,0}$ to be the set of level-$\lfloor K/\theta \rfloor$ approximate squares for which~\eqref{eq:secondthreshold} holds for all $j \in \{1,2,\dotsc,l\}$ %
but not for $j=l+1$, and~\eqref{eq:firstthreshold} holds for all $j \in \{1,2,\dotsc,l\}$, %
 and~\eqref{eq:tsaverage} does not hold, and~\eqref{eq:secondpartverybig} does not hold, and~\eqref{eq:thirdthreshold} does not hold, and~\eqref{eq:lowerthresholds} holds for all $j \in \{ 1,2,\dotsc,L-l-2\}$: 
\begin{equation}\label{eq:lowerthresholds}
 \tau(\iih,  \lfloor K/(\gamma^{j}\theta) \rfloor, \lfloor K/(\gamma^{j-1}\theta) \rfloor )   <    t_j(s).
\end{equation}
Note that we imposed no restriction on $\ih_{\lfloor K/\theta \rfloor + 1},\dotsc,\ih_{\gamma(\lfloor K/\theta \rfloor)}$, and in the case $l=L-2$ we did not require the extra condition~\eqref{eq:lowerthresholds}. %
If $L=2$ then we have constructed the cover 
\[ \Lambda \subseteq \mathcal{U}_0 \cup \mathcal{U}_{0,0} \cup \mathcal{U}_{0,1}\cup \mathcal{U}_{0,1}'  \cup  \mathcal{U}_{0,2}\cup \mathcal{U}_{1,0} \cup \mathcal{U}_{1,1}.\] 

If $L >2$, then for $l = 0,1,\dotsc,L-3$ and $k = 1,2,\dotsc,L-l-2$
define $\mathcal{U}_{l,k}$ to be the set of level-$\lfloor K/(\gamma^k \theta) \rfloor$ approximate squares for which~\eqref{eq:secondthreshold} holds for all $j \in \{1,2,\dotsc,l\}$ %
 but not for $j=l+1$, and~\eqref{eq:firstthreshold} holds for all $j \in \{1,2,\dotsc,l\}$, %
 and~\eqref{eq:tsaverage} does not hold, and~\eqref{eq:secondpartverybig} does not hold, and~\eqref{eq:thirdthreshold} does not hold, and~\eqref{eq:lowerthresholds} holds for all $j \in \{ k+1,k+2,\dotsc, L-l-2 \}$ %
 but not for $j=k$. 
We have finally constructed a cover of~$\Lambda$: 
\begin{equation}\label{eq:definerealcover}
\Lambda \subseteq \mathcal{U}_{L-1,0} \cup \mathcal{U}_{L-1,1}  \cup   \bigcup_{l = 0}^{L-2} \bigg(  \mathcal{U}_l \cup  \mathcal{U}_{l,L-l-1}' \cup \bigcup_{k=0}^{L-l} \mathcal{U}_{l,k}  \bigg). 
\end{equation}
 For simplicity, we denote the cover by $\{ V_j\}_j$. Observe that any two elements of this cover are either disjoint or intersect on their boundary; it can never happen that one is contained within the other. 
 Figure~\ref{fig:CoverHelp} depicts the different parts of the cover in the most complicated case, namely when $\gamma^{-L} < \theta < \gamma^{-(L-1)}$ for some natural number $L \geq 3$. 
 
 We will bound the $s$-cost of this cover in Lemma~\ref{lem:mainuppercost}. For this, we need Lemmas~\ref{lem:simpleindcomb},~\ref{lem:inductioncombinatorics} and~\ref{lem:tripleindcomb}, which we prove using the method of types. 
The inequalities in Lemma~\ref{lem:inductioncombinatorics} mimic~\eqref{eq:secondthreshold} and~\eqref{eq:firstthreshold}. 

\begin{lemma}\label{lem:inductioncombinatorics}
Suppose $c \in (0,1)$, $\overline{t} > t_1 > t_2 > \underline{t}$ %
and $J \in \mathbb{N}$. Then as $J \to \infty$, 
\begin{enumerate}

\item\label{eq:inductioncombinatorics1}
\begin{align*}
\# &\{ \, \iih \in [M]^J : \tau(\iih,\lfloor cJ \rfloor ,J) \leq t_2, \, \tau(\iih,0,\lfloor cJ \rfloor) \geq t_1 + ((1-c)/c) ( t_2 - \tau(\iih,\lfloor cJ \rfloor,J) ) \, \} \\*
&\asymp e^{J(c(\log M - I(t_1)) + (1-c)(\log M - I(t_2)))};
\end{align*}

\item\label{eq:inductioncombinatorics2}
\begin{align*}
\# &\{ \, \mathbf{i} \in [N]^{J} : \tau(\iih,\lfloor cJ \rfloor ,J) \leq t_2, \, \tau(\iih,0,\lfloor cJ \rfloor) \leq t_1 + ((1-c)/c) ( t_2 - \tau(\iih,\lfloor cJ \rfloor,J) ) \, \} \\*
&\asymp e^{J (c(t_1 + \log M - I(t_1)) + (1-c)(t_2 + \log M - I(t_2)))}. 
\end{align*}

\end{enumerate}
\end{lemma}

\begin{proof}
The lower bounds for the asymptotic growth follow from considering those strings for which $\ih_{\lfloor cJ \rfloor + 1}, \dotsc, \ih_J$ and $\ih_{1},\dotsc,\ih_{\lfloor cJ \rfloor}$ are the best approximations to $\mathbf{Q}^*_{t}$ and $\mathbf{Q}^*_{T_s(t)}$ respectively in $\mathcal{T}_{J-\lfloor cJ \rfloor}$ and $\mathcal{T}_{\lfloor cJ \rfloor}$ for which the required inequalities hold. 
The strategy for the upper bounds is to fix arbitrary type classes for the different parts of the string which satisfy the desired inequalities and then use the fact that there are only polynomially many type classes. 

\eqref{eq:inductioncombinatorics1} 
Fix $\mathbf{p} \in \mathcal{T}_{\lfloor cJ \rfloor}$ and $\mathbf{q} \in \mathcal{T}_{J - \lfloor cJ \rfloor}$ such that $t_{\mathbf{q}} \leq t_2$ and $t_{\mathbf{p}} \geq t_1 + ((1-c)/c) ( t_2 - t_{\mathbf{q}} )$, recalling that $t_{\mathbf{p}} = \sum_{\ih} p_{\ih} \log N_{\ih}$. Then 
\begin{align}
\# T_{\lfloor cJ \rfloor}(\mathbf{p}) \cdot \# T_{J - \lfloor cJ \rfloor}(\mathbf{q}) &\leq e^{J( c(\log M - I(t_{\mathbf{p}}))    +   (1-c) (\log M - I(\max\{t_{\mathbf{q}},\underline{t}\}))      )} \label{eq:alignindcomb1} \\
&\leq e^{J( c(\log M - I(t_1 + ((1-c)/c) ( t_2 - \max\{t_{\mathbf{q}},\underline{t}\} )))    +   (1-c) (\log M - I(\max\{t_{\mathbf{q}},\underline{t}\}))      )} \label{eq:alignindcomb2} \\ 
&\leq e^{J(c(\log M - I(t_1)) + (1-c)(\log M - I(t_2)))}. \label{eq:alignindcomb3}
\end{align}
In~\eqref{eq:alignindcomb1} we used~\eqref{eq:206} and Step~1 of Proposition~\ref{prop:1}. 
In~\eqref{eq:alignindcomb2} we used that the rate function is increasing. 
In~\eqref{eq:alignindcomb3} we used the convexity of the rate function. 
Therefore using~\eqref{eq:205} we can bound the cardinality of the set in the statement of~\eqref{eq:inductioncombinatorics1} from above by 
\[(\lfloor cJ \rfloor + 1 )^M (J-\lfloor cJ \rfloor +1)^M e^{J(c(\log M - I(t_1)) + (1-c)(\log M - I(t_2)))}. \]

\eqref{eq:inductioncombinatorics2} 
Similarly, fix $\mathbf{p} \in \mathcal{T}_{\lfloor cJ \rfloor}$ and $\mathbf{q} \in \mathcal{T}_{J - \lfloor cJ \rfloor}$ such that $t_{\mathbf{q}} \leq t_2$ and $t_{\mathbf{p}} \leq t_1 + {((1-c)/c) ( t_2 - t_{\mathbf{q}} )}$. 
Then 
\begin{align} 
&\# \{ \, \mathbf{i} \in [N]^{\lfloor cJ \rfloor} : \iih \in T_{\lfloor cJ \rfloor}(\mathbf{p}) \, \} \cdot \# \{ \,  \mathbf{j} \in [N]^{J- \lfloor cJ \rfloor} : \jjh \in T_{J- \lfloor cJ \rfloor}(\mathbf{q}) \, \} \nonumber \\
&\leq e^{J(c(\max\{t_{\mathbf{p}},\underline{t}\} + \log M - I(\max\{t_{\mathbf{p}},\underline{t}\})) + (1-c)(\max\{t_{\mathbf{q}},\underline{t}\} + \log M - I(\max\{t_{\mathbf{q}},\underline{t}\})))} \nonumber \\
&\leq e^{ J c(\min\{t_1 + \frac{1-c}{c} ( t_2 - \max\{t_{\mathbf{q}},\underline{t}\} ),\overline{t} \} + \log M -  I(\min\{t_1 + \frac{1-c}{c} ( t_2 - \max\{t_{\mathbf{q}},\underline{t}\} ),\overline{t} \}) )} \nonumber \\*
&\phantom{\leq}\times e^{J (1-c)( \max\{t_{\mathbf{q}},\underline{t}, t_2 - \frac{c}{1-c} (\overline{t} - t_1)\} + \log M - I( \max\{t_{\mathbf{q}},\underline{t}, t_2 - \frac{c}{1-c} (\overline{t} - t_1) \} )  )  } \label{eq:bothlower2} \\
&\leq e^{J (c(t_1 + \log M - I(t_1)) + (1-c)(t_2 + \log M - I(t_2)))}. \label{eq:bothlower3}
\end{align}
In~\eqref{eq:bothlower2} we used the fact that $I'(t) < 1$ if $t \in (\underline{t},\overline{t})$ and $I'(t) > 1$ if $t \in (\overline{t},\max_{1\leq i \leq M} \log N_i)$. In~\eqref{eq:bothlower3} we used the convexity of the rate function. 
In light of~\eqref{eq:205}, the result follows. 
\end{proof}

The inequalities in Lemma~\ref{lem:tripleindcomb} mimic~\eqref{eq:secondthreshold}, \eqref{eq:secondpartverybig}, \eqref{eq:tsaverage} and \eqref{eq:thirdthreshold}. 

\begin{lemma}\label{lem:tripleindcomb}
Suppose $s \in (\dim_{\mathrm H} \Lambda,\dim_{\mathrm B} \Lambda)$, $c \in (0,1)$, $\underline{t} < t < T_s(t) < \overline{t}$ and $J \in \mathbb{N}$. Then as $J \to \infty$, 
\begin{enumerate}
\item\label{eq:tripleindcomb1}
\begin{align*}
\# &\{ \, \iih \in [M]^J  :     \tau(\iih,\lfloor cJ \rfloor, J )  \geq  t \\
&\mbox{ and }  \tau(\iih,0,\lfloor cJ \rfloor)    \geq  T_s  (  \max\{ \underline{t},     ( 1 - c + \gamma c ) t / (\gamma c)    -    (1-c) \tau(\iih,\lfloor cJ \rfloor, J ) / (\gamma c)    \}         )  \, \}   \\
&\asymp   e^{ J( c( \log M - I(T_s(t)) )   +   (1-c) ( \log M - I(t) )     ) }    ;
\end{align*}

\item\label{eq:tripleindcomb2}
\begin{align*}
\# \{ \, &(i_1,\dotsc,i_{\lfloor cJ \rfloor} , \ih_{\lfloor cJ \rfloor + 1},\dotsc, \ih_J, \ih_{J+1},\dotsc, \ih_{\lfloor (1+ \gamma c) J \rfloor}  \in [N]^{\lfloor cJ \rfloor}  \times [M]^{\lfloor (1 + \gamma c) J \rfloor - \lfloor cJ \rfloor} ) : \\
&t \leq \tau(\iih,\lfloor cJ \rfloor, J  ) \leq ( (1-c+\gamma c) t  -   \gamma c \underline{t} )/(1-c) \\
&\mbox{ and } \tau(\iih,0,\lfloor cJ \rfloor)  \leq  T_s( (  ( 1 - c + \gamma c ) t   -    (1-c) \tau(\iih,\lfloor cJ \rfloor, J )) / (\gamma c) ) \\
&\mbox{ and }   \tau(\iih,J, \lfloor (1 + \gamma c) J \rfloor  ) \geq  (( 1 - c + \gamma c ) t    -    (1-c) \tau(\iih,\lfloor cJ \rfloor, J ) )/ (\gamma c)      \, \} \\
&\asymp   e^{J(c(T_s(t) + \log M - I(T_s(t))) + (1-c + \gamma c) (\log M - I(t))   )};
 \end{align*}

\item\label{eq:tripleindcomb3}
\begin{align*}
\# &\{ \, \mathbf{i}  \in [N]^{\lfloor (1 + \gamma c) J\rfloor }  : t \leq \tau(\iih,\lfloor cJ \rfloor, J  ) \leq ( (1-c+\gamma c) t  -   \gamma c \underline{t}  )/(1-c) \\
&\mbox{ and } \tau(\iih,0,\lfloor cJ \rfloor)  \leq  T_s( ( ( 1 - c + \gamma c ) t   -    (1-c) \tau(\iih,\lfloor cJ \rfloor, J )) / (\gamma c) ) \\
&\mbox{ and }   \tau(\iih,J, \lfloor (\gamma c + 1)J \rfloor  ) \leq  (( 1 - c + \gamma c ) t  -    (1-c) \tau(\iih,\lfloor cJ \rfloor, J )) / (\gamma c)      \, \} \\
&\asymp   e^{J(c(T_s(t) + \log M - I(T_s(t))) + (1-c + \gamma c) (t + \log M - I(t))   )}.
\end{align*}

\end{enumerate}
\end{lemma}

\begin{proof}
The proof strategy is rather similar to that of Lemma~\ref{lem:inductioncombinatorics}. The lower bounds follow from considering those strings for which $\ih_{\lfloor cJ \rfloor + 1}, \dotsc, \ih_J$ and $\ih_{1},\dotsc,\ih_{\lfloor cJ \rfloor}$ are the best approximations to $\mathbf{Q}^*_{t}$ and $\mathbf{Q}^*_{T_s(t)}$ respectively in $\mathcal{T}_{J-\lfloor cJ \rfloor}$ and $\mathcal{T}_{\lfloor cJ \rfloor}$, and (for~\eqref{eq:tripleindcomb2} and~\eqref{eq:tripleindcomb3}) $\ih_{J+1},\dotsc,\ih_{\lfloor \gamma c J \rfloor}$ is the best approximation to $\mathbf{Q}^*_{t}$ in $\mathcal{T}_{\lfloor (1 + \gamma c) J \rfloor - J}$, for which the required inequalities hold. 
The upper bounds follow from the following estimates and~\eqref{eq:205}. 

 \eqref{eq:tripleindcomb1} 
 Fix $\mathbf{p} \in \mathcal{T}_{\lfloor cJ \rfloor}$ and $\mathbf{q} \in \mathcal{T}_{J - \lfloor cJ \rfloor}$ such that $t_{\mathbf{q}} \geq t$ and 
 \[ t_{\mathbf{p}} \geq T_s  (  \max\{ \underline{t},     ( 1 - c + \gamma c ) t / (\gamma c)    -    (1-c) t_{\mathbf{q}} / (\gamma c).\]
  Then 
 \begin{align*} 
\# T_{\lfloor cJ \rfloor}(\mathbf{p}) \cdot \# T_{J- \lfloor cJ \rfloor}(\mathbf{q}) &\leq e^{Jc(\log M -  I( T_s  (  \max\{ \underline{t},     ( 1 - c + \gamma c ) t / (\gamma c)  -  (1-c) t_{\mathbf{q}} / (\gamma c) \}  )) )}  \\*
&\phantom{\leq}\times e^{J(1-c) (\log M -  I(\min\{t_{\mathbf{q}},( (1-c+\gamma c) t  -   \gamma c \underline{t} )/(1-c)\})    )  )} \\
&\leq e^{ J( c( \log M - I(T_s(t)) )   +   (1-c) ( \log M - I(t) )     ) }.  
\end{align*}
The final step holds since 
\begin{align*}
 ( (1-c+\gamma c) t  -   \gamma c \underline{t} )/(1-c)  &\leq t \leq t_{\mathbf{q}}; \\*
 T_s\big(( (1-c+\gamma c) t  -   \gamma c \underline{t} )/(1-c)\big)  &\leq  T_s(t) < \overline{t}, 
\end{align*}
so using standard properties of the rate function, the derivative of the exponent with respect to $t_{\mathbf{q}}$ is negative. 

\eqref{eq:tripleindcomb2} 
Now fix $\mathbf{p} \in \mathcal{T}_{\lfloor cJ \rfloor}$, $\mathbf{q} \in \mathcal{T}_{J - \lfloor cJ \rfloor}$ and $\mathbf{r} \in \mathcal{T}_{\lfloor (1 + \gamma c) J \rfloor - J}$ such that $t \leq t_{\mathbf{q}} \leq ( (1-c+\gamma c) t  -  \gamma c \underline{t} )/(1-c)$, $t_{\mathbf{p}} \leq  T_s( (  ( 1 - c + \gamma c ) t   -    (1-c) t_{\mathbf{q}}) / (\gamma c) )$ and $t_{\mathbf{r}} \geq (( 1 - c + \gamma c ) t    -    (1-c) t_{\mathbf{q}} )/ (\gamma c)$. Then 
\begin{align*}
&\# \{ \, \mathbf{i} \in [N]^{\lfloor cJ \rfloor} : \iih \in T_{\lfloor cJ \rfloor}(\mathbf{p}) \, \} \cdot \# T_{J- \lfloor cJ \rfloor}(\mathbf{q}) \cdot \# T_{\lfloor (1 + \gamma c) J \rfloor - J}(\mathbf{r}) \\
&\leq  e^{J( c(T_s( (  ( 1 - c + \gamma c ) t   -    (1-c) t_{\mathbf{q}}) / (\gamma c) ) + \log M - I(T_s( (  ( 1 - c + \gamma c ) t   -    (1-c) t_{\mathbf{q}}) / (\gamma c) )))                    +   (1-c) (\log M - I(t_{\mathbf{q}})) )} \\*
&\phantom{\leq}\times e^{J(  \gamma c (\log M  -  I((( 1 - c + \gamma c ) t    -    (1-c) t_{\mathbf{q}} )/ (\gamma c)))        )}   \\
&\leq e^{J(c(T_s(t) + \log M - I(T_s(t))) + (1-c + \gamma c) (\log M - I(t))   )}
\end{align*}
since the derivative of the exponent with respect to $t_{\mathbf{q}}$ is negative. 

\eqref{eq:tripleindcomb3} 
Now fix $\mathbf{p} \in \mathcal{T}_{\lfloor cJ \rfloor}$, $\mathbf{q} \in \mathcal{T}_{J - \lfloor cJ \rfloor}$ and $\mathbf{r} \in \mathcal{T}_{\lfloor (1 + \gamma c) J \rfloor - J}$ such that $t \leq t_{\mathbf{q}} \leq ( (1-c+\gamma c) t  -  \gamma c \underline{t} )/(1-c)$, $t_{\mathbf{p}} \leq  T_s( (  ( 1 - c + \gamma c ) t   -    (1-c) t_{\mathbf{q}}) / (\gamma c) )$ and $t_{\mathbf{r}} \leq (( 1 - c + \gamma c ) t    -    (1-c) t_{\mathbf{q}} )/ (\gamma c)$. Then 
\begin{align}
 &\# \{ \, \mathbf{i} \in [N]^{\lfloor cJ \rfloor} : \iih \in T_{\lfloor cJ \rfloor}(\mathbf{p}) \, \} \cdot \# \{ \, \mathbf{j} \in [N]^{J - \lfloor cJ \rfloor} : \jjh \in T_{J - \lfloor cJ \rfloor}(\mathbf{q}) \, \} \nonumber \\*
 &\phantom{\leq}\times \# \{ \, \mathbf{k} \in [N]^{\lfloor (1 + \gamma c) J \rfloor - J} : \kkh \in T_{\lfloor (1 + \gamma c) J \rfloor - J}(\mathbf{r}) \, \} \nonumber \\
 &\leq e^{J c(T_s( (  ( 1 - c + \gamma c ) t   -    (1-c) t_{\mathbf{q}}) / (\gamma c) ) + \log M - I(T_s( (  ( 1 - c + \gamma c ) t   -  (1-c) t_{\mathbf{q}}) / (\gamma c) ))) } \\*
 &\phantom{\leq}\times e^{J(1-c) (\min\{t_{\mathbf{q}},\overline{t}\} + \log M - I(\min\{t_{\mathbf{q}},\overline{t}\})) } \nonumber \\*
 &\phantom{\leq}\times e^{J  \gamma c  ((( 1 - c + \gamma c ) t    -    (1-c) t_{\mathbf{q}} )/ (\gamma c) + \log M  -  I((( 1 - c + \gamma c ) t    -    (1-c) t_{\mathbf{q}} )/ (\gamma c))        )}  \nonumber \\
 &\leq e^{J   c(T_s( (  ( 1 - c + \gamma c ) t   -    (1-c) \min\{t_{\mathbf{q}},\overline{t}\}) / (\gamma c) ) + \log M - I(T_s( (  ( 1 - c + \gamma c ) t   -    (1-c) \min\{t_{\mathbf{q}},\overline{t}\}) / (\gamma c) )))                    }\nonumber \\*
 &\phantom{\leq}\times e^{J   (1-c) (\min\{t_{\mathbf{q}},\overline{t}\} + \log M - I(\min\{t_{\mathbf{q}},\overline{t}\})) } \nonumber \\*
 &\phantom{\leq}\times e^{J  \gamma c  ((( 1 - c + \gamma c ) t    -    (1-c) \min\{t_{\mathbf{q}},\overline{t}\} )/ (\gamma c) + \log M  -  I((( 1 - c + \gamma c ) t    -    (1-c) \min\{t_{\mathbf{q}},\overline{t}\} )/ (\gamma c))        )} \label{eq:aligntripleindcomb1} \\
 &\leq e^{J(c(T_s(t) + \log M - I(T_s(t))) + (1-c + \gamma c) (t + \log M - I(t))   )}. \label{eq:aligntripleindcomb2}
\end{align}
We have~\eqref{eq:aligntripleindcomb1} because $t \leq t_{\mathbf{q}} \leq ( (1-c+\gamma c) t  -  \gamma c \underline{t} )/(1-c)$, so 
\[ \underline{t} \leq (( 1 - c + \gamma c ) t  -  (1-c) t_{\mathbf{q}} )/ (\gamma c) \leq T_s\big((( 1 - c + \gamma c ) t    -    (1-c) t_{\mathbf{q}} )/ (\gamma c)\big)  \leq T_s(t) < \overline{t},\] 
so the derivative of the rate function here is between 0 and~1. 
We have~\eqref{eq:aligntripleindcomb2} since the derivative of the exponent in~\eqref{eq:aligntripleindcomb1} with respect to $t_{\mathbf{q}}$ is negative. %
In light of~\eqref{eq:205}, this completes the proof. 
\end{proof}

We are now ready to prove an upper bound for the $s$-cost of the cover that we have constructed. Recall that $G(\theta,s)$ is defined in~\eqref{eq:definemainquantity}. 
\begin{lemma}\label{lem:mainuppercost}
For $L \in \mathbb{N}$, $\theta \in (\gamma^{-L},\gamma^{-(L-1)})$, $s \in (\dim_{\mathrm H} \Lambda,\dim_{\gamma^{-(L-1)}} \Lambda]$ and $0 < \delta \ll 1$, let $\{V_j\}_j$ be the cover of $\Lambda$ defined in~\eqref{eq:definerealcover}. Then as $K = K(\delta) \to \infty$, 
\[ \sum_j |V_j|^s \asymp e^{K \cdot G(\theta,s)/(\gamma^L \theta)}. \]
\end{lemma}
\begin{proof}
The strategy is to bound the $s$-costs of the different parts of the cover separately using Lemmas~\ref{lem:simpleindcomb},~\ref{lem:inductioncombinatorics} and~\ref{lem:tripleindcomb}, which can be applied since the $t_i(s)$ lie in the right range by Lemma~\ref{lem:51}. 
We use the convention that an empty product equals 1. 
In the following, as above, $R \in \mathbb{N}$ will be fixed and arbitrary. 
We first consider the $s$-cost of $\mathcal{U}_{l}$ for $l \in \{0,1,\dotsc,L-2\}$: %
\begin{align*}
\sum_{U \in \mathcal{U}_{l}} |U|^s &\asymp \# \mathcal{U}_{l} \cdot  n^{-\gamma^{l}K s} \\
  &\asymp N^K \prod_{j = 1}^{l} e^{K(t_{L-j}(s) + \log M - I(t_{L-j}(s)))(\gamma^j - \gamma^{-(L-j)}\theta^{-1})} \\*
  &\phantom{\leq}\times e^{K(t_{L-j+1}(s) + \log M - I(t_{L-j+1}(s)) ) (\gamma^{-(L-j)}\theta^{-1} - \gamma^{j-1}) ) }  \\*
&\phantom{\leq}\times e^{K((\log M - I(t_{L-l-1}(s)))(\gamma^{l+1} - \gamma^{-(L-l-1)}\theta^{-1})  +  (\log M - I(t_{L-l}))  (\gamma^{-(L-l-1)}\theta^{-1} - \gamma^{l})   )}  \\*
&\phantom{\leq}\times n^{-\gamma^{l}K s}  \qquad \qquad \text{by Lemma~\ref{lem:inductioncombinatorics}~\eqref{eq:inductioncombinatorics1} and~\eqref{eq:inductioncombinatorics2}} \\
&\asymp \# (\mathcal{B}_{\gamma^{l}(K)} \cap \mathrm{supp}(\mu_{\delta,s,\theta,R})) \cdot n^{-\gamma^{l}K s} \qquad \substack{\text{ by~\eqref{eq:lowerboundcard}} \\ \text{and Proposition~\ref{prop:1}, Step~1}} \\ 
&\asymp \mu_{\delta,s,\theta,R} (\Lambda) \qquad  \qquad \text{by~\eqref{eq:definemu} and Lemma~\ref{lem:lowernicescales}}  \\
&\asymp e^{K \cdot G(\theta,s)/(\gamma^L \theta)} \qquad \text{by Lemma~\ref{lem:lowertotalmass}}. \stepcounter{equation}\tag{\theequation}\label{eq:mainupperfirstcost}
\end{align*}

The $s$-costs of $\mathcal{U}_{l,0}$ are equal for all $l \in \{0,1,\dotsc,L-1\}$: 
\begin{align*}
\sum_{U \in \mathcal{U}_{l,0}} |U|^s &\asymp \# \mathcal{U}_{L-1,0} \cdot n^{-K s/\theta}  \\
&\asymp  \prod_{j = 1}^{L-1} e^{K(t_{L-j}(s) + \log M - I(t_{L-j}(s)))(\gamma^j - \gamma^{-(L-j)}\theta^{-1})} \\*
&\phantom{\leq}\times e^{K(t_{L-j+1}(s) + \log M - I(t_{L-j+1}(s)) ) (\gamma^{-(L-j)}\theta^{-1} - \gamma^{j-1})  )   }  \\*
&\phantom{\leq}\times N^K e^{(t_1(s) + \log M - I(t_1(s)) )(1/\theta - \gamma^{L-1}) K }  M^{(\gamma - 1)K/\theta}   \\*
&\phantom{\leq}\times n^{-K s/\theta} \qquad \substack{\text{by Lemmas~\ref{lem:simpleindcomb},~\ref{lem:inductioncombinatorics}~\eqref{eq:inductioncombinatorics2}} \\ \text{and (when }l<L-1)\text{ Lemma~\ref{lem:tripleindcomb}~\eqref{eq:tripleindcomb3}} } \\
&\asymp \# (\mathcal{B}_{\lfloor K/\theta\rfloor} \cap \mathrm{supp}(\mu_{\delta,s,\theta,R})) \cdot n^{-Ks/\theta} \asymp \mu_{\delta,s,\theta,R} (\Lambda) \asymp e^{K \cdot G(\theta,s)/(\gamma^L \theta)}. 
\end{align*}

To bound the $s$-cost of $\mathcal{U}_{l,L-l}$ when $l \in \{0,1,\dotsc,L-2\}$, note that by Lemma~\ref{lem:33-first} and Step~1 of Proposition~\ref{prop:1}, 
\begin{align*}
\# &\Big\{ \, \iih \in [M]^{\gamma^{l+1}(K) - \gamma^l(K)}  :  \tau(\iih,\lfloor K/(\gamma^{L-l-1}\theta) \rfloor - \gamma^l(K),\gamma^{l+1}(K) -   \gamma^l(K) ) \\
&\phantom{--}\geq \frac{1}{(\gamma^{l+1}  -  (\gamma^{L-l-1}\theta)^{-1}  )}   \Big(  \Big(  \frac{1}{\gamma^{L-l-2}\theta}  - \frac{1}{\gamma^{L-l-1} \theta} \Big)  t_{L-l-1}(s)  -   ((\gamma^{L-l-2}\theta)^{-1}  -  \gamma^{l+1}) \underline{t}  \Big)  \\
&\phantom{--}\mbox{ and } \tau( \iih, 0,   \lfloor K/(\gamma^{L-l-1}\theta) \rfloor - \gamma^l(K)  )  \geq  T_s(\underline{t}) \, \Big\} \\
&\asymp  e^{-K (\gamma^{l+1}  -  \gamma^{l+1-L}\theta ) I\left(\frac{1}{(\gamma^{l+1}  -  (\gamma^{L-l-1}\theta)^{-1}    )}   \big(  \big(  \frac{1}{\gamma^{L-l-2}\theta}  - \frac{1}{\gamma^{L-l-1} \theta} \big)  t_{L-l-1}(s) - ((\gamma^{L-l-2}\theta)^{-1}  -  \gamma^{l+1}) \underline{t}  \big)\right)  }\\*
&\phantom{\leq}\times e^{K \gamma^{l+1}  -  \gamma^{l+1-L}\theta ) \log M }\cdot e^{K \left(\gamma^{l+1-L}\theta   -  \gamma^l)  (\log M - I(T_s(\underline{t})))   \right)} \\
&\leq  e^{K\left(   (\gamma^{l+1}  -  \gamma^{l+1-L}\theta^{-1}  )\left( \log M - I(t_{L-l-1}(s))    \right)   +  (\gamma^{l+1-L}\theta^{-1}  -  \gamma^l)  (\log M - I(t_{L-l}(s)))   \right)},\stepcounter{equation}\tag{\theequation}\label{eq:combwithincostproof}
\end{align*}
where the last step follows from the case $t_{\mathbf{q}} = ( (1-c+\gamma c) t  -   \gamma c \underline{t} )/(1-c)$ of Lemma~\ref{lem:tripleindcomb}~\eqref{eq:tripleindcomb1}. 

Now, for $l \in \{0,1,\dotsc,L-1\}$, %
\begin{align*}
\sum_{U \in \mathcal{U}_{l,L-l}} |U|^s &\asymp  N^K \prod_{j = 1}^{l} e^{K(t_{L-j}(s) + \log M - I(t_{L-j}(s)))(\gamma^j - \gamma^{-(L-j)}\theta^{-1})} \\*
&\phantom{\leq}\times e^{K(t_{L-j+1}(s) + \log M - I(t_{L-j+1}(s)) ) (\gamma^{-(L-j)}\theta^{-1} - \gamma^{j-1})  )   }  \\*
&\phantom{\leq}\times e^{K((\log M - I(t_{L-l-1}(s)))(\gamma^{l+1} - \gamma^{-(L-l-1)}\theta^{-1})  +  (\log M - I(t_{L-l}))  (\gamma^{-(L-l-1)}\theta^{-1} - \gamma^{l})   )}  \\*
&\phantom{\leq}\times n^{-\gamma^{l}K s}  \\
&\asymp \# (\mathcal{B}_{\gamma^{l}(K)} \cap \mathrm{supp}(\mu_{\delta,s,\theta,R})) \cdot n^{-\gamma^{l}K s} \asymp \mu_{\delta,s,\theta,R} (\Lambda) \asymp e^{K \cdot G(\theta,s)/(\gamma^L \theta)}. 
\end{align*}
In the case $l=L-1$ we used Lemma~\ref{lem:inductioncombinatorics}~\eqref{eq:inductioncombinatorics2} and Lemma~\ref{lem:33-first} and Step~1 of Proposition~\ref{prop:1}. 
In the case $l<L-1$ we used Lemma~\ref{lem:inductioncombinatorics}~\eqref{eq:inductioncombinatorics2} and (in the case when the maximum in~\eqref{eq:tsaverage} does not take the value $\underline{t}$, or equivalently when~\eqref{eq:secondpartverybig} does not hold) Lemma~\ref{lem:tripleindcomb}~\eqref{eq:tripleindcomb1}, and (in the case when the maximum does take the value $\underline{t}$) we used~\eqref{eq:combwithincostproof}. 

Now, for $l \in \{0,1,\dotsc,L-2\}$, 
 \begin{align*}
\sum_{U \in \mathcal{U}_{l,L-l-1}'} |U|^s &\asymp N^K \prod_{j = 1}^{l} e^{K(t_{L-j}(s) + \log M - I(t_{L-j}(s)))(\gamma^j - \gamma^{-(L-j)}\theta^{-1})} \\*
&\phantom{\leq}\times e^{K(t_{L-j+1}(s) + \log M - I(t_{L-j+1}(s)) ) (\gamma^{-(L-j)}\theta^{-1} - \gamma^{j-1})   } \cdot e^{K(\gamma^{l+1} - \gamma^{l+1-L}\theta^{-1} )}  \\*
&\hspace{-2.5cm}\phantom{\leq}\times e^{-K (\gamma^{l+1} - \gamma^{l+1-L}\theta^{-1} ) I\left(  \frac{1}{\gamma^{l+1}  -  (\gamma^{L-l-1}\theta)^{-1}  }   \big(  \big(  \frac{1}{\gamma^{L-l-2}\theta}  - \frac{1}{\gamma^{L-l-1} \theta} \big)  t_{L-l-1}(s)    -   ((\gamma^{L-l-2}\theta)^{-1}  -  \gamma^{l+1}) \underline{t}  \big)  \right)   } \\*
&\phantom{\leq}\times e^{K(\gamma^{l+1-L}\theta^{-1}  -  \gamma^l) ( T_s(\underline{t})  +  \log M - I(T_s(\underline{t})) ) } \cdot  M^{K(\gamma^{l+2-L}\theta^{-1} - \gamma^{l+1})}   \cdot n^{K/(\gamma^{L-l-1}\theta)} \\
&\leq N^K \prod_{j = 1}^{l} e^{K(t_{L-j}(s) + \log M - I(t_{L-j}(s)))(\gamma^j - \gamma^{-(L-j)}\theta^{-1})} \\*
&\phantom{\leq}\times e^{K(t_{L-j+1}(s) + \log M - I(t_{L-j+1}(s)) ) (\gamma^{-(L-j)}\theta^{-1} - \gamma^{j-1})     }   \\*
&\phantom{\leq}\times e^{ K  (\gamma^{l+2-L}\theta^{-1} - \gamma^{l+1-L}\theta^{-1} ) (\log M - I(t_{L-l-1}(s))} \\*
&\phantom{\leq}\times e^{K(\gamma^{l+1-L}\theta^{-1}  -  \gamma^l) ( t_{L-l}(s)  +  \log M - I(t_{L-l}(s))) )  } \cdot  n^{K/(\gamma^{L-l-1}\theta)} \\
&\asymp  \# (\mathcal{B}_{\lfloor K/(\gamma^{L-l-1}\theta) \rfloor}  \cap \mathrm{supp}(\mu_{\delta,s,\theta,R})) \cdot n^{K/(\gamma^{L-l-1}\theta)} \asymp \mu_{\delta,s,\theta,R} (\Lambda) \\
&\asymp e^{K \cdot G(\theta,s)/(\gamma^L \theta)}, 
\end{align*}
where the inequality follows from the case $t_{\mathbf{q}} = ( (1-c+\gamma c) t  -   \gamma c \underline{t} )/(1-c)$ of the proof of Lemma~\ref{lem:tripleindcomb}~\eqref{eq:tripleindcomb2}. 

For $l \in \{0,1,\dotsc,L-2\}$, using Lemma~\ref{lem:tripleindcomb}~\eqref{eq:tripleindcomb2} (and Lemma~\ref{lem:inductioncombinatorics}~\eqref{eq:inductioncombinatorics2}), 
\begin{align*}
\sum_{U \in \mathcal{U}_{l,L-l-1}} |U|^s 
&\asymp N^K \prod_{j = 1}^{l} e^{K(t_{L-j}(s) + \log M - I(t_{L-j}(s)))(\gamma^j - \gamma^{-(L-j)}\theta^{-1})} \\*
&\phantom{\leq}\times e^{K(t_{L-j+1}(s) + \log M - I(t_{L-j+1}(s)) ) (\gamma^{-(L-j)}\theta^{-1} - \gamma^{j-1})    }  \\*
&\phantom{\leq}\times e^{ K (\gamma^{l+2-L}\theta^{-1} - \gamma^{l+1-L}\theta^{-1} ) (\log M - I(t_{L-l-1}(s))} \\*
&\phantom{\leq}\times e^{K(\gamma^{l+1-L}\theta^{-1}  -  \gamma^l) ( t_{L-l}(s)  +  \log M - I(t_{L-l}(s))) )  } n^{K/(\gamma^{L-l-1}\theta)} \\
&\asymp  \# (\mathcal{B}_{\lfloor K/(\gamma^{L-l-1}\theta) \rfloor}  \cap \mathrm{supp}(\mu_{\delta,s,\theta,R})) \cdot n^{K/(\gamma^{L-l-1}\theta)} \\
&\asymp \mu_{\delta,s,\theta,R} (\Lambda) \asymp e^{K \cdot G(\theta,s)/(\gamma^L \theta)}.
\end{align*}

Finally, if $L \geq 3$, $l \in \{0,1,\dotsc,L-3\}$ and $k \in \{1,2,\dotsc, L-l-2\}$, 
\begin{align*}
\sum_{U \in \mathcal{U}_{l,k}} |U|^s 
&\asymp N^K \prod_{j = 1}^{L-k-1} e^{K(t_{L-j}(s) + \log M - I(t_{L-j}(s)))(\gamma^j - \gamma^{-(L-j)}\theta^{-1})}\\*
&\phantom{\leq}\times e^{K(t_{L-j+1}(s) + \log M - I(t_{L-j+1}(s)) ) (\gamma^{-(L-j)}\theta^{-1} - \gamma^{j-1})  }   \\*
&\phantom{\leq}\times e^{K(\gamma^{-k}\theta^{-1}  -  \gamma^{L-k-1})(t_{L-k-1}(s)  +  \log M - I(t_{L-k-1})) } \\*
&\phantom{\leq}\times e^{K(\gamma^{-(k+1)}\theta^{-1} - \gamma^{-k}\theta^{-1})(\log M - I(t_{L-k-2}(s))) } \cdot n^{-Ks/(\gamma^k \theta)} \\
&\asymp \# \mathcal{B}_{\lfloor K/(\gamma^k \theta)\rfloor} \cap \mathrm{supp}(\mu_{\delta,s,\theta,R}) \cdot n^{-Ks/(\gamma^k \theta)} \asymp \mu_{\delta,s,\theta,R} (\Lambda) \asymp e^{K \cdot G(\theta,s)/(\gamma^L \theta)},
\end{align*}
using Lemma~\ref{lem:tripleindcomb}~\eqref{eq:tripleindcomb3} (and Lemma~\ref{lem:inductioncombinatorics}~\eqref{eq:inductioncombinatorics2} and Lemmas~\ref{lem:simpleindcomb},~\ref{lem:33-first} and~\ref{lem:lowertotalmass}). 
We have now bounded the $s$-cost of each part of the cover, so the proof is complete.  
\end{proof}

Note that when we applied Lemma~\ref{lem:inductioncombinatorics}, in the proof of Lemma~\ref{lem:mainuppercost} (for example in~\eqref{eq:mainupperfirstcost}), we used that $t_L(s) < \overline{t}$ for all $s \in [\dim_{\gamma^{-L}} \Lambda,\dim_{\gamma^{-(L-1)}} \Lambda]$, see Lemma~\ref{lem:51}. 
We needed Section~\ref{sec:upperinteger} to establish the inequality $t_L(s) < \overline{t}$ before Section~\ref{sec:mainupper}. 

\subsection{Conclusion and discussion of the proof}\label{sec:proofconclusion} 

We now conclude the proof of Theorem~\ref{thm:main} by combining the upper bounds in Sections~\ref{sec:upperinteger} and~\ref{sec:mainupper} and a lower bound from the mass distribution principle Proposition~\ref{prop:mdp}, which can be applied by virtue of the results in Section~\ref{sec:prooflower}. 

\begin{proof}[Proof of Theorem~\ref{thm:main}]%
Fix $\theta \in (0,1)$ and $s \in (\dim_{\mathrm{H}} \Lambda,\overline{\dim}_{\gamma^{-(L-1)}} \Lambda]$. Then 
\[ \limsup_{\delta\searrow 0} \frac{\log S_{\delta, \theta}^{s}(\Lambda)}{-\log \delta} \leq \frac{G(\theta,s)}{\gamma^L \theta \log n}.\]
This follows from Lemma~\ref{lem:51} and Lemma~\ref{lem:50} with $\btau = (t_1(s),\dotsc, t_{L-1}(s))$ in the case $\theta = \gamma^{-(L-1)}$, and from Lemma~\ref{lem:mainuppercost} in the case $\gamma^{-L} < \theta < \gamma^{-(L-1)}$. 
If $U \subset \mathbb{R}^2$ is Borel with $|U| <1$ then the number of approximate squares of level $\lceil \frac{-\log |U|}{\log n} \rceil$ which $U$ intersects is at most an absolute constant depending only on the carpet. Therefore by Lemma~\ref{lem:lowerallscales}, for all $s' < s$ there exists $\delta_0 > 0$ and $R \in \mathbb{N}$ such that for all $\delta \in (0,\delta_0)$, if $\delta^{1/\theta} \leq |U| \leq \delta$ then $\mu_{\delta,s,\theta,R}(U) < |U|^{s'}$.  This means that we can use Lemma~\ref{lem:lowertotalmass} and apply the mass distribution principle Proposition~\ref{prop:mdp} and deduce that 
\[ \liminf_{\delta\searrow 0} \frac{\log S_{\delta, \theta}^{s'}(\Lambda)}{-\log \delta} \geq \frac{G(\theta,s)}{\gamma^L \theta \log n}.\] %
Since $s' < s$ was arbitrary and $\liminf_{\delta\searrow 0} \frac{\log S_{\delta, \theta}^{s'}(\Lambda)}{-\log \delta}$ is continuous in $s'$ by \cite[Lemma~2.1]{Burrell2021projections}, $\liminf_{\delta\searrow 0} \frac{\log S_{\delta, \theta}^{s}(\Lambda)}{-\log \delta} \geq G(\theta,s)/(\gamma^L \theta \log n)$. Thus 
\[ \frac{\log S_{\delta, \theta}^{s}(\Lambda)}{-\log \delta} \to \frac{G(\theta,s)}{\gamma^L \theta \log n} \mbox{ as } \delta \searrow 0. \]
Therefore by~\eqref{eq:45}, $\dim_{\, \theta} \Lambda$ exists and $G(\theta,\dim_{\, \theta} \Lambda) = 0$. 
By induction $t_1(s),t_2(s),\dotsc,t_L(s)$ are strictly increasing functions of $s$. Thus for fixed $\theta$, $G(\theta,s)$ is strictly decreasing in $s$, so $s=\dim_{\, \theta} \Lambda$ is the only solution $s \in [\dim_{\mathrm{H}} \Lambda, \dim_{\mathrm{B}} \Lambda]$ to the equation $G(\theta,s) = 0$. 
\end{proof}

\begin{remark}\label{rem:pressureexp}
The significance of the pressure function can be illustrated by the simple case $\theta \geq \gamma^{-1}$. Indeed, in this case the optimal cover which gives the smallest possible $s$-cost (up to absolute multiplicative constants depending only on the carpet) involves subdividing a level-$K$ approximate square to a level $k \in \{K,K+1,\dotsc,\lfloor K/\theta \rfloor \}$ which minimises $\psi_{(\ih_{K+1},\dotsc,\ih_{\lfloor K/\theta \rfloor}) | k}(s)$. By considering the symbolic representation of the approximate squares in such a cover, the $s$-cost is, up to multiplicative constants, 
\begin{align*}
 N^K \Psi_{\lfloor K/\theta \rfloor - K}(s) M^{\gamma (K) -\lfloor K/\theta \rfloor} n^{-Ks} &\asymp e^{K( \log N + (\theta^{-1} - 1)P(s) - \theta^{-1} \log M - s\log n)} \\*
 &\asymp e^{K( (\dim_{\mathrm B} \Lambda) \log n - (\theta^{-1} - 1)I(t) - s\log n )}
 \end{align*}
by Proposition~\ref{prop:1}, where $(s,t)$ are related by~\eqref{eq:23}. 
Therefore the exponential growth rate of this $s$-cost is in fact the same as that of the cover constructed in Section~\ref{sec:mainupper} using just the two extreme scales $K$ and $\lfloor K/\theta \rfloor$. 
\end{remark}

\section{Proof of corollaries of Theorem \ref{thm:main}}\label{sec:proofcorollaries}

In this section we prove the corollaries and consequences of Theorem~\ref{thm:main}. 

\subsection{Proof of Corollary \ref{cor:allprop}}\label{subsec:proofform}
\begin{proof}[Proof of part~\eqref{itemi}]
For 
\[ (\theta, s) \in (\gamma^{-L},\gamma^{-(L-1)}) \times (\dim_{\mathrm H} \Lambda, \dim_{\mathrm B} \Lambda) \coloneqq D \subset \mathbb{R}^2,\]
define $G(\theta,s)$ by~\eqref{eq:definemainquantity}. Then $G(\theta,s)$ has continuous partial derivatives of all orders, so is $C^\infty$, on $D$. Moreover, the rate function $I$ is analytic (as the Legendre transform of an analytic function) and compositions of analytic functions are analytic. It follows that for all $(\theta,s) \in D$ and $(\theta_1,s_1) \in \mathbb{R}^2$ there exists $r = r(\theta,s,\theta_1,s_1) > 0$ such that the function $\lambda \mapsto G(\theta+\lambda \theta_1,s+\lambda s_1)$ is real analytic for $\lambda \in (-r,r)$. Therefore by a result of Siciak \cite[Theorem~1]{Siciak1970analytic}, $G(\theta,s)$ is jointly analytic in $(\theta,s) \in D$. Thus by the analytic implicit function theorem, the function $\theta \mapsto \dim_{\, \theta} \Lambda$ (describing the zero set of $G(\theta,s)$) is analytic for $\theta \in (\gamma^{-L},\gamma^{-(L-1)})$. 
\end{proof}

The next lemma gives a formula for the derivative. Recall, if $\theta\in(\gamma^{-L},\gamma^{-(L-1)}]$ then the formula for the intermediate dimension $s(\theta) = \dim_{\, \theta} \Lambda$ is
\begin{equation}\label{eq:61}
\gamma^L\theta \log N - (\gamma^L\theta - 1) t_L(s(\theta)) + \gamma(1-\gamma^{L-1}\theta)\big(\log M - I(t_L(s(\theta)))\big) - s(\theta)\log n = 0.
\end{equation}

\begin{lemma}\label{lem:60}
For all $L\in\mathbb{N}$ and $\theta\in(\gamma^{-L},\gamma^{-(L-1)})$, we have $\partial_- \dim_{\, \theta} \Lambda  =  \partial_+ \dim_{\, \theta} \Lambda = s'(\theta)$, where
\begin{equation*}
s'(\theta) = \frac{\gamma^L}{\log n} \cdot \frac{\log N - t_L(s(\theta))-\log M + I\big(t_L(s(\theta))\big) }{ 1+\big(\gamma^L \theta -1 + \gamma(1-\gamma^{L-1}\theta)I'\big(t_L(s(\theta))\big)\big) \cdot A_L(\theta) },
\end{equation*}
where $I'$ denotes the derivative of the rate function $I$, and 
\begin{equation*}
A_L(\theta)\coloneqq \sum_{\ell=0}^{L-1}\gamma^{\ell}\prod_{m=1}^{\ell}I'\big(t_{L-m}(s(\theta))\big), 
\end{equation*}
with the empty product defined to be 1. It follows that for $\theta=\gamma^{-L}$
\begin{align}
\partial_- \dim_{\gamma^{-L}} \Lambda &= \frac{\gamma^{L+1}}{\log n} \cdot \frac{\log N - t_{L+1}(s(\gamma^{-L}))-\log M + I\big(t_{L+1}(s(\gamma^{-L}))\big) }{ 1+(\gamma-1) \cdot A_{L+1}(\gamma^{-L}) }, \label{eq:62} \\*
\partial_+ \dim_{\gamma^{-L}} \Lambda &= \frac{\gamma^L}{\log n} \cdot \frac{\log N - t_L(s(\gamma^{-L}))-\log M + I\big(t_L(s(\gamma^{-L}))\big) }{ 1+(\gamma-1)I'\big(t_L(s(\gamma^{-L}))\big) \cdot A_L(\gamma^{-L}) }. \label{eq:63}
\end{align}
\end{lemma}

\begin{proof}
We first show by induction on $L$ that
\begin{equation}\label{eq:60}
\frac{\mathrm{d}}{\mathrm{d}\theta}t_L(s(\theta)) = s'(\theta)\cdot \log n \cdot A_L(\theta).
\end{equation}
For $L=1$, we have $t_1(s(\theta))=\big(s(\theta)-\frac{\log M}{\log m}\big)\log n$ and $A_1(\theta)=1$, so the claim holds. Assuming~\eqref{eq:60} for $L-1$, we now prove for $L$:
\begin{align*}
\frac{\mathrm{d}}{\mathrm{d}\theta}t_L(s(\theta)) &= \frac{\mathrm{d}}{\mathrm{d}\theta}T_{s(\theta)}\big(t_{L-1}(s(\theta))\big) = s'(\theta)\cdot \log n + \gamma I'\big( t_{L-1}(s(\theta)) \big)\cdot \frac{\mathrm{d}}{\mathrm{d}\theta}t_{L-1}(s(\theta)) \\
&= s'(\theta)\cdot \log n + \gamma I'\big( t_{L-1}(s(\theta)) \big)\cdot s'(\theta)\cdot \log n \cdot A_{L-1}(\theta) \\
&= s'(\theta)\cdot \log n \cdot \bigg( 1+ \sum_{\ell=0}^{L-2}\gamma^{\ell+1}\prod_{m=0}^{\ell}I'\big(t_{L-1-m}(s(\theta))\big) \bigg) \\
&= s'(\theta)\cdot \log n \cdot A_L(\theta),
\end{align*}
completing the proof of~\eqref{eq:60}.

Differentiating~\eqref{eq:61} with respect to $\theta$,
\begin{multline*}
s'(\theta)\cdot \log n = \gamma^L \log N - \gamma^L t_L(s(\theta)) -(\gamma^L \theta-1)\frac{\mathrm{d}}{\mathrm{d}\theta}t_L(s(\theta)) \\*
-\gamma^L\big( \log M - I\big(t_L(s(\theta)) \big) \big)  -\gamma(1-\gamma^{L-1}\theta) \frac{\mathrm{d}}{\mathrm{d}\theta}I\big(t_L(s(\theta))\big).
\end{multline*}
Using~\eqref{eq:60}, after rearranging we obtain the formula for $s'(\theta)$. The claims for $\theta=\gamma^{-L}$ follow by analyticity.
\end{proof}

\begin{proof}[Proof of part~\eqref{itemii}]
Follows directly from Lemma~\ref{lem:60}.
\end{proof}

The following lemma describes the behaviour of $t_{L(\theta)} (\dim_{\, \theta} \Lambda)$ for small $\theta$. 
\begin{lemma}\label{lem:tlconvergence}
We have \begin{itemize} \item $\lim_{L \to \infty} t_{L-1}(\dim_{\gamma^{-(L-1)}} \Lambda) = \liminf_{\theta \to 0^+} t_{L(\theta)} (\dim_{\, \theta} \Lambda) = t^*$
\item $\lim_{L \to \infty} t_L(\dim_{\gamma^{-(L-1)}} \Lambda) = \limsup_{\theta \to 0^+} t_{L(\theta)} (\dim_{\, \theta} \Lambda) = T_{\dim_{\mathrm H} \Lambda}(t^*)$
\end{itemize}
\end{lemma}
\begin{proof}
By~\eqref{eq:44}, the equation that $\dim_{\gamma^{-(L-1)}} \Lambda$ satisfies from Theorem~\ref{thm:main}, and the fact that the intermediate dimensions and rate function are continuous, we have $t_{L-1}(\dim_{\gamma^{-(L-1)}} \Lambda) \to t^*$ as $L \to \infty$. It follows that $t_L(\dim_{\gamma^{-(L-1)}} \Lambda) \to T_{\dim_{\mathrm H} \Lambda}(t^*)$, and that $\limsup_{\theta \to 0^+} t_{L(\theta)} (\dim_{\, \theta} \Lambda) \geq T_{\dim_{\mathrm H} \Lambda}(t^*)$. By considering $\theta > \gamma^{-L}$ very close to $\gamma^{-L}$, we see that $\liminf_{\theta \to 0^+} t_{L(\theta)} \dim_{\, \theta} \Lambda \leq t^*$. If $\gamma^{-L} < \theta \leq \gamma^{-(L-1)}$ then $t_L(\dim_{\gamma^{-L}} \Lambda) < t_L(\dim_{\, \theta} \Lambda) \leq t_L(\dim_{\gamma^{-(L-1)}} \Lambda)$. Therefore 
\begin{align*}
 t^* = \lim_{L \to \infty} t_L(\dim_{\, \theta} \Lambda) \leq \liminf_{\theta \to 0^+} t_{L(\theta)} (\dim_{\, \theta} \Lambda) &\leq \limsup_{\theta \to 0^+} t_{L(\theta)} (\dim_{\, \theta} \Lambda) \\
 &\leq \lim_{L \to \infty} t_L(\dim_{\gamma^{-(L-1)}} \Lambda) \\*
 &= T_{\dim_{\mathrm H} \Lambda}(t^*),
 \end{align*}
completing the proof. 
\end{proof}

\begin{proof}[Proof of part~\eqref{itemiii}]
For brevity, let us write 
\begin{equation}\label{eq:64}
s'(\theta) = \frac{\gamma^L}{\log n}\cdot \frac{f_L(\theta)}{1+g_L(\theta)A_L(\theta)}.
\end{equation}
Lemma~\ref{lem:51} ensures that $\underline{t}<t_1(\dim_{\mathrm{H}}\Lambda)<t_{\ell}(s(\theta))<T_{\dim_{\mathrm{H}}\Lambda}(t^{\ast})<\overline{t}$ for all $1\leq \ell\leq L$. Using that $I$ is strictly increasing and convex, there exist constants $c_1,c_1',c_2,c_2'>0$ independent of $\theta$ such that for all $t_1(\dim_{\mathrm{H}}\Lambda)\leq t\leq T_{\dim_{\mathrm{H}}\Lambda}(t^{\ast})$,
\begin{equation*}
0<c_1< I(t) <c_1'<\log M - H(\mathbf{P}) \;\text{ and }\; 0<c_2< I'(t) <c_2'<1.
\end{equation*}
Hence, recalling that $\overline{t} \coloneqq \log N - H(\mathbf{P})$ and $I(\overline{t}) = \log M - H(\mathbf{P})$, there exists $c_3>0$ such that the numerator
\begin{equation*}
f_L(\theta)= \overline{t} - t_L(s(\theta)) - \big( I(\overline{t}) - I\big(t_L(s(\theta))\big) \big) \geq c_3 >0.
\end{equation*}
Furthermore, there also exists $c_4>0$ such that $0<c_4\leq g_L(\theta)\leq c_4^{-1}<\infty$. Therefore,
\begin{equation*}
s'(\theta)\geq \frac{c_3}{\log n} \cdot \frac{\gamma^L}{1+c_4^{-1}A_L(\theta)} \geq \frac{c_3}{\log n} \cdot \frac{\gamma^L}{1+c_4^{-1}\cdot \frac{\gamma^L-1}{\gamma-1}} \geq \frac{c_3}{\log n \big(\gamma^{-1}+\frac{c_4^{-1}}{\gamma-1}\big)} =:C_0 > 0. 
\end{equation*}%

Next we show that $\partial_- \dim_{\gamma^{-L}} \Lambda < \partial_+ \dim_{\gamma^{-L}} \Lambda$ for all $L\in\mathbb{N}$ from~\eqref{eq:62} and~\eqref{eq:63}. We can divide both~\eqref{eq:62} and~\eqref{eq:63} by $\gamma^L/\log n$. We claim that
\begin{equation*}
\gamma^{-1}\cdot  A_{L+1}(\gamma^{-L}) - I'\big(t_L(s(\gamma^{-L}))\big) \cdot A_L(\gamma^{-L}) = \gamma^{-1},
\end{equation*} 
implying $\gamma^{-1}\cdot (1+g_{L+1}(\gamma^{-L})\cdot A_{L+1}(\gamma^{-L})) = 1+g_{L}(\gamma^{-L})\cdot A_{L}(\gamma^{-L})$. Indeed, applying the definition of $A_{L+1}(\gamma^{-L})$ and $A_{L}(\gamma^{-L})$, observe that it is a telescopic sum with only $\gamma^{-1}$ not cancelling out. Since we also have 
\begin{equation*}
t_{L+1}(s(\gamma^{-L})) - t_{L}(s(\gamma^{-L})) -\big( I\big(t_{L+1}(s(\gamma^{-L}))\big) - I\big(t_{L}(s(\gamma^{-L}))\big) \big)>0 
\end{equation*}
for the same reason that $f_L(\theta)>0$, it follows that $\partial_- \dim_{\gamma^{-L}} \Lambda < \partial_+ \dim_{\gamma^{-L}} \Lambda$. Finally, by Lemma~\ref{lem:tlconvergence}, 
\[ \frac{\partial_+ \dim_{\gamma^{-L}} \Lambda}{\partial_- \dim_{\gamma^{-L}} \Lambda} \xrightarrow[(L \to \infty)]{} \frac{\log N - t^* - \log M + I(t^*)}{\log N - T_{\dim_{\mathrm H} \Lambda}(t^*) - \log M + I(T_{\dim_{\mathrm H} \Lambda}(t^*))} \in (1,\infty). \qedhere \]%
\end{proof}

\begin{proof}[Proof of part~\eqref{itemiv}]
The idea is to estimate $L(\theta)$ (which is the number of the $t_i(s(\theta))$) in terms of $s(\theta) \coloneqq \dim_{\, \theta} \Lambda$. We do this by using the fact that $t'$ is a neutral fixed point of the function $T_{\dim_{\mathrm H} \Lambda}$, and 
\[ t_1(s(\theta)) \xrightarrow[\theta \to 0^+]{} \left(\dim_{\mathrm H} \Lambda - \frac{\log M}{\log m}\right) \log n < t' \] 
and $\liminf_{\theta \to 0^+} t_{L(\theta)}(s(\theta)) > t'$ by Lemmas~\ref{lem:tlconvergence} and~\ref{lem:tprimelesststar}, so most of the $t_i(s(\theta))$ lie close to $t'$. 

By Lemma~\ref{lem:41} and Taylor's theorem, since $T_s'(t') = 1$ and $T_s''(t') > 0$, there exists $c \geq 1$ such that for all $t \in (\underline{t},\overline{t})$, 
\begin{align}\label{eq:taylor}
\begin{split}
 T_s(t') + T_s'(t')(t-t') + c^{-1}(t-t')^2 &\leq T_s(t) \leq T_s(t') + T_s'(t')(t-t') + c(t-t')^2; \\*
t +  c^{-1}(t-t')^2 &\leq T_s(t) \leq t + c(t-t')^2.
\end{split}
\end{align}
If $L$ is large enough and $k_L \coloneqq \big\lfloor \max\left\{ L/10 , \frac{2}{\log 2} \log \left( \frac{\overline{t} - \underline{t}}{L(s(\theta)-\dim_{\mathrm H}\Lambda)\log n}\right) \right\} \big\rfloor$ then 
\begin{equation*}
t'-2^{k_L}\frac{L(s(\theta)-\dim_{\mathrm H}\Lambda)\log n}{4} < \underline{t} \;\text{ and }\; t' + 2^{k_L}\frac{L(s(\theta)-\dim_{\mathrm H}\Lambda)\log n}{4} > \overline{t}. 
\end{equation*}

Suppose $k \in \{1,2,\dotsc,k_L\}$. Then by~\eqref{eq:taylor}, 
\begin{align*}
 & \# \left\{ i \in \{1,2,\dotsc,L(\theta) \} : t' - 2^k\frac{L(s(\theta)-\dim_{\mathrm H}\Lambda)\log n}{4} < t_i(s(\theta)) \leq 2^{k-1} \frac{L(s(\theta)-\dim_{\mathrm H}\Lambda)\log n}{4} \right\} \\
 & \leq 1 + \frac{16c}{2^k L(s(\theta)-\dim_{\mathrm H}\Lambda)\log n}. 
 \end{align*}
 Summing up, it follows that 
 \[ \# \left\{  i : t_i(s(\theta)) \leq t' - \frac{L(s(\theta)-\dim_{\mathrm H}\Lambda)\log n}{4} \right\} \leq k_L +  \frac{16c}{L(s(\theta)-\dim_{\mathrm H}\Lambda)\log n}, \]
 and we similarly obtain the same bound for the number of $t_i(s)$ which are greater than $t' + \frac{L(s(\theta) + \dim_{\mathrm H}\Lambda)\log n}{4}$. 
 But 
 \[ \# \left\{ i : t' - \frac{L(s(\theta)-\dim_{\mathrm H}\Lambda)\log n}{4} < t_i(s(\theta)) \leq t' + \frac{L(s(\theta)-\dim_{\mathrm H}\Lambda)\log n}{4} \right\} \leq 1 + L/2. \]
 Therefore for $L$ sufficiently large, 
 \[ L \leq  2 k_L +  \frac{32c}{L(s(\theta)-\dim_{\mathrm H}\Lambda)\log n} + 1 + \frac{L}{2} \leq 0.9 L + \frac{40c}{ L(s(\theta)-\dim_{\mathrm H}\Lambda)\log n}. \]
 Decreasing $\theta_0$ further if required, this tells us that for all $\theta < \theta_0$, 
 \[ s(\theta) \leq  \dim_{\mathrm H} \Lambda +  \frac{400c}{L^2 \log n}  \leq  \dim_{\mathrm H} \Lambda + \frac{500c\cdot (\log \gamma)^2}{(\log \theta)^2 \log n} , \]
 proving the upper bound. This shows in particular that $L(s(\theta)-\dim_{\mathrm H}\Lambda)\log n \to 0$ as $\theta \to 0^+$. 
 
 For the lower bound, we may decrease $\theta_0$ further to assume that for all $\theta<\theta_0$, $t_1(\dim_{\mathrm H} \Lambda) < t_1(s(\theta)) < t' - 3L(s(\theta)-\dim_{\mathrm H}\Lambda)\log n$ and $t_L(s(\theta)) > t'$ by Lemmas~\ref{lem:tlconvergence} and~\ref{lem:tprimelesststar}. Then by~\eqref{eq:taylor}, for large enough $L$, 
 \begin{align*}
  L &\geq \# \left\{ i : t' - 3L(s(\theta)-\dim_{\mathrm H}\Lambda)\log n \leq t_i(s(\theta)) \leq t' - L(s(\theta)-\dim_{\mathrm H}\Lambda)\log n \right\} \\
  &\geq \frac{2L(s(\theta)-\dim_{\mathrm H}\Lambda)\log n}{1.1((s(\theta)-\dim_{\mathrm H}\Lambda)\log n + c( 3L(s(\theta)-\dim_{\mathrm H}\Lambda)\log n)^2)}. 
  \end{align*}%
 Rearranging, for large enough $L$,  
 \[ s(\theta) \geq \dim_{\mathrm H}\Lambda + \frac{0.9}{L^2 \cdot 9.9 c \log n} \geq  \dim_{\mathrm H}\Lambda + \frac{(\log \gamma)^2}{20c \log n (\log \theta)^2 }. \qedhere \]
\end{proof} 

\begin{proof}[Proof of part~\eqref{itemv}]
One can differentiate~\eqref{eq:64} to obtain
\begin{equation*}
s''(\theta) = \frac{\gamma^L\cdot \big( f'_L(\theta)(1+g_L(\theta)A_L(\theta)) - f_L(\theta)(g'_L(\theta)A_L(\theta)+g_L(\theta)A'_L(\theta)) \big)}{(1+g_L(\theta)A_L(\theta))^2\log n}.
\end{equation*}
The sign of $s''(\theta)$ is determined by the sign of the term in parenthesis in the numerator. We know that there exists $c>0$ such that $f_L(\theta), g_L(\theta), A_L(\theta)\geq c$. There also exists $c'>0$ such that
\begin{equation*}
f'_L(\theta) = \underbrace{\big( I'\big(t_L(s(\theta))\big)-1 \big)}_{\leq -c'<0}\cdot \underbrace{\frac{\mathrm{d}}{\mathrm{d}\theta}t_L(s(\theta))}_{\geq c'>0 \text{ by~\eqref{eq:60}}} \leq -(c')^2<0
\end{equation*}
and $A'_L(\theta)>c'$ because all $I'\big(t_{\ell}(s(\theta))\big)$ and $I''\big(t_{\ell}(s(\theta))\big)$ are uniformly positive. Finally,
\begin{equation*}
g'_L(\theta) = \gamma^L \underbrace{\big( 1-I'\big(t_L(s(\theta))\big) \big)}_{\geq c'>0} + \gamma \underbrace{( 1-\gamma^{L-1}\theta)}_{\geq 0} \underbrace{I''\big( t_L(s(\theta)) \big)}_{>0}\cdot \underbrace{\frac{\mathrm{d}}{\mathrm{d}\theta}t_L(s(\theta))}_{\geq c'>0} \geq \gamma^L c'>0.
\end{equation*}
Hence, $s''(\theta)\leq C<0$ for some uniform constant $C>0$, implying that $s(\theta)$ is strictly concave on every interval $[\gamma^{-L},\gamma^{-(L-1)}]$.
\end{proof}
Note that $s''(\theta)\leq C<0$ holds for a constant $C$ that is independent of both $L$ and $\theta$. 

\subsection{Proof of Corollary \ref{cor:twoscales}}
We now use the fact that $\dim_{\, \theta} \Lambda$ is strictly increasing in $\theta$ to prove Corollary~\ref{cor:twoscales}. 

\begin{proof}[Proof of Corollary~\ref{cor:twoscales}]
Fix $\eta>0$ small enough that $\dim_{\mathrm H} \Lambda + \eta < \dim_{\mathrm B} \Lambda$. Since $\dim_{\, \theta} \Lambda$ exists and is continuous in $\theta$ including at $\theta=0$ (see Theorem~\ref{thm:main} and \cite[Proposition~4.1]{Falconer2020firstintermediate}), we can fix $\theta_1 < 1/\gamma$ small enough that $\dim_{\, \theta_1} \Lambda + \eta < \dim_{\mathrm B} \Lambda$. 
Let $s \in (\dim_{\mathrm H} \Lambda, \dim_{\, \theta_1} \Lambda + \eta]$, and for each $\delta \in (0,1)$, let $K = K(\delta) \in \mathbb{N}$ be such that $n^{-K} \leq \delta < n^{-(K-1)}$. 
For some string $\ii$, let $\mathcal{B}_{\lfloor K/\theta\rfloor}^{K,\ii}$ denote the set of level-$\lfloor K/\theta\rfloor$ approximate squares within the level-$K$ approximate square $B_{K}(\ii)$. For each $B_{K}(\ii)$ it is more cost efficient (in terms of $s$-cost, up to irrelevant multiplicative constants depending only on $\Lambda$) to subdivide it into level-$\lfloor K/\theta\rfloor$ approximate squares if and only if
\begin{equation}\label{eq:50}
n^{-Ks} \geq \#\mathcal{B}_{\lfloor K/\theta\rfloor}^{K,\ii} \cdot n^{-sK/\theta}.
\end{equation}

To determine $\#\mathcal{B}_{\lfloor K/\theta\rfloor}^{K,\ii}$ for some $\theta < 1/\gamma$, we compare the sequences that define $B_{K}(\ii)$ and a level-$\lfloor K/\theta\rfloor$ approximate square within it:
\[
\begin{array}{c|c|c|ccc}
	i_1  \;\dotsb\;  i_{K} & \ih_{K+1}  \;\dotsb\;  \ih_{\gamma(K)} &   &&& \\
	\underbrace{j_1  \;\dotsb\;  j_{K}}_{\text{equal}} & \underbrace{j_{K+1} \;\dotsb\; j_{\gamma(K)}}_{\text{same column}} & \underbrace{j_{\gamma(K)+1} \;\dotsb\; j_{\lfloor K/\theta \rfloor }}_{\in[N] \text{ freely}} & \underbrace{\jh_{\lfloor K/\theta \rfloor+1} \;\dotsb\; \jh_{\gamma(K/\theta)}}_{\in[M] \text{ freely}}.
\end{array}
\]
Thus, $\#\mathcal{B}_{\lfloor K/\theta\rfloor}^{K,\ii} = N^{\lfloor K/\theta \rfloor - \gamma(K)}\cdot M^{\gamma(K/\theta)- \lfloor K/\theta \rfloor }\cdot \prod_{\ell=K+1}^{\gamma(K)}N_{\ih_{\ell}}$. Substituting this back into~\eqref{eq:50}, we get after algebraic manipulations that it is more cost-efficient to subdivide if and only if 
\begin{equation*}
s\geq \frac{\theta}{(1-\theta)}\frac{\gamma-1}{\log n} \Bigg( \frac{1}{\gamma(K)-K} \sum_{\ell=K+1}^{\gamma(K)} \log N_{\ih_{\ell}} \Bigg)  
+ \frac{\theta}{1-\theta}( 1/\theta-\gamma ) \frac{\log N}{\log n} + \frac{\gamma-1}{1-\theta}\frac{\log M}{\log n}. 
\end{equation*}
But since $s \leq \dim_{\, \theta_1} \Lambda + \eta$, if is more cost-efficient to subdivide then the following condition for the average must hold: 
\begin{equation*}
	\frac{1}{\gamma(K)-K} \sum_{\ell=K+1}^{\gamma(K)} \log N_{\ih_{\ell}} \leq \log (N/M) - \Big(\frac{1}{\theta}-1\Big) \frac{\log n}{\gamma-1} ( \dim_{\mathrm B} \Lambda - \dim_{\, \theta_1} \Lambda - \eta).
\end{equation*}

As $\theta\to 0$, the right-hand side tends to $-\infty$, so there exists $\theta_0 < \theta_1 /2$ small enough that for all $\theta \leq 2\theta_0$ it is more cost efficient not to subdivide any of the level-$K$ approximate squares, if using only scales $\delta$ and $\delta^{1/\theta}$. 
Now, since $\dim_{\, \theta} \Lambda$ is strictly increasing in $\theta$ by Corollary~\ref{cor:allprop}, there exists $\epsilon < \eta$ small enough that $\dim_{\, \theta_0} \Lambda + \epsilon < \dim_{2\theta_0} \Lambda$. Then by the definition of $\theta_0$ and since $\dim_{\, \theta_1} \Lambda + \eta < \dim_{\mathrm B} \Lambda$, there exists $\delta_1 < 1$ such that if $\{U_i\}$ is a cover of $\Lambda$ using just two scales $\delta$ and $\delta'$ with $\delta' \leq \delta^{1/(2\theta_0)} < \delta \leq \delta_1$, then for all $\theta \leq 2\theta_0$, 
\[ \sum_i |U_i|^{\dim_{\, \theta} \Lambda + \epsilon} \geq \sum_i |U_i|^{\dim_{\, \theta_1} \Lambda + \eta} \geq 1.\]
Since  $\dim_{\, \theta_0} \Lambda + \epsilon < \dim_{2\theta_0} \Lambda$, there exists $\delta_0 < \delta_1$ such that if $\{U_i\}$ is a cover using just two scales $\delta$, $\delta'$ with $\delta^{1/(2\theta_0)} < \delta' \leq \delta \leq \delta_0$, then for all $\theta \leq \theta_0$, $\sum_i |U_i|^{\dim_{\, \theta} \Lambda + \epsilon} \geq \sum_i |U_i|^{\dim_{\, \theta_0} \Lambda + \epsilon} \geq 1$, completing the proof. 
\end{proof}

\subsection{Connection to multifractal analysis, proof of Theorem \ref{thm:multifractal}}\label{s:multifractalproof}

We use primes to denote the parameters of $\Lambda'$, and use notation from Section~\ref{subsec:multifractal}. 
In particular, in this section only, $I'$ will denote the rate function associated with $\Lambda'$, and not the derivative of $I$. 

\begin{lemma}\label{lem:multifractallemma}
Let $\Lambda$ and $\Lambda'$ be two Bedford--McMullen carpets with non-uniform fibres defined on grids of size $m \times n$ and $m' \times n'$ respectively. 
If either of~\eqref{item1} or~\eqref{itemnow2} from Theorem~\eqref{thm:multifractal} hold, then $\frac{\log n}{\log n'} = \frac{\log m}{\log m'} \in \mathbb{Q}$. 
\end{lemma}

\begin{proof}
First assume that~\eqref{item1} holds. 
Since the intermediate dimensions of $\Lambda$ and $\Lambda'$ have phase transitions at $\gamma^{-1}$ and $(\gamma')^{-1}$ respectively (see Corollary~\ref{cor:allprop}), we must have $\gamma=\gamma'$. To show that $\frac{\log n'}{\log n} \in \mathbb{Q}$, note that Theorem~\ref{thm:main} and the equality of the intermediate dimensions for $\theta \in (\gamma^{-1},1)$ tells us that for an open interval of $s$, 
\[ \frac{1}{\log n} I\left(\left(s-\frac{\log M}{\log m}\right)\log n\right) = \frac{1}{\log n'} I'\left(\left(s-\frac{\log M'}{\log m'}\right)\log n'\right). \]
Taking Legendre transforms of both sides and using scaling properties of Legendre transforms, using~\eqref{eq:22}, for all $\lambda$,  
\begin{equation}\label{eq:rational1}
 \frac{1}{\log n}\log \left( \frac{1}{M}\sum_{\ih=1}^{M_0}
 R_{\ih} N_{\ih}^{\lambda} \right) + \lambda \frac{\log M}{\log m}  =
  \frac{1}{\log n'}\log \left( \frac{1}{M'}\sum_{\jh=1}^{M_0'} R'_{\jh}  (N'_{\jh})^{\lambda} \right) + \lambda \frac{\log M'}{\log m'}.  
  \end{equation}
  Fix $K \in \mathbb{N}$ large enough that \[ N_1^{\frac{\log n'}{\log n}} (N_2/N_1)^K  <   N'_{M_0'} \cdot (n')^{\left( \frac{\log M'}{\log m'} - \frac{\log M}{\log m}\right)}.\] 
  Then exponentiating~\eqref{eq:rational1}, 
\begin{align*}
  &\frac{M^{\frac{\log n'}{\log n}}}{M'} \sum_{\jh=1}^{M_0'} R'_{\jh} \left( N'_{\jh} \cdot (n')^{\left( \frac{\log M'}{\log m'} - \frac{\log M}{\log m}\right)}\right)^{\lambda} =  \left( \sum_{\ih=1}^{M_0} R_{\ih} N_{\ih}^{\lambda} \right)^{\frac{\log n'}{\log n}}  \\
  &\phantom{--}= (R_1 N_1^\lambda)^{\frac{\log n'}{\log n}} \left( \sum_{k=0}^{K}  \frac{\frac{\log n'}{\log n} \left(\frac{\log n'}{\log n} - 1\right) \dotsb \left(\frac{\log n'}{\log n} - k + 1 \right) }{k!} \left( \sum_{\ih=2}^{M_0}  \frac{R_{\ih}}{R_1}\left(\frac{N_{\ih}}{N_1}\right)^{\lambda} \right)^{k}  \right)  \\*
  &\phantom{----}+  O\left(\left( N_1^{\frac{\log n'}{\log n}} \left(\frac{N_2}{N_1}\right)^{K+1}\right)^{\lambda}  \right)  
  \end{align*}
as $\lambda \to +\infty$, by the generalised binomial theorem. Therefore the coefficient of $N_1^{\frac{\log n'}{\log n}} (N_2/N_1)^{K \lambda}$, which is a polynomial equation in $\frac{\log n'}{\log n}$ with rational coefficients, must be 0. So $\frac{\log n'}{\log n}$ is algebraic. But $n^{\frac{\log n'}{\log n}} = n'$, so by the Gelfond--Schneider theorem, $\frac{\log n'}{\log n} \in \mathbb{Q}$. 

Now we assume only~\eqref{itemnow2}. We first show that $\frac{\log m'}{\log m} \in \mathbb{Q}$. Since the functions $\beta_{\nu}$ and $\beta_{\nu'}$ are equal, using~\eqref{eq:definebeta} with the change of variable $\lambda \coloneqq \gamma^{-1} + (1-\gamma^{-1}) \xi$, 
\[ \frac{ \frac{\lambda - \gamma^{-1}}{1-\gamma^{-1}} \log N - \log \sum_{\ih=1}^{M_0} R_{\ih} N_{\ih}^{\lambda}}{\log m} 
= \frac{ \frac{\lambda - \gamma^{-1}}{1-\gamma^{-1}} \log N' - \log \sum_{\jh=1}^{M'_0} 
{R_{\jh}N_{\jh}'}^{(\gamma')^{-1} + 
(1-(\gamma'))^{-1}
\frac{\lambda - \gamma^{-1}}{1-\gamma^{-1}} } }{\log m'}\] 
for all $\lambda$. A similar argument to the above using the generalised binomial theorem and Gelfond--Schneider theorem now gives that $\frac{\log m'}{\log m} \in \mathbb{Q}$. 
To show moreover that $\gamma = \gamma'$, note that these quantities are bi-Lipschitz invariants which depend only on the respective carpets, not the choice of iterated function system (see \cite[Theorem~3.3]{Fraser2018secondassouad}). So since $m$ and $m'$ are multiplicatively dependent, we can iterate the IFS to assume without loss of generality that $\Lambda$ is defined on an $m \times n$ grid, and $\Lambda'$ is defined on an $m \times n'$ grid, for the same $m$. Then by~\eqref{eq:definebeta}, 
\begin{equation}\label{eq:rational2} N^{-\xi} \sum_{\ih=1}^{M_0} R_{\ih} N_{\ih}^{\gamma^{-1} + (1-\gamma^{-1})\xi}  =  {N'}^{-\xi} \sum_{\jh=1}^{M'_0} R_{\jh}{N_{\jh}'}^{(\gamma')^{-1} + (1-(\gamma')^{-1})\xi} 
\end{equation}
for all $\xi$. So using a similar argument to the proof of \cite[Theorem~1.2]{RaoPreprintlipschitz}, equating exponential terms and coefficients gives $M_0=M'_0$. Also, $N_{\ih}^{\frac{\log m}{\log n}} = (N'_{\ih})^{\frac{\log m}{\log n'}}$, so 
\begin{equation}\label{eq:rational3} 
N_{\ih}' =  N_{\ih}^{\frac{\log n'}{\log n}} \qquad \mbox{for all} \quad \ih \in \{1,\dotsc,M_0\}. 
\end{equation}
Equating corresponding exponential bases in~\eqref{eq:rational2}, applying~\eqref{eq:rational3} and using that the carpet has non-uniform fibres (so not all $N_{\ih}$ are equal) shows that $n=n'$. %
In particular, $\frac{\log n'}{\log n} \in \mathbb{Q}$, as required. 
\end{proof}

\begin{proof}[Proof of Theorem~\ref{t:grid}]
\eqref{i:gridonecarpet}
The equality $\log n / \log m = \log n' / \log m'$ follows from observing the phase transitions of the Assouad spectrum (see~\cite[Theorem~3.3]{Fraser2018secondassouad}) or intermediate dimensions. The fact that $\log n / \log n' \in \mathbb{Q}$ follows from Lemma~\ref{lem:multifractallemma}. 

\eqref{i:gridtwocarpets}
The forward implication follows from~\eqref{i:gridonecarpet}, and the backward implication holds by iterating the IFSs (recalling the discussion after the statement of Theorem~\ref{t:grid}). 
\end{proof}

We now prove Theorem~\ref{thm:multifractal}. 
Since the intermediate dimension and multifractal spectra obviously do not depend on the grid on which the carpet is defined, in light of Lemma~\ref{lem:multifractallemma} and Theorem~\ref{t:grid} it suffices to assume henceforth that both carpets are defined on the same $m \times n$ grid to begin with. 
We already mentioned that the equivalence of~\eqref{itemnow2} and~\eqref{itemnow5} was proved in~\cite[Theorem~1.2]{RaoPreprintlipschitz}. To complete the proof, we show that $\eqref{item1}\Rightarrow\eqref{itemnow3}\Rightarrow\eqref{item4}\Rightarrow\eqref{itemnow5}\Rightarrow\eqref{itemlq}\Rightarrow\eqref{itemnow5}\Rightarrow\eqref{item4}\Rightarrow\eqref{item1}$. Of these, the implication $\eqref{item1}\Rightarrow\eqref{itemnow3}$ is obvious. 

\begin{proof}[Proof of $\eqref{itemnow3}\Rightarrow\eqref{item4}$]
Assume $\dim_{\, \theta}\Lambda = \dim_{\, \theta}\Lambda'$ on the open interval $(a,b)\subset [\gamma^{-1},1]$. After rearranging the formula in Theorem~\ref{thm:main} for $\dim_{\, \theta}\Lambda$ in the case $L=1$, we obtain that
\begin{equation*}
\dim_{\, \theta}\Lambda=\dim_{\mathrm B}\Lambda - \Big( \frac{1}{\theta}-1 \Big) \frac{I\big(t_1(\dim_{\, \theta}\Lambda)\big)}{\log n}.
\end{equation*}
By Corollary~\ref{cor:allprop} part~\eqref{itemi}, $\dim_{\, \theta}\Lambda$ and $\dim_{\, \theta}\Lambda'$ are real analytic on $(\gamma^{-1},1)$, hence $\dim_{\, \theta}\Lambda = \dim_{\, \theta}\Lambda'$ on the whole interval $[\gamma^{-1},1]$. In particular, $\dim_{\mathrm B}\Lambda=\dim_{\mathrm B}\Lambda'$, so 
\begin{align*}
I\big(t_1(\dim_{\, \theta}\Lambda)\big) &= I'\big(t'_1(\dim_{\, \theta}\Lambda')\big) \\
&= I'\Big(\Big(\dim_{\, \theta}\Lambda'-\frac{\log (M' M/M)}{\log m'}\Big)\log n'\Big) \\
&= I'\Big( \underbrace{\Big(\dim_{\, \theta}\Lambda-\frac{\log M}{\log m}\Big)\log n}_{=t_1(\dim_{\, \theta}\Lambda)} -\gamma \log(M'/M) \Big).
\end{align*}
Setting $t=t_1(\dim_{\, \theta}\Lambda)$, we see that 
\[ I(t)=I'(t-\gamma\log(M'/M)) \qquad \mbox{ for all } t \in (t_1(\dim_{\, a}\Lambda),t_1(\dim_{\, b}\Lambda)).\] Since the rate function is analytic,~\eqref{item4} follows. 
\end{proof}

\begin{proof}[Proof of $\eqref{item4}\Rightarrow\eqref{itemnow5}$]
Assume $I(t) = I'(t-\gamma \log(M'/M))$ on an open interval of $t$. Without loss of generality we assume that $M'\geq M$. Using definition~\eqref{eq:22} of the rate function,
\begin{equation*}
I'(t-\gamma \log(M'/M))=\sup_{\lambda\in \R} \bigg\{\lambda t - \log\bigg( \frac{(M'/M)^{\gamma \lambda}}{M'} \sum_{\ih=1}^{M'_0} R'_{\ih} (N'_{\ih})^{\lambda}\bigg)\bigg\}.
\end{equation*} 
Since $I$ and $I'$ are convex functions, their Legendre transforms must agree on an open interval, implying that
\begin{equation}\label{eq:65}
\frac{1}{M} \sum_{\ih=1}^{M_0} R_{\ih} N_{\ih}^{\lambda} = \frac{1}{M'} \sum_{\ih=1}^{M'_0} R'_{\ih} \big((M'/M)^{\gamma }N'_{\ih}\big)^{\lambda}
\end{equation}
on an open interval of $\lambda$. 

From here the proof follows the idea of the proof of \cite[Theorem~1.2]{RaoPreprintlipschitz}. Taking the $k$-th derivative of both sides of~\eqref{eq:65} with respect to $\lambda$ gives
\begin{equation}\label{eq:66}
\frac{1}{M} \sum_{\ih=1}^{M_0} R_{\ih} N_{\ih}^{\lambda}\cdot (\log N_{\ih})^k = \frac{1}{M'} \sum_{\ih=1}^{M'_0} R'_{\ih} \big((M'/M)^{\gamma }N'_{\ih}\big)^{\lambda} \cdot \big(\log( (M'/M)^{\gamma }N'_{\ih})\big)^k.
\end{equation}
Recall that the $N_{\ih}$ and $N'_{\ih}$ are ordered in decreasing order. Since~\eqref{eq:66} holds for all $k$, the largest term on either side must be equal, so $N_{1}/N'_{1} = (M'/M)^{\gamma}$, and also its coefficient
\begin{equation*}
\frac{R_{1} N_{1}^{\lambda}}{M} = \frac{R'_{1} \big((M'/M)^{\gamma }N'_{1}\big)^{\lambda} }{M'} \;\Longleftrightarrow\; \frac{R'_{1}}{R_{1}} = \frac{M'}{M}\cdot \left(\frac{M}{M'}\right)^{\gamma\lambda} \cdot \left(\frac{N_1}{N'_1}\right)^{\lambda} =  \frac{M'}{M}.
\end{equation*}
After subtracting these terms from both sides of~\eqref{eq:66}, we repeat the argument for the next largest term and so on. If $M_0\neq M'_0$ then after $\min\{M_0,M'_0\}$ steps one side would be 0 and the other non-zero, a contradiction. Hence, we conclude that $M_0= M'_0$ and $N_{\ih}/N'_{\ih}=(R'_{\ih}/R_{\ih})^{\gamma}= (M'/M)^{\gamma}$ for all $\ih=1,\dotsc,M_0$.
\end{proof}

\begin{proof}[Proof of $\eqref{itemnow5}\Rightarrow\eqref{itemlq}$]
If~\eqref{itemnow5} holds, then substituting $R_{\ih}' = R_{\ih}M'/M$ and $N_{\ih}' = N_{\ih}(M'/M)^{-\gamma}$ gives $N' = N(M'/M)^{1-\gamma}$. Substituting into~\eqref{e:lqformula} gives $T_{\nu}(q) = T_{\nu'}(q)$. 
\end{proof}

\begin{proof}[Proof of $\eqref{itemlq}\Rightarrow\eqref{itemnow5}$]
Equating the constant term and coefficient of $q$ from~\eqref{e:lqformula} gives that 
\[ \sum_{\ih=1}^{M_0'} R_{\ih}' (N_{\ih}')^q = \left(\frac{N'}{N}\right)^{\frac{\log n - q\log n}{\log(n/m)}} \sum_{\ih=1}^{M_0} R_{\ih} (N_{\ih})^q. \]
By the same differentiation argument from~\cite{RaoPreprintlipschitz} that was used in the proof of~$\eqref{item4}\Rightarrow\eqref{itemnow5}$, $M_0 = M_0'$ and $N_{\ih}/N'_{\ih}=(R'_{\ih}/R_{\ih})^{\gamma}= (N'/N)^{\frac{\log n}{\log (n/m)}}$ for all $\ih$. 
Now observing that 
\[ \frac{M'}{M} = \frac{\sum_{\ih=1}^{M_0} R_{\ih}'}{\sum_{\jh=1}^{M_0} R_{\jh}} = \left(\frac{N'}{N}\right)^{\frac{\log n}{\log (n/m)}}\]
 shows that~\eqref{itemnow5} holds. 
\end{proof}

\begin{proof}[Proof of $\eqref{itemnow5}\Rightarrow\eqref{item4}$]
Assume that $M_0= M'_0$ and $N_{\ih}/N'_{\ih}=(R'_{\ih}/R_{\ih})^{\gamma}= (M'/M)^{\gamma}$ for all $\ih=1,\dotsc,M_0$. Then~\eqref{eq:65} holds for every $\lambda\in\mathbb{R}$. Since both sides of~\eqref{eq:66} are strictly positive for $k=2$, both sides of~\eqref{eq:65} are convex functions of $\lambda$. Hence, their Legendre transforms are also equal: 
\begin{equation*}
\sup_{\lambda\in \R} \bigg\{\lambda t - \log\bigg( \frac{1}{M} \sum_{\ih=1}^{M_0} R_{\ih} N_{\ih}^{\lambda}\bigg)\bigg\} = \sup_{\lambda\in \R} \bigg\{\lambda t - \log\bigg( \frac{(M'/M)^{\gamma \lambda}}{M'} \sum_{\ih=1}^{M'_0} R'_{\ih} (N'_{\ih})^{\lambda}\bigg)\bigg\}
\end{equation*}
for all $t$, which is precisely $I(t) = I'(t-\gamma \log(M'/M))$.
\end{proof}

\begin{proof}[Proof of $\eqref{item4}\Rightarrow\eqref{item1}$]
Assume $I(t) = I'(t-\gamma \log(M'/M))$ for every $t\in\mathbb{R}$. We claim that for every $s\in(\dim_{\mathrm H}\Lambda,\dim_{\mathrm B}\Lambda)$,
\begin{equation}\label{eq:67}
t'_{\ell}(s) = t_{\ell}(s) - \gamma \log(M'/M) \quad\text{ for every } \ell\in\N.
\end{equation}
The proof goes by induction on $\ell$. For $\ell=1$, $t_1(s)= \big( s-\log(M'M/M)\log m \big)\log n = t_1(s)-\gamma\log(M'/M)$. Assuming~\eqref{eq:67} for $\ell-1$, we prove for $\ell$:
\begin{align*}
t'_{\ell}(s) &= T'_s(t'_{\ell-1}(s)) \\*
&= \left( s-\frac{\log M'}{\log m} \right)\log n +\gamma I'\left(t_{\ell-1}(s) - \gamma \log\left(\frac{M'}{M}\right)\right) \\
&\stackrel{\eqref{item4}}{=} \left( s-\frac{\log(M'M/M)}{\log m} \right)\log n +\gamma I\big(t_{\ell-1}(s)\big) \\
&= T_s(t_{\ell-1}(s)) - \gamma\log\left(\frac{M'}{M}\right), 
\end{align*} 
which completes the proof of~\eqref{eq:67} since $T_s(t_{\ell-1}(s))=t_{\ell}(s)$. 

From~\eqref{eq:67} and assumption~\eqref{item4} it immediately follows that 
\begin{equation}\label{eq:68}
I(t_{\ell}(s))=I'(t'_{\ell}(s)). 
\end{equation}
Assumption~\eqref{item4} also implies~\eqref{itemnow5}, thus we know that
\begin{equation}\label{eq:69}
N'=\sum_{\ih=1}^{M'_0} R'_{\ih} N'_{\ih} \stackrel{\eqref{itemnow5}}{=} \sum_{\ih=1}^{M_0} \frac{M'R_{\ih}}{M} \cdot N_{\ih} \left( \frac{M}{M'}\right)^{\gamma} = N\left( \frac{M}{M'}\right)^{\gamma-1}. 
\end{equation}
Writing $s'_{\theta}=\dim_{\, \theta}\Lambda'$, using~\eqref{eq:67}, \eqref{eq:68},~\eqref{eq:69} and algebraic manipulations, we obtain
\begin{align*}
0&= -s'_{\theta} \log n + \gamma^L\theta \log N' - (\gamma^L\theta - 1) t'_L(s'_{\theta}) + \gamma(1-\gamma^{L-1}\theta)(\log M' - I'(t'_L(s'_{\theta}))) \\
&= -s'_{\theta} \log n + \gamma^L\theta \log N - (\gamma^L\theta - 1) t_L(s'_{\theta}) + \gamma(1-\gamma^{L-1}\theta)(\log M - I(t_L(s'_{\theta}))).
\end{align*} 
By Theorem~\ref{thm:main}, $\dim_{\, \theta}\Lambda$ is the unique solution to this equation, so $\dim_{\, \theta}\Lambda=\dim_{\, \theta}\Lambda'$. This completes the proof of Theorem~\ref{thm:multifractal}. 
\end{proof}

\subsection*{Acknowledgements}
This work was completed while both authors were based at the University of St Andrews, and supported by a \emph{Leverhulme Trust Research Project Grant} (RPG-2019-034). 
The authors thank Kenneth J. Falconer, Jonathan M. Fraser and Alex Rutar for useful discussions. 
Chapter~5 of AB's PhD thesis~\cite{Banaji2023thesis} is based on this paper, and he thanks his examiners Lars Olsen and P\'eter Varj\'u for helpful comments. 
\vspace{0.2cm}

\section*{References}

\printbibliography[heading=none]%

\end{document}